\documentclass[11pt,reqno]{amsart}

\usepackage{amsmath, amsthm, amssymb}
\usepackage{amsfonts}

\usepackage{pdfsync}
\usepackage{esint}
\usepackage[ansinew]{inputenc}
\usepackage[dvips]{epsfig}
\usepackage{graphicx}
\usepackage[english]{babel}
\usepackage{hyperref}
\usepackage{color}
\theoremstyle{plain}
\newtheorem{thm}{Theorem}[section]
\newtheorem{cor}[thm]{Corollary}
\newtheorem{lem}[thm]{Lemma}
\newtheorem{prop}[thm]{Proposition}

\theoremstyle{definition}
\newtheorem{defi}[thm]{Definition}

\theoremstyle{remark}
\newtheorem{rem}[thm]{Remark}

\numberwithin{equation}{section}

\newcommand{\average}{{\mathchoice {\kern1ex\vcenter{\hrule height.4pt
width 6pt depth0pt} \kern-9.7pt} {\kern1ex\vcenter{\hrule
height.4pt width 4.3pt depth0pt} \kern-7pt} {} {} }}

\def\R{\mathbb{R}}
\def\RR{\mathbb{R}}

\def\NN{\mathbb{N}}

\newcommand{\ZZ}{\mathbb{Z}}

\DeclareMathOperator*{\einf}{ess\,inf}

\def\supp{\mathrm{supp}}
\def\osc{\mathrm{osc}}

\usepackage[colorinlistoftodos]{todonotes} 
\begin{document}

\title[The Neumann problem for the fractional Laplacian: boundary regularity]{The Neumann problem for the fractional Laplacian: regularity up to the boundary}

\author{Alessandro Audrito}
\address{Universit\"at Z\"urich, Institut f\"ur Mathematik, Winterthurerstrasse 190, 8057 Z\"urich, Switzerland}
\email{alessandro.audrito@math.uzh.ch}

\author{Juan-Carlos Felipe-Navarro}
\address{Universitat Polit\`ecnica de Catalunya and BGSMath, Departament de Matem\`atiques, Avda. Diagonal 647, 08028 Barcelona, Spain}
\email{juan.carlos.felipe@upc.edu}

\author{Xavier Ros-Oton}
\address{Universit\"at Z\"urich,
Institut f\"ur Mathematik,
Winterthurerstrasse 190, 8057 Z\"urich, Switzerland \&
ICREA, Pg.\ Llu\'is Companys 23, 08010 Barcelona, Spain \&
Universitat de Barcelona, Departament de Matem\`atiques i Inform\`atica, Gran Via de les Corts Catalanes 585, 08007 Barcelona, Spain.}
\email{xavier.ros-oton@math.uzh.ch}




\begin{abstract}
We study the regularity up to the boundary of solutions to the Neumann problem for the fractional Laplacian.
We prove that if $u$ is a weak solution of $(-\Delta)^s u=f$ in $\Omega$, \ $\mathcal N_s u=0$ in $\Omega^c$, then $u$ is $C^\alpha$ up tp the boundary for some $\alpha>0$.
Moreover, in case $s>\frac12$, we then show that $u\in C^{2s-1+\alpha}(\overline\Omega)$.
To prove these results we need, among other things, a delicate Moser iteration on the boundary with some logarithmic corrections.

Our methods allow us to treat as well the Neumann problem for the regional fractional Laplacian, and we establish the same boundary regularity result.

Prior to our results, the interior regularity for these Neumann problems was well understood, but near the boundary even the continuity of solutions was open.
\end{abstract}

\maketitle


\section{Introduction and main results}

We study the regularity of solutions to the Neumann problem
\begin{equation}\label{eq-intro}
\left\{ \begin{array}{rcll}
(-\Delta)^su&=&f&\textrm{in }\Omega \\
\mathcal N_su&=&0&\textrm{in }\Omega^c,
\end{array}\right.
\end{equation}
where $\mathcal N_s$ is a ``nonlocal normal derivative'', given by
\begin{equation}\label{normal-s}
\mathcal N_su(x):=c_{N,s}\int_\Omega \frac{u(x)-u(y)}{|x-y|^{N+2s}}\,dy, \qquad x\in \Omega^c.
\end{equation}
The constant $c_{N,s}$ is the one appearing in the definition the fractional Laplacian
\begin{equation}\label{operator}
(-\Delta)^s u(x)=c_{N,s}\,\textrm{PV}\int_{\R^N}\frac{u(x)-u(y)}{|x-y|^{N+2s}}\,dy.
\end{equation}

The Neumann problem \eqref{eq-intro} was first introduced in \cite{DRV,DGLZ}, and has been subsequently studied in several papers; see for example \cite{Ab,AT,CC,LMPPS,Vo}.
As explained in detail in \cite{DRV}, \eqref{eq-intro} is a natural Neumann problem for the fractional Laplacian, for several reasons:

\vspace{2mm}

\noindent \ $\bullet$ The problem has a variational structure, and weak solutions are obtained by minimizing the energy functional
\begin{equation}\label{energy}
 \mathcal E(u) :=\frac{c_{N,s}}{4}\int\int_{\R^{2N}\setminus(\Omega^c)^2}\frac{|u(x)-u(y)|^2}{|x-y|^{N+2s}}\,dx\,dy-\int_\Omega f\,u.
 \end{equation}
Solutions exist if and only if $\int_\Omega f=0$.

\vspace{2mm}

\noindent \ $\bullet$ The following integration by parts formulas hold for $C^2$ functions $u,v$:
\[\int_\Omega (-\Delta)^su\,dx=-\int_{\Omega^c} \mathcal N_su\,dx\]
and
\begin{equation}\label{int-by-parts}
 \frac{c_{n,s}}{2}\int\int_{\R^{2N}\setminus(\Omega^c)^2}\frac{\bigl(u(x)-u(y)\bigr)\bigl(v(x)-v(y)\bigr)}{|x-y|^{N+2s}}\,dx\,dy
=\int_\Omega v\,(-\Delta)^su+\int_{\Omega^c} v\,\mathcal N_su.
\end{equation}

\vspace{2mm}

\noindent \ $\bullet$ The corresponding heat equation with homogeneous Neumann conditions
possesses natural properties like conservation of mass inside $\Omega$ or convergence to a constant as $t\rightarrow \infty$.

\vspace{2mm}

\noindent \ $\bullet$ The problem has a natural probabilistic interpretation, heuristically described in~\cite{DRV}, and rigorously studied  in \cite{Vo}.

\vspace{2mm}

\noindent \ $\bullet$ As $s\uparrow1$, we recover the classical Neumann problem for the Laplacian in $\Omega$.

\vspace{2mm}

\noindent \ $\bullet$ The energy functional \eqref{energy} is the same that yields solutions to the Dirichlet problem for the fractional Laplacian; see \cite{RS-Dir,R-survey}.

\vspace{3mm}

The aim of this paper is to study the boundary regularity of solutions to \eqref{eq-intro}.

\subsection{Main results}

While the Dirichlet problem is very well understood \cite{AR,AuR,BCI2,BKK08,CKS,FKV,Grubb,Grubb2,Landkov,R-survey,RS-Dir}, much less is known for the Neumann case.
Our main result reads as follows:

\begin{thm}\label{thm-intro}
Let $\Omega\subset \R^N$ be any bounded Lipschitz domain.
Let $s\in(0,1)$, and $u$ be any weak solution of \eqref{eq-intro} with $f\in L^q(\Omega)$, with $q>\frac{N}{2s}$ and $\int_\Omega f =0$.

Then,
\[\|u\|_{C^\alpha(\overline\Omega)} \leq C\left( \|f\|_{L^q(\Omega)} + \|u\|_{L^2(\Omega)}\right),\]
for some $\alpha>0$.
Moreover, if $s>\frac12$, $q>N$, and $\Omega$ is $C^1$, we then have
\[\|u\|_{C^{2s-1+\alpha}(\overline\Omega)} \leq C\left( \|f\|_{L^q(\Omega)} + \|u\|_{L^2(\Omega)}\right).\]
The constants $C$ and $\alpha$ depend only on $N$, $s$, $q$, and $\Omega$.
\end{thm}

This is the first boundary regularity result for the Neumann problem \eqref{eq-intro}, and even the continuity of solutions is new.

As in case of the Dirichlet problem \cite{RS-Dir}, it turns out that the boundary regularity is much more delicate than the interior one, and does not follow easily by adapting the classical methods used for $s=1$ \cite{N,MS}.
This is because in this nonlocal context one cannot use any even/odd reflection to study solutions near the boundary, and a completely different strategy is needed.

In \cite{RS-Dir}, a key idea was to use the methods coming from equations with bounded measurable coefficients in non-divergence form.
Here, instead, we will need to use methods coming from equations with bounded measurable coefficients in divergence form.
More precisely, we will need (among other things) a delicate Moser iteration on the boundary involving some logarithmic corrections on $\partial\Omega$.
This will be explained in more detail later on in the paper.

In a sense, Theorem \ref{thm-intro} can be seen as the Neumann version of the boundary regularity theory for the Dirichlet problem developed in \cite{RS-Dir}.

\begin{rem}
It is important to remark that $2s-1$ is a natural critical exponent in this problem.
This can be seen easily when $\Omega=\{x_N>0\}$, in which the function $|x_N|^{2s-1}$ solves \eqref{eq-intro} pointwise, even though it is not a weak solution --- nor it satisfies \eqref{int-by-parts}.
Thus,  $C^{2s-1+\alpha}(\overline\Omega)$  is the minimum regularity needed in order to discard this kind of solutions.
This will become even more clear in case of the regional fractional Laplacian, explained below.
\end{rem}

\subsection{Regional fractional Laplacian}

The methods developed in this paper allow us to treat as well the Neumann problem for the \emph{regional} fractional Laplacian.
This corresponds to a censored stochastic process; see \cite{BBC}.

Solutions to this problem are obtained by minimizing the energy
\begin{equation}\label{energy-reg}
 \mathcal E(u) :=\frac{c_{N,s}}{4}\int_\Omega\int_{\Omega}\frac{|u(x)-u(y)|^2}{|x-y|^{N+2s}}\,dx\,dy-\int_\Omega f\,u,
 \end{equation}
and the operator is given by
\begin{equation}\label{operator-reg}
(-\Delta)_\Omega^s u(x)=c_{N,s}\,\textrm{PV}\int_{\Omega}\frac{u(x)-u(y)}{|x-y|^{N+2s}}\,dy.
\end{equation}
This problem shares many of the properties of \eqref{eq-intro} described above: it has a variational formulation, a nice probabilistic interpretation, convergence as $s\uparrow1$ to the Neumann problem for the Laplacian, and conservation of mass for its parabolic version.
The main difference is that the operator given by \eqref{operator-reg} depends on $\Omega$, and that in this case $\R^N\setminus\overline{\Omega}$ plays no role.

The Dirichlet problem in this setting is obtained by considering \eqref{energy-reg} among all functions $u=0$ on $\partial\Omega$.
Notice that, by trace theorems for $H^s(\Omega)$ spaces \cite{Di}, this makes sense only when $s>\frac12$.
It turns out then that solutions to the Dirichlet problem are $C^{2s-1}(\overline \Omega)$, and if $f>0$ they actually satisfy
 \[u\asymp d^{2s-1}\quad \textrm{in}\quad \Omega;\]
see \cite{BBC,Ch,CK,GM}.

However, as in case of the fractional Laplacian \eqref{eq-intro}, the Neumann case is much less understood, and it is not even clear what is the right pointwise Neumann condition for solutions in this case.

An integration by parts formula found in \cite{Gu} suggests that the right quantity in this context is given by\footnote{Notice that when $u=0$ on $\partial\Omega$ (Dirichlet case), then this quantity is the same as $u/d^{2s-1}|_{\partial\Omega}$.}
\[\qquad \qquad\qquad \partial_\nu^{2s-1} u(z) := \lim_{t\downarrow0}\frac{u(z+t\nu)-u(z)}{t^{2s-1}},\qquad\qquad z\in\partial\Omega,\]
where $\nu$ is the (inward) unit normal to $\partial\Omega$.
More precisely, it is proved in \cite{Gu} that, if $u,v\in d^{2s-1}C^2(\overline\Omega)+C^2(\overline\Omega)$ then\footnote{A function $w$ belongs to $d^{2s-1}C^2(\overline\Omega)+C^2(\overline\Omega)$ if it can be written as $w= d^{2s-1} g +h$, with $g,h\in C^2(\overline\Omega)$.}
\begin{equation}\label{int-by-parts-reg}
 \frac{c_{N,s}}{2}\int_\Omega\int_\Omega \frac{\bigl(u(x)-u(y)\bigr)\bigl(v(x)-v(y)\bigr)}{|x-y|^{N+2s}}\,dx\,dy
=\int_\Omega v\,(-\Delta)^s_\Omega u+\kappa_{N,s}\int_{\partial\Omega} v\,\partial_\nu^{2s-1} u.
\end{equation}
This is the analogue of \eqref{int-by-parts} in this context, and suggests that the pointwise Neumann condition in this setting should be
\begin{equation}\label{Neumann-reg}
\partial_\nu^{2s-1} u = 0 \quad \textrm{on}\quad \partial\Omega.
\end{equation}

Our main result in this context answers positively this question, and reads as follows.

\begin{thm}\label{thm-intro-reg}
Let $\Omega\subset \R^N$ be any bounded Lipschitz domain.
Let $s\in(0,1)$, $f\in L^q(\Omega)$, with $q>\frac{N}{2s}$, be such that $\int_\Omega f=0$, and $u$ be any free minimizer of \eqref{energy-reg}.

Then,
\[\|u\|_{C^\alpha(\overline\Omega)} \leq C\left( \|f\|_{L^q(\Omega)} + \|u\|_{L^2(\Omega)}\right),\]
for some $\alpha>0$.
Moreover, if $s>\frac12$, $q>N$, and $\Omega$ is $C^1$, we then have
\[\|u\|_{C^{2s-1+\alpha}(\overline\Omega)} \leq C\left( \|f\|_{L^q(\Omega)} + \|u\|_{L^2(\Omega)}\right).\]
In particular, for every $s\in(0,1)$ we have \eqref{Neumann-reg}.
The constants $C$ and $\alpha$ depend only on $N$, $s$, $q$, and $\Omega$.
\end{thm}

In particular, thanks to Theorem \ref{thm-intro-reg}, we find that the Neumann problem for the regional fractional Laplacian is actually
\begin{equation}\label{eq-intro-reg}
\left\{ \begin{array}{rcll}
(-\Delta)^s_\Omega u&=&f&\textrm{in }\Omega \\
\partial_\nu^{2s-1} u&=&0&\textrm{on }\partial\Omega.
\end{array}\right.
\end{equation}
Notice that our result also implies that solutions to the Neumann problem are more regular than those corresponding to the Dirichlet case, as expected.

\begin{rem}
Other Neumann problems for the fractional Laplacian $(-\Delta)^s$ have been introduced in \cite{BCGJ,BGJ} and \cite{Grubb2}.
These different Neumann problems recover the classical Neumann problem as a limit case, and the one in \cite{BCGJ,BGJ} has a probabilistic interpretation as well.
We refer to \cite{DRV} for a comparison between these different models, and related problems for the other operators.
\end{rem}

\subsection{Acknowledgements}

XR was supported by the European Research Council (ERC) under the Grant Agreement No 801867.
AA and XR were supported by the Swiss National Science Foundation (SNF).
JF and XR were supported by MINECO grant MTM2017-84214-C2-1-P (Spain).
JF acknowledges financial support from the Spanish Ministry of Economy and Competitiveness (MINECO), through the Mar\'ia de Maeztu Programme for Units of Excellence in R\&D (MDM-2014-0445-16-4).
Moreover, he is a member of the Barcelona Graduate School of Mathematics (BGSMath) and part of the Catalan research group 2017 SGR 01392.
Part of this work has been done while JF was visiting Universit\"at Z\"urich.

\subsection{Organization of the paper}

In Section \ref{sec2} we transform the Neumann problem \eqref{eq-intro} into a regional-type operator inside $\Omega$.
In Section \ref{sec3} we prove an $L^\infty$ bound for solutions of \eqref{eq-intro} and \eqref{eq-intro-reg}.
Then, in Section \ref{sec4} we develop a Moser iteration (with logarithmic corrections), and deduce that solutions are $C^\alpha$ for some $\alpha>0$.
In Section \ref{sec5} we establish a Neumann Liouville-type theorem in a half-space, and finally in Section \ref{sec6} we use it to prove higher regularity of solutions.

\section{An equivalent problem in $\Omega$}
\label{sec2}

As first noticed in \cite{Ab}, problem \eqref{eq-intro} can be reformulated as a regional-type problem in~$\Omega$ for a new operator
\begin{equation}\label{new-operator}
L_\Omega u(x):= \textrm{PV}\int_{\Omega}\big(u(x)-u(y)\big)K_\Omega(x,y)\,dy,
\end{equation}
with
\begin{equation}\label{new-kernel}
K_\Omega(x,y) = \frac{c_{N,s}}{|x-y|^{N+2s}} + k_\Omega(x,y),
\end{equation}
\begin{equation}\label{new-kernel2}
\qquad\qquad \qquad k_\Omega(x,y)=c_{N,s} \int_{\Omega^c} \frac{dz}{|x-z|^{N+2s} |y-z|^{N+2s} \int_\Omega \frac{dw}{|z-w|^{N+2s}} },\qquad x,y\in\Omega.
\end{equation}
Moreover, it was proved in \cite{Ab} that, for every fixed $x\in\Omega$, the kernel $k_\Omega(x,y)$ has a logarithmic singularity along $\partial\Omega$.
Here we need more precise estimates, with constants that are independent of $x,y\in\Omega$.

\subsection{Fine estimates on the new kernel}

Here, and throughout the paper, we denote $A\asymp B$ whenever $C^{-1}A\leq B\leq CA$ for some positive constant $C$.

\begin{prop}\label{kernel-estimates}
Let $\Omega\subset \R^N$ be any Lipschitz domain, let $d$ be the distance function to the boundary, and denote
\[\qquad \qquad d_{x,y} := \min\{d(x),\, d(y)\},\qquad x,y\in \Omega.\]
Then, the kernel $k_\Omega$ satisfies
\begin{equation}\label{new-kernel2-estimates}
k_\Omega(x,y) \asymp \left\{
\begin{array}{ll}
\displaystyle \frac{1+\left|\log\left(\frac{d_{x,y}}{|x-y|}\right)\right|}{|x-y|^{N+2s}} & \quad\textrm{if} \quad d_{x,y} \leq |x-y|  \vspace{2mm}\\
d_{x,y}^{-N-2s} & \quad\textrm{if} \quad d_{x,y}\geq |x-y|
\end{array}
\right.
\end{equation}
In particular, the kernel $K_\Omega$ satisfies
\begin{equation}\label{new-kernel-estimates}
\qquad\qquad\qquad K_\Omega(x,y) \asymp  \frac{1+\log^-\left(\frac{d_{x,y}}{|x-y|}\right)}{|x-y|^{N+2s}} \qquad \textrm{for all}\quad x,y\in\Omega,
\end{equation}
where $\log^- t:=\max\{0,\, -\log t\}$.

The constants  in \eqref{new-kernel2-estimates} and \eqref{new-kernel-estimates} depend only on $\Omega$.
Moreover, if $\Omega\cap B_2$ can be written as a Lipschitz graph, then \eqref{new-kernel2-estimates} and \eqref{new-kernel-estimates} hold for $x,y \in \Omega\cap B_1$ with constants depending only on $N$ and the Lipschitz norm of such graph.
\end{prop}

\begin{proof}
Since \eqref{new-kernel-estimates} follows immediately from \eqref{new-kernel2-estimates}, it suffices to prove \eqref{new-kernel2-estimates}.
Moreover, since any Lipschitz domain can be locally written as a Lipschitz graph, we will assume that $\Omega\cap B_2$ is a Lipschitz graph, and prove the estimate for $x,y\in \Omega\cap B_1$.

By \cite[Lemma 2.1]{Ab} we have that
\[\int_\Omega \frac{dw}{|z-w|^{N+2s}} \asymp \min\big\{d^{-2s}(z),\, d^{-N-2s}(z)\big\}\]
for $z\in \Omega^c$, so we deduce that
\[k_\Omega(x,y) \asymp \int_{\Omega^c} \frac{d^{2s}(z)\, dz}{|x-z|^{N+2s} |y-z|^{N+2s}\min\big\{1,\,d^{-N}(z)\big\} },\qquad x,y\in\Omega\cap B_1.\]

On the other hand, notice that the kernel is scale invariant, in the sense that
\[k_\Omega(rx,ry) = r^{-N-2s} k_{r^{-1}\Omega}(x,y),\]
and it is symmetric in $x,y$.
Moreover, the estimate we want to prove is also scale invariant and symmetric.
Therefore, to prove the desired estimate, we may assume that
\[d(y)\leq d(x)\qquad \textrm{and}\qquad \max\{d(x),\, |x-y|\}=1.\]

Moreover, since for $x,y\in \Omega\cap B_1$ the contributions from $\Omega^c\cap B_2^c$ in \eqref{new-kernel2} are bounded, we have
\begin{equation}\label{boundss}
k_\Omega(x,y) \asymp \int_{\Omega^c \cap B_2} \frac{d^{2s}(z)\, dz}{|x-z|^{N+2s} |y-z|^{N+2s} },\qquad x,y\in\Omega\cap B_1.
\end{equation}

Now, notice that since such integral is obviously bounded when $d(x)\geq d(y)\geq \frac12$, since $z\in \Omega^c$ and therefore the integrand is bounded.
Further, notice that if $|x-y|\geq\frac12$ then the singularities are well separated, and therefore we can split the integral into two pieces.

 Because of this, we split the proof into different cases.
First, assume that $|x-y|\leq d(y)\leq d(x)=\frac12$.
Then, by triangle inequality we have $d(y)+|x-y|\geq d(x)$, and therefore $d(y)\geq \frac12$, which yields that the integrand in \eqref{boundss} is bounded.
Hence, in this case, $k_\Omega\asymp 1$.

For the second case, assume that $d(y)\leq |x-y|\leq d(x)=1$.
By triangle inequality, we have $|x-y|\geq\frac12$ in this case.
The factor $|x-z|^{-n-2s}$ is bounded, and hence we have
\[
k_\Omega(x,y) \asymp \int_{\Omega^c \cap B_2} \frac{d^{2s}(z)\, dz}{|y-z|^{N+2s} }.
\]
Then, by doing a bi-Lipschitz transformation, it suffices to consider the case in which $\Omega\cap B_2$ is flat, i.e., $\Omega\cap B_2 = \{x_N>0\}\cap B_2$.
(Notice that the estimates are invariant under a biLipschitz transformation, since all distances stay comparable.)
Then, we get
\[
k_\Omega(x,y) \asymp \int_{\{z_N<0\} \cap B_2} \frac{|z_N|^{2s}\, dz}{|y-z|^{N+2s} }
 \asymp 1+\big|\log d(y)\big|.
\]
The last estimate can be proved as follows: denote $d(y)=y_N=:\delta>0$, so that by a change of variables $z\mapsto \delta z$ we have
\[
\int_{\{z_N<0\} \cap B_2} \frac{|z_N|^{2s}\, dz}{|y-z|^{N+2s} } \asymp
 \int_{\{z_N<0\} \cap B_{1/\delta}} \frac{|z_N|^{2s}}{1+|z|^{N+2s} }\, dz \asymp 1+\big|\log \delta\big|,\]
as claimed.

Finally, for the third case, assume that $d(y)\leq d(x) \leq |x-y|=1$.
Then, by the same argument we have
\[\begin{split}
k_\Omega(x,y) &\asymp \int_{\Omega^c \cap B_{1/2}(x)} \frac{d^{2s}(z)\, dz}{|x-z|^{N+2s} } + \int_{\Omega^c \cap B_{1/2}(y)} \frac{d^{2s}(z)\, dz}{|y-z|^{N+2s} } + C\\
& \asymp 1+\big|\log d(y)\big|,\end{split}
\]
where we used that $d(y)\leq d(x)$.
Thus, the result is proved.
\end{proof}

Thanks to these estimates, we will treat problem \eqref{eq-intro} as a problem inside $\Omega$ for an operator \eqref{new-operator} with kernel satisfying \eqref{new-kernel-estimates}.
This will allow us to treat at the same time both problems \eqref{eq-intro} and \eqref{eq-intro-reg}.

More precisely, throughout the next two sections we assume that $L_\Omega$ is an operator of the form \eqref{new-operator}, with kernel $K_\Omega$ satisfying either
\begin{equation}\label{K1}
\qquad\qquad\qquad K_\Omega(x,y) \asymp  \frac{1+\log^-\left(\frac{d_{x,y}}{|x-y|}\right)}{|x-y|^{N+2s}} \qquad \textrm{for}\quad x,y\in\Omega,
\end{equation}
or
\begin{equation}\label{K2}
\qquad\qquad\qquad K_\Omega(x,y) \asymp  \frac{1}{|x-y|^{N+2s}} \qquad \textrm{for}\quad x,y\in\Omega.
\end{equation}
The first case covers the Neumann problem for the fractional Laplacian, while the second case covers the regional fractional Laplacian.
The constants in \eqref{K1} and \eqref{K2} are given by Proposition \ref{kernel-estimates}.

The corresponding bilinear form is given by
\begin{equation}\label{B}
 B(u,v) :=\int_\Omega\int_\Omega \big(u(x)-u(y)\big)\big(v(x)-v(y)\big)K_\Omega(x,y)\,dx,
 \end{equation}
and the definition of weak solution to the Neumann problem is the following.

\subsection{Weak solutions}

Here, and throughout the paper, we denote with $H_K(\Omega)$ the space of functions for which
$$ ||w||_{H_K(\Omega)}^2 = ||w||_{L^2(\Omega)}^2 + \int_\Omega \int_{\Omega} |w(x)-w(y)|^2 \, K_\Omega(x,y) dxdy$$
is finite.

Similar, we denote with $H_{K,loc}(\Omega)$ the space of functions for which the quantity
\[
||w||_{L^2(\Omega \cap B)}^2 + \int_{\Omega \cap B} \int_{\Omega \cap B} |w(x)-w(y)|^2 \, K_\Omega(x,y) dx\, dy
\]
is finite for any ball $B \subset \mathbb{R}^N$.

\begin{defi}\label{Def:WeakSolNeumann}
Let $\Omega \subset \RR^N$ be any Lipschitz domain, $B\subset \RR^N$ be a ball, and $D := B\cap\Omega$.
Let $K_\Omega$ be any kernel of the form either \eqref{K1} or \eqref{K2}, and let $L_\Omega$ and $B$ be given by \eqref{new-operator} and \eqref{B}, respectively.
Let $\mu,f \in L^q(D)$ with $q \in \left(\frac{N}{2s},\infty\right]$.

We say that $u \in H_{K,loc}(\Omega)$ is a weak supersolution in $D$, with Neumann conditions on $\partial\Omega \cap B$, and we write
\[
L_\Omega u \geq \mu u + f \quad \text{ in } D,
\]
if
\[
B(u,\eta) \geq \int_{D} \mu u \eta dx + \int_{D} f\eta dx \quad \text{ for all } \eta \in C_0^{\infty}(B), \, \eta \geq 0.
\]
We say that $u \in H_{K,loc}(\Omega)$ is a weak subsolution in $D$, with Neumann conditions on $\partial\Omega \cap B$, and we write
\[
L_\Omega u \leq \mu u + f \quad \text{ in } D,
\]
if
\[
B(u,\eta) \leq \int_{D} \mu u \eta dx + \int_{D} f\eta dx \quad \text{ for all } \eta \in C_0^{\infty}(B), \, \eta \geq 0.
\]
We say that $u \in H_{K,loc}(\Omega)$ is a weak solution to
\[
L_\Omega u = \mu u + f \quad \text{ in } D,
\]
with Neumann conditions on $\partial\Omega \cap B$, if it is both a weak supersolution and subsolution in $D$ with Neumann conditions on $\partial\Omega \cap B$.

Finally, we say that $u$ is a weak (sub/super)-solution in $\Omega$ if the previous definition holds for all balls $B\subset\RR^N$.
\end{defi}

We will also need the following.

\begin{lem}\label{Lemma:Subsolutions}
Let $\Omega$ be a bounded Lipschitz domain  and $K_\Omega$, $B$, $f$, $\mu$, as in Definition \ref{Def:WeakSolNeumann}.

Then, the following statements hold.
\begin{itemize}
\item[(i)] Let $u$ satisfy
\[
L_\Omega u = \mu u + f \quad \text{ in } D,
\]
with Neumann condition on $\partial\Omega \cap B$. Then $u_+$  and $u_-$ satisfy respectively
\[
L_\Omega u_+ \leq \mu u_+ + f_+ \quad \text{ in } D,
\]
and
\[
L_\Omega u_- \geq \mu u_- + f_- \quad \text{ in } D,
\]
with Neumann condition on $\partial\Omega \cap B$.

\item[(ii)] Let $\mu, f \geq 0$ and $u$ a nonnegative function weakly satisfying
\[
L_\Omega u \leq \mu u + f \quad \text{ in } D,
\]
with Neumann condition on $\partial\Omega \cap B$. Then for any $l \geq 0$, the function $\underline{u} = \max\{u,l\}$ also satisfies
\[
L_\Omega \underline{u} \leq \mu \underline{u} + f \quad \text{ in } D,
\]
with Neumann condition on $\partial\Omega \cap B$.

\end{itemize}
\end{lem}
\begin{proof}
We follow the proof of \cite[Lemma 2.4]{K}.
The proof is very general and does not really use the explicit form of the kernel.

\medskip

Let us first prove (i). Setting $p(x) = x_+$, we consider a sequence of smooth and convex functions $p_k:\RR \to \RR$, such that
\begin{equation}\label{eq:Proppk}
p_k,p_k' \geq 0, \quad p_k(x) = p(x), \; x\in \RR\setminus (-\tfrac{1}{k},\tfrac{1}{k}), \quad \|p - p_k\|_{H^1(\RR)} \leq \tfrac{1}{k},
\end{equation}
for all positive integer $k$. Using the convexity of $p_k$, it is not difficult to verify that
\[
B(p_k(u),\eta) \leq B(u,p_k'(u)\eta),
\]
for all $k$ and all $\eta \in H_K(\Omega)$, $\eta \geq 0$. Further, we notice that, thanks to the properties of $p_k$ and the fact that $u \in H_K(\Omega)$, $p_k'(u)\eta$ is an admissible test, whenever $\eta \in H_K(\Omega)$ (by approximation it is always possible to test with functions belonging to $H_K(\Omega)$).

Consequently,
\[
\begin{aligned}
B(p_k(u),\eta) &- \int_\Omega \mu p_k(u) \eta dx - \int_\Omega f_+\eta dx \\
&\leq B(u,p_k'(u)\eta) - \int_\Omega \mu p_k(u) \eta dx - \int_\Omega f_+\eta dx \\
& = \int_\Omega \mu u p_k'(u) \eta dx + \int_\Omega f p_k'(u)\eta dx - \int_\Omega \mu p_k(u) \eta dx - \int_\Omega f_+\eta dx,
\end{aligned}
\]
for all $k$ and all $\eta \in H_K(\Omega)$, $\eta \geq 0$. Finally, passing to the limit as $k \to +\infty$, and noticing that $\int_\Omega \mu u p_k'(u) \eta, \int_\Omega \mu p_k(u) \eta \to \int_\Omega \mu u_+ \eta$, it follows
\[
B(u_+,\eta) - \int_\Omega \mu u_+ \eta dx - \int_\Omega f_+\eta dx \leq \int_{\Omega\cap\{u > 0\}} f\eta dx - \int_\Omega f_+ \eta dx \leq 0,
\]
for all $\eta \in H_K(\Omega)$, $\eta \geq 0$, which proves the first part of our claim. To prove the second part, it is enough to notice that $-u$ is a solution with $-f$ and apply the first part of our statement.
We obtain that $u_- = (-u)_+$ is a subsolution with $f_- = (-f)_+$, which is exactly what we wanted to prove.

\medskip

To prove part (ii), we proceed as before. We fix $l \geq 0$ and we define $p(x) := \max\{x,l\}$. Then, we consider a sequence of smooth and convex functions $p_k$ satisfying \eqref{eq:Proppk}. Thus,
\[
\begin{aligned}
B(p_k(u),\eta) &- \int_\Omega \mu p_k(u) \eta dx - \int_\Omega f\eta dx \\
&\leq B(u,p_k'(u)\eta) - \int_\Omega \mu p_k(u) \eta dx - \int_\Omega f\eta dx \\
& \leq \int_\Omega \mu u p_k'(u) \eta dx + \int_\Omega f p_k'(u)\eta dx - \int_\Omega \mu p_k(u) \eta dx - \int_\Omega f\eta dx,
\end{aligned}
\]
for all $k$ and all $\eta \in H_K(\Omega)$, $\eta \geq 0$. Passing to the limit as $k \to + \infty$, we obtain
\[
B(p(u),\eta) - \int_\Omega \mu p(u) \eta dx - \int_\Omega f\eta dx   \leq  - l \int_{\Omega\cap\{u < l\}} \mu \eta dx -\int_{\Omega\cap\{u < l\}} f\eta dx  \leq 0,
\]
for all $\eta \in H_K(\Omega)$, $\eta \geq 0$, and our statement follows.


\end{proof}

\section{$L^\infty$ bounds}
\label{sec3}

The aim of this section is to prove $L^\infty$ bounds for solutions to the Neumann problems that we study.
For this, we only need the lower bound $K_\Omega(x,y) \gtrsim |x-y|^{-N-2s}$.

We next prove the boundedness of solutions to \eqref{eq-intro} and \eqref{eq-intro-reg}.
We start with the following.

\begin{lem}\label{Lemma:L2LinfEstimate}
Let $\Omega \subset \RR^N$ be a bounded Lipschitz domain and $c \in L^q(\Omega)$ and $q>\frac{N}{2s}$.
Let $K_\Omega$ be of the form either \eqref{K1} or \eqref{K2}.
Assume that $u$ satisfies
\begin{equation}\label{eq:SubsolPot}
\begin{cases}
L_\Omega u \leq c(x)u \quad &\text{ in } \Omega \\
u \geq 0       \quad &\text{ in } \Omega,
\end{cases}
\end{equation}
in the weak sense with Neumann conditions on $\partial\Omega$. Then
\[
\|u\|_{L^{\infty}(\Omega)} \leq C \left( 1+ \| c \|_{L^q(\Omega)}^{\frac{qN}{4qs - 2N}}\right) \|u\|_{L^2(\Omega)},
\]
for some constant $C > 0$ depending only on $N$, $s$, $q$, and $\Omega$.
\end{lem}

\begin{proof}
Note that by scaling properties we can assume $\| c \|_{L^q(\Omega)}\leq 1$. That is, we only need to work with the auxiliary function $w(x) = u \left(||c||_{L^q(\Omega)}^\frac{q}{N-2qs} x\right)$ in $\widetilde{\Omega}=||c||_{L^q(\Omega)}^\frac{q}{2N-4qs} \Omega \subset \Omega$ when $\| c \|_{L^q(\Omega)}> 1$. Given $\beta \geq 2$, the idea is to take $u^{\beta-1}$ as test function in the weak formulation and thanks to Sobolev inequality, improve iteratively the integrability of $u$. Since a priori we cannot guaranteed that $u^{\beta-1}\in H_K(\Omega)$ we need to truncate it in some sense in order to be an admissible test function. That is, let us consider the sequence
\[
u_k := \min\{u,k\},
\]
for all $k \in \NN$, $k \geq 1$. We have $u_k \in H_K(\Omega)$, $0 \leq u_k \leq u_{k+1}$ and $u_k \to u$ a.e. in $\Omega$. Testing the inequality with $\eta = u_k^{\beta - 2} u$, we immediately deduce
\begin{equation}\label{eq:BoundIneq0}
B(u,u_k^{\beta-2}u) \leq \int_\Omega c(x) u_k^{\beta-2}u^2 dx.
\end{equation}
Note that the fact $u_k^{\beta - 2}u \in H_K(\Omega)$, for $\beta \geq 2$, can be easily checked.

Now, setting $v := u_k^{\beta/2-1}u$ and applying \cite[Lemma 2.3]{K}, we obtain
\begin{equation}\label{eq:BoundIneq1}
B(v,v) \leq \beta B(u,u_k^{\beta-2}u)
\end{equation}
for all $\beta \geq 2$.  On the other hand, by H\"older inequality, we have
\begin{equation} \label{eq:BoundIneq2}
\int_\Omega c(x) u_k^{\beta-2}u^2 dx \leq \|c\|_{L^q(\Omega)} \|v\|^2_{L^{2q'}(\Omega)} \leq \|v\|^2_{L^{2q'}(\Omega)}.
\end{equation}
Since $q > \frac{N}{2s}$, it follows that $2 < 2q' < 2_s^\ast$ and so, taking $\vartheta \in (0,1)$ satifying
\[
\frac{1}{2q'} = \frac{\vartheta}{2} + \frac{1-\vartheta}{2_s^\ast}, \qquad \text{ i.e. } \vartheta = \frac{2qs - N}{2qs},
\]
and using the interpolation and the Sobolev inequality, we obtain
\begin{equation} \label{eq:BoundIneq3}
\|v\|^2_{L^{2q'}(\Omega)} \leq \|v\|^{2\vartheta}_{L^2(\Omega)} \|v\|^{2(1-\vartheta)}_{L^{2_s^\ast}(\Omega)} \leq C \left(\|v\|^{2}_{L^2(\Omega)} + B(v,v) \right)^{1-\vartheta}  \|v\|^{2\vartheta}_{L^2(\Omega)}.
\end{equation}
Now, thanks to the fact that $\vartheta \in (0,1)$, we infer
\begin{equation}\label{eq:BoundIneq4}
\left(\|v\|^{2}_{L^2(\Omega)} + B(v,v) \right)^{1-\vartheta}  \|v\|^{2\vartheta}_{L^2(\Omega)} \leq \varepsilon B(v,v) + (1 + \varepsilon^{-\frac{1-\vartheta}{\vartheta}} ) \|v\|^2_{L^2(\Omega)},
\end{equation}
for all $\varepsilon > 0$. Putting together \eqref{eq:BoundIneq0}, \eqref{eq:BoundIneq1}, \eqref{eq:BoundIneq2}, \eqref{eq:BoundIneq3}, \eqref{eq:BoundIneq4} and choosing
\[
\varepsilon = \left(C\, \beta \right)^{-1},
\]
it follows by taking into account that $\beta\geq 2$ that
\[
B(v,v) \leq C \beta^{\frac{1}{\vartheta}} \|v\|^2_{L^2(\Omega)},
\]
and, using Sobolev inequality again, we deduce
\begin{equation}\label{eq:BoundFundIneq}
\left( \int_\Omega u^2 u_k^{\beta\gamma-2} dx\right)^{\frac{1}{\beta\gamma}} \leq (C\beta)^{\frac{1}{\beta\vartheta}} \left( \int_\Omega u^2 u_k^{\beta-2} dx\right)^{\frac{1}{\beta}},
\end{equation}
for some new constant $C > 0$ depending only on $N$, $s$, $q$, and the Lipschitz norm of $\partial \Omega$. Here, $\gamma := 2_s^\ast/2 > 1$.

Now, taking $\beta_0 = 2$ and $\beta_i := \gamma\beta_{i-1} = \beta_0 \gamma^i$ for all integers $i \geq 1$, and iterating \eqref{eq:BoundFundIneq}, we obtain
\[
\begin{aligned}
\| u_k \|_{L^{2\gamma^j}(\Omega)}^\vartheta \leq \| u \|_{L^2(\Omega)}^\vartheta \sum_{i=0}^{j-1} (C\,\gamma^i)^\frac{1}{2\gamma^i}  \leq \| u \|_{L^2(\Omega)}^\vartheta \sum_{i=0}^{\infty} (C\,\gamma^i)^\frac{1}{2\gamma^i} = C\,\| u \|_{L^2(\Omega)}^\vartheta.
\end{aligned}
\]
Thus, passing to the limit as $j \to +\infty$, it follows
\begin{equation} 
\| u_k \|_{L^\infty(\Omega)}^\vartheta \leq C \| u \|_{L^2(\Omega)}^\vartheta.
\end{equation}

Finally, since the previous inequality holds for any $k$ with the same constant $C$, we conclude that
\[
\| u \|_{L^\infty(\Omega)} \leq C \| u \|_{L^2(\Omega)}.
\]
\normalcolor
\end{proof}

We now prove the following result, which gives the boundedness of solutions.
We notice that, in case of \eqref{eq-intro}, a similar result has been obtained in \cite{DPV}, with a different proof.

\begin{prop}\label{Prop:GlobalL2Linf}
Let $\Omega \subset \mathbb{R}^N$ be a bounded Lipschitz  domain, $\mu,f \in L^q(\Omega)$, with $q > \frac{N}{2s}$.
Let $K_\Omega$ be of the form either \eqref{K1} or \eqref{K2}.
Let $u$ be a weak solution to
\[
L_\Omega u = \mu u + f \quad \text{ in } \Omega,
\]
with Neumann conditions on $\partial\Omega$. Then,
\[
\|u\|_{L^{\infty}(\Omega)} \leq C \left( \|u\|_{L^2(\Omega)} + \|f\|_{L^q(\Omega)} \right),
\]
for some constant $C > 0$ depending only on $N$, $s$, $q$, $\|\mu\|_{L^q(\Omega)}$ and $\Omega$.
\end{prop}
\begin{proof}
Thanks to Lemma \ref{Lemma:Subsolutions} (part (i)), we know that $u_+$ is a nonnegative subsolution with $\mu = \mu_+$ and $f = f_+$. Consequently, the function $v = \max\{u_+,1\}$ is still a subsolution and, furthermore, $v \geq 1$ (Lemma \ref{Lemma:Subsolutions} part (ii)). Consequently, $v$ satisfies
\[
L_\Omega v \leq c(x)v \quad \text{ in } \Omega
\]
in the weak sense (with Neumann conditions on $\partial\Omega$), where $c = \mu_+ + f_+$.

Now, note that if $\|u_+\|_{L^2(\Omega)} \leq 1$ then $\|v\|_{L^2(\Omega)} \leq \sqrt{1 + |\Omega|}$ and so, under the assumptions $\|u_+\|_{L^2(\Omega)} \leq 1$ and $\|f_+\|_{L^q(\Omega)} \leq 1$, it follows by Lemma \ref{Lemma:L2LinfEstimate}
\[
\|u_+\|_{L^\infty(\Omega)} \leq \|v\|_{L^\infty(\Omega)} \leq C,
\]
for some constant depending only on $N$, $s$, $q$, $\|\mu_+\|_{L^q(\Omega)}$ and $\Omega$. Applying the above inequality to the subsolution
\[
w = \frac{u_+}{\|u_+\|_{L^2(\Omega)} + \|f_+\|_{L^q(\Omega)}},
\]
we deduce
\[
\|u_+\|_{L^\infty(\Omega)} \leq C \left( \|u_+\|_{L^2(\Omega)} + \|f_+\|_{L^q(\Omega)} \right),
\]
for some constant depending only on $N$, $s$, $q$, $\|\mu_+\|_{L^q(\Omega)}$ and $\Omega$. Finally, repeating the same procedure for the subsolution $u_-$ (with $\mu = \mu_-$ and $f = f_-$), we complete the proof of the theorem.
\end{proof}

We will also need the following.
Here, we denote $D_R(x_0) = \Omega \cap B_R(x_0)$.

\begin{lem} \label{lemma: SolutionDirichletNeumann}
	Let $\Omega \subset \RR^N$ be a domain, $R > 0$, $x_0 \in \Omega$ and $f \in L^{q}(D_{2R}(x_0))$ with $q>\frac{N}{2s}$.
Let $K_\Omega$ be of the form either \eqref{K1} or \eqref{K2}.
Moreover, assume that $\partial \Omega \cap B_{3R}(x_0)$ is a Lipschitz graph. Then, there is a weak solution to
	\begin{equation}\label{eq:ProbLinfLqestimate}
	\begin{cases}
	L_\Omega v = |f| \quad &\text{ in } D_{2R}(x_0), \\
	v = 0       \quad &\text{ in } \Omega \setminus D_{2R}(x_0), \\
	\end{cases}
	\end{equation}
	with Neumann conditions on $\partial\Omega \cap B_{2R}(x_0)$ in the sense of Definition \ref{Def:WeakSolNeumann}. Furthermore, it satisfies
	$$0\leq v \leq \kappa_0 \, R^{2s-\frac{N}{q}}\,||f||_{L^q(D_{2R}(x_0))} \ \ \text{ in }  \ \ D_{2R}(x_0),$$
	for some nonnegative constant $\kappa_0$ depending only on $N$, $s$, $q$, and the Lipschitz norm of $\partial\Omega \cap B_{3R}(x_0)$.
\end{lem}
\begin{proof}
Since the general case comes by scaling, we take $R = 1$. First, let us notice that the existence (and uniqueness) of such solution $v$ can be obtained by minimizing the functional
	$$ \mathcal{E}(w) = \frac{1}{4} \int_\Omega \int_\Omega |w(x)-w(y)|^2 K_\Omega(x-y) \, dx\, dy - \int_\Omega |f(x)|w(x)\,dx $$
	among all functions $w\in H_K(\Omega)$ such that $w\equiv 0$ in $\Omega \setminus D_{2}(x_0)$. See \cite[Section~3]{R-survey} for the details in case of the fractional Laplacian.
	
	Next, in order to prove that the solution is nonnegative we can use the same argument of \cite[Theorem~4.1]{R-survey}, consisting on using $v_-$ as a test function in the weak formulation, which yields $v_-\equiv 0$ in $\Omega$. The bound from above is more delicate and we need to repeat the arguments from Lemma~\ref{Lemma:L2LinfEstimate} and Proposition~\ref{Prop:GlobalL2Linf} adapted to this setting of mixed Dirichlet and Neumann conditions. In that way we obtain that
	$$ v\leq C ( ||v||_{L^2(D_{2}(x_0))} + ||f||_{L^q(D_{2}(x_0))}) \ \ \text{ in }  \ \ D_{2}(x_0), $$
	where $C$ is a nonnegative constant depending only on $N$, $s$, $q$, and the Lipschitz norm of $\partial\Omega \cap B_3(x_0)$.
	
	Finally, we need to estimate the $L^2$-norm of $v$ in terms of the $L^q$-norm of $f$. In order to do that it is sufficient to use $v$ as a test function in the weak formulation and applying the fractional Poincar\'e inequality in $D_3(x_0)$. That is,
	\begin{align*}
	||v||_{L^2(D_{2}(x_0))}^2 &= ||v||_{L^2(D_{3}(x_0))}^2 \leq C_P [v]_{H^s(D_{3}(x_0))}^2 \leq C_P [v]_{H^s(\Omega)}^2 \leq C\int_{D_2(x_0)} f v \\ &\leq C ||f||_{L^q(D_{2}(x_0))} \, ||v||_{L^2(D_{2}(x_0))}.
	\end{align*}
	Let us remark that we apply the fractional Poincar\'e inequality in $D_{3}(x_0)$ since we need $v$ to be zero in some subset of the domain of $v$.
\end{proof}

\section{Moser-type iteration and H\"older regularity up to the boundary}
\label{sec4}

The goal of this section is to develop a Moser-type iteration for our nonlocal problem with Neumann boundary conditions.
The overall strategy follows that of Kassmann \cite{K} for interior regularity but, as we will see, the logarithmic singularity of the kernel in \eqref{K1} will introduce several difficulties.

From now on,  for any $r >0$ and $x_0 \in \Omega$ we denote
\[
D_r(x_0) := B_r(x_0) \cap \Omega.
\]
The main result of this section is the following.

\begin{thm}\label{Theorem:HolderRegularity1}
 Let $\Omega \subset \RR^N$ be a domain, $R > 0$, $x_0 \in \Omega$ and $f \in L^{q}(D_{R}(x_0))$ with $q>\frac{N}{2s}$.
 Assume that $\partial \Omega \cap B_{R}(x_0)$ is a Lipschitz graph.
Let $K_\Omega$ be of the form either \eqref{K1} or \eqref{K2}.
Assume that $u$ is a weak bounded solution to
	\[
	L_\Omega u = f \quad \text{ in } D_{R}(x_0),
	\]
	with Neumann conditions on $\partial\Omega \cap B_{R}(x_0)$ in the sense of Definition \ref{Def:WeakSolNeumann}.
	
	Then there exist $\alpha \in (0,1)$ and $C$ depending only on the Lipschitz norm of $\partial \Omega \cap B_{R}(x_0)$, $N$, $s$, and $q$, such that
	\begin{equation}\label{eq:HolderEstimate1}
	|u(x) - u(y)| \leq C \left( \frac{|x - y|}{R} \right)^\alpha \left[ \|u\|_{L^{\infty}(\Omega)} + R^{2s-\frac{N}{q}} \|f\|_{L^{q}(D_{R}(x_0))}\right]
	\end{equation}
	for a.e. $x,y \in D_{R/2}(x_0)$.
\end{thm}

Theorem \ref{Theorem:HolderRegularity1} will be obtained through several auxiliary results. The first step in the proof is the following.

\begin{lem}\label{Lemma:TechAssumNonlocal1}
Let $\Omega \subset \RR^N$ be a domain, $R > 0$ and $x_0 \in \Omega$.
Let $K_\Omega$ be of the form either \eqref{K1} or \eqref{K2}.
Assume that $\partial \Omega \cap B_{2R}(x_0)$ is a Lipschitz graph.

Then for any $c>0$, $\delta_0 \in (0,1/2)$ and $\vartheta > 1$, there exists $\gamma \in (0,2s)$ depending only on the Lipschitz constant of $\partial \Omega \cap B_{2R}(x_0)$, $N$, $s$, $c$, $\delta_0$ and $\vartheta$	such that for any $u \in L^{\infty}(\Omega)$ satisfying
	\begin{equation}\label{eq:AssDecay1}
		\begin{cases}
			u(x) \geq 0 \quad &\text{ for a.e. } x \in D_R(x_0) \\
			u(x) \geq c \left[ 1 - \left( \vartheta \frac{|x-x_0|}{R} \right)^\gamma \right]  \quad &\text{ for a.e. } x \in \Omega \setminus B_R(x_0) \\
			\frac{|\{ u \geq 1 \} \cap D_R(x_0)|}{|D_R(x_0)|} \geq \frac{1}{2},
		\end{cases}
	\end{equation}
	it holds
	\begin{equation}\label{eq:SignIntegralOutside1}
		\int_{\Omega \setminus B_r(x_0)} u(x) K_\Omega(x,y)dx \geq 0 \quad \text{ for a.e. } y \in D_r(x_0),
	\end{equation}
	for all $r < R$ such that
	\begin{equation}\label{eq:AssTechLemOutside1}
		\frac{|\{u \geq 1\} \cap (D_R(x_0) \setminus D_r(x_0))|}{|D_R(x_0)|} \geq \delta_0.
	\end{equation}
\end{lem}

\begin{proof} Taking $u_R(x) = u(x_0 + Rx)$ instead of $u$, we may assume $R = 1$ and $x_0 = 0$.
	We prove the result for $K$ of the form \eqref{K1}; the case \eqref{K2} is simpler.
	
By the third assumption in \eqref{eq:AssDecay1}, we deduce the existence of $r_0 \in (0,1)$ depending only $\delta_0 >0$, $N$ and the Lipschitz constant of $\partial \Omega$ such that \eqref{eq:AssTechLemOutside1} holds if $r\leq r_0$.

	Let us take $r \leq r_0$ satisfying \eqref{eq:AssTechLemOutside1} and set $A_r := \{u \geq 1\} \cap (D_1 \setminus D_r)$.
	By assumption we have $|A_r| \geq \delta_0 |D_1|$, $u \geq 0$ in $D_1$ and so for a.e. $y \in D_r$, it follows
	\[
	\int_{D_1\setminus D_r} u(x) K_\Omega(x,y)dx \geq \int_{A_r} K_\Omega(x,y)dx \geq c\int_{A_r} \frac{1+\log^-\left(\frac{d_{x,y}}{|x-y|}\right)}{|x-y|^{N+2s}} dx
	\]
with $c>0$, where $d_{x,y} = \min\{d(x),d(y)\}$.
We have to find a suitable lower bound for the above integral.
To do so, we first notice that for any fixed $d > 0$, the function
	\[
	\varrho \to \frac{1 + \log^-\left(d/\varrho\right)}{\varrho^{N+2s}}, \quad \varrho > 0
	\]
	is decreasing and thus, since $|x-y| \leq 2$, we find
	\[
	\int_{A_r} \frac{1+ \log^-\left(\frac{d_{x,y}}{|x-y|}\right)}{|x-y|^{N+2s}} dx \geq 2^{-N-2s} \int_{A_r} 1 + \log^-\left(\frac{d_{x,y}}{2}\right) dx \geq c |A_r| \left( 1  + \log^-\left(\frac{d(y)}{2}\right) \right).
	\]
	Consequently, whenever $d(y) \geq 1$, we have
	\begin{equation}\label{eq:PosLemNewKerEst1}
		\int_{D_1\setminus D_r} u(x) K_\Omega(x,y)dx \geq c |A_r| \geq c\delta_0,
	\end{equation}
	for some $c > 0$ depending only on $N$, $s$ and $\Omega$.
	Conversely, when $0 < d(y) < 1$, we obtain by the inequality above
	\begin{equation}\label{eq:PosLemNewKerEst3}
		\int_{D_1\setminus D_r} u(x) K_\Omega(x,y)dx \geq C\delta_0 |D_1| (1 + |\log d(y)|).
	\end{equation}
	On the other hand, for a.e. $y \in D_r$, it holds
	\[
	\int_{\Omega \setminus B_1} u(x) K_\Omega(x,y)dx \geq - c  \int_{\Omega \setminus B_1} \left| 1 - \left( \vartheta |x| \right)^\gamma \right| K_\Omega(x,y) dx,
	\]
	thanks to the second inequality in \eqref{eq:AssDecay1}. Moreover,
	\begin{equation}\label{eq:PosLemNewKerEst2}
		\begin{aligned}
			\int_{\Omega \setminus B_1} \left| 1 - \left( \vartheta |x| \right)^\gamma \right| K_\Omega(x,y) dx &\leq C \int_{\Omega \setminus B_1} \left| 1 - \left( \vartheta |x| \right)^\gamma \right| \frac{1+\log^-\left(\frac{d_{x,y}}{|x-y|}\right)}{|x-y|^{N+2s}} dx \\
			& = C \int_{\Omega_1} \frac{\left| 1 - \left( \vartheta |x| \right)^\gamma \right|}{|x-y|^{N+2s}} \left| \log\left(\frac{d_{x,y}}{|x-y|}\right)\right| dx  \\
			&\quad + C \int_{\Omega \setminus B_1} \frac{\left| 1 - \left( \vartheta |x| \right)^\gamma \right|}{|x-y|^{N+2s}} dx := I_1(\gamma) + I_2(\gamma),
		\end{aligned}	
	\end{equation}
	where $\Omega_1 := (\Omega \setminus B_1)\cap \{d_{x,y} \leq |x-y|\}$. Notice that
	\[
	\begin{aligned}
	I_1(\gamma) &= \int_{\Omega_1\cap\{d_{x,y} = d(x)\}} \frac{\left| 1 - \left( \vartheta |x| \right)^\gamma \right|}{|x-y|^{N+2s}} \left| \log\left(\frac{d(x)}{|x-y|}\right)\right| dx \\
	&\quad + \int_{\Omega_1\cap\{d_{x,y} = d(y)\}} \frac{\left| 1 - \left( \vartheta |x| \right)^\gamma \right|}{|x-y|^{N+2s}} \left| \log\left(\frac{d(y)}{|x-y|}\right)\right| dx \\
	&\leq \int_{\Omega_1\cap\{d_{x,y} = d(x)\}} \frac{\left| 1 - \left( \vartheta |x| \right)^\gamma \right|}{|x-y|^{N+2s}} \left| \log\left(\frac{d(x)}{|x-y|}\right)\right| dx \\
	&\quad + \int_{\Omega_1\cap\{d_{x,y} = d(y)\}} \frac{\left| 1 - \left( \vartheta |x| \right)^\gamma \right|}{|x-y|^{N+2s}} \log |x-y| dx \\
	&\quad + |\log d(y)| \int_{\Omega_1\cap\{d_{x,y} = d(y)\}} \frac{\left| 1 - \left( \vartheta |x| \right)^\gamma \right|}{|x-y|^{N+2s}} dx.
	\end{aligned}
	\]
	Further, $\left| 1 - \left( \vartheta |x| \right)^\gamma \right| \to 0$ for a.e. $x \in \Omega\setminus B_1$ as $\gamma \to 0^+$. So, since $|\log d(x)|$ is integrable near $\partial\Omega$ and recalling that $|x-y| \geq 1 - r > 0$, we deduce the existence of $\delta_\gamma \to 0^+$ as $\gamma \to 0^+$ such that $I_1(\gamma) \leq \delta_\gamma (1 + |\log d(y)|)$ for all small $\gamma > 0$, by dominated convergence. Similar for $I_2(\gamma)$. Therefore, by \eqref{eq:PosLemNewKerEst1}, \eqref{eq:PosLemNewKerEst3} and \eqref{eq:PosLemNewKerEst2},
	\[
	\begin{aligned}
	\int_{\Omega \setminus B_r} u(x) K_\Omega(x,y)dx &= \int_{D_1 \setminus D_r} u(x) K_\Omega(x,y) dx + \int_{\Omega \setminus B_1} u(x) K_\Omega(x,y) dx \\
	& \geq C \delta_0 |D_1| \left( 1 + |\log d(y)|  \right) - \delta_\gamma (1 + |\log d(y)|) \geq 0,
	\end{aligned}
	\]
	if $\gamma > 0$ is small enough and our statement follows.
\end{proof}

Using the previous lemma, we can now prove the following.

\begin{lem} \label{Lemma:LpLminuspBound1}
Let $\Omega \subset \RR^N$ be a domain, $R > 0$ and $x_0 \in \Omega$.
Let $K_\Omega$ be of the form either \eqref{K1} or \eqref{K2}.
Assume that $\partial \Omega \cap B_{2R}(x_0)$ is a Lipschitz graph, and that $u$ satisfies
	\[
	\begin{cases}
	L_\Omega u \geq 0 \quad &\text{ in } D_{R}(x_0) \\
	u > 0       \quad &\text{ in } D_{R}(x_0),
	\end{cases}
	\]
	with Neumann conditions on $\partial\Omega \cap B_R(x_0)$.
	Assume also that $u$ satisfies \eqref{eq:SignIntegralOutside1} with $r=R$.
	
	Then,
	\[
	\left( \fint_{D_R(x_0)} u(x)^{\beta_0} dx \right)^{1/\beta_0} \leq C \left(\fint_{D_R(x_0)} u(x)^{-\beta_0} dx \right)^{-1/\beta_0},
	\]
for some $\beta_0 \in (0,1)$ and $C > 0$ depending only on the Lipschitz constant of $\partial \Omega \cap B_{2R}(x_0)$, $N$, and $s$.
\end{lem}
\begin{proof} The proof is basically the same for both classes of kernels,  \eqref{K1} and \eqref{K2}.

By scaling and translation we may assume $R = 1$ and $x_0 = 0$.
Given any arbitrary $z_0 \in D_1$ and $\varrho > 0$ such that $B_{2\varrho}(z_0) \subset D_1$, we take $B_\varrho = B_\varrho(z_0)$. Then, exactly as in \cite[Lemma 3.3]{K} with $r=\varrho$ (here we use the assumption \eqref{eq:SignIntegralOutside1}), we find
	\[
	\int_{B_\varrho \times B_\varrho} \frac{[\log u(x) - \log u(y)]^2}{|x-y|^{N+2s}} dxdy \leq C \varrho^{N-2s},
	\]
	for some constant $C > 0$ depending only on $N$, $s$ and the constants in \eqref{K1}-\eqref{K2} (which depend only on the Lipschitz norm of the domain).
This yields $\log u \in H^s(B_\varrho)$ and thus, by the Poincar\'e inequality,
\[
	\int_{B_\varrho} \big| \log u(x) - [\log u]_{B_\varrho} \big|^2 dx \leq C \varrho^N,
	\]
	for some constant $C$ depending only on $N$, $s$ and the constants in \eqref{K1}-\eqref{K2}, where $[\log u]_{B_\varrho} := \fint_{B_\varrho} \log u$.   By H\"older inequality, it follows that
	\[
	\int_{B_\varrho} \big| \log u(x) - [\log u]_{B_\varrho} \big| dx \leq C \varrho^N,
	\]
	and therefore, thanks to the arbitrariness of $z_0$ and $\varrho > 0$, we deduce that $\log u \in \text{BMO}(D_1)$ (see \cite[Theorem 0.3]{1999Buckley:art}).
	Now, by the John-Nirenberg inequality (see \cite[Theorem 0.3 and Theorem 0.4]{1999Buckley:art}), we deduce the existence of $\beta_0 \in (0,1)$ and $C$, depending only on the Lipschitz constant of $\partial \Omega$, $N$, and $s$, such that
	\[
	\int_{D_1} e^{\beta_0 | \log u(x) - [\log u]_{D_1}|} dx \leq C.
	\]
	Finally, since
	\[
	\begin{aligned}
	& \left( \fint_{D_1} u(x)^{\beta_0} dx \right)^{1/\beta_0} \cdot \left(\fint_{D_1} u(x)^{-\beta_0} dx \right)^{1/\beta_0} \\
	= & \left(\fint_{D_1} e^{\beta_0 \{ \log u(x) - [\log u]_{D_1}\}} dx \right)^{1/\beta_0} \cdot \left( \fint_{D_1} e^{-\beta_0 \{ \log u(x) - [\log u]_{D_1}\}} dx \right)^{1/\beta_0} \leq C,
	\end{aligned}
	\]
	the result follows.
\end{proof}

On the other hand, we next prove a key lemma for the Moser-type iteration.

\begin{lem} \label{lemma-noname}
Let $\Omega \subset \RR^N$ be a domain, $R > 0$, $x_0 \in \Omega$ and $\beta>1$.
Let $K_\Omega$ be of the form either \eqref{K1} or \eqref{K2}.
Assume that $u$ satisfies
	\[
	\begin{cases}
	L_\Omega u \geq 0 \quad &\text{ in } D_R(x_0) \\
	u > 0       \quad &\text{ in } D_R(x_0),
	\end{cases}
	\]
	with Neumann conditions on $\partial\Omega \cap B_R(x_0)$, in the sense of Definition \ref{Def:WeakSolNeumann}.

	Then, there exists a constant $C$ depending only on $N$, $s$, and the Lipschitz constant of $\partial \Omega \cap B_{2R}(x_0)$, such that
	\begin{equation}\label{eq:LocalRevPoin1}
		\begin{aligned}
			\int_{D_r(x_0) \times D_r(x_0)} &\frac{\big[ u(x)^{\frac{1-\beta}{2}} - u(y)^{\frac{1-\beta}{2}} \big]^2}{|x - y|^{N+2s}} dxdy \\
			&\leq \frac{C \beta^2}{(R-r)^{2s}}  \int_{D_R(x_0)} u(x)^{1-\beta} \left( 1 +  \left|\log \frac{d(x)}{R-r}\right| \right) dx,
		\end{aligned}
	\end{equation}
	for all $0 < r < R$.
	In case \eqref{K2}, the same estimate holds without the logarithmic term.
\end{lem}
\begin{proof} Since the kernels and \eqref{eq:LocalRevPoin1} are scale-invariant, after a rescaling we may assume that $R-r = 1$.
We take a smooth cut-off function $0 \leq \varphi \leq 1$ satisfying
	\[
	\varphi = 1 \quad \text{in } \overline{B}_r, \quad \supp (\varphi) \subset B_R, \quad \sup |\nabla \varphi| \leq c.
	\]
	Testing $L_\Omega u \geq 0$ in $D_{R}$ with $\eta := \varphi^{1 + \beta}u^{-\beta}$ (notice that $\eta$ is an admissible test since $u > 0$ in $D_R$ and $\varphi = 0$ in $\Omega \setminus B_R$), it follows that
	\[
	\int_{\Omega}\int_\Omega  [u(x) - u(y)][\varphi^{1 + \beta}(x)u^{-\beta}(x) - \varphi^{1 + \beta}(y)u^{-\beta}(y)] K_\Omega(x,y)dxdy \geq 0.
	\]
In particular, for any $\varepsilon > 0$,
	\[
	\begin{aligned}
	\int\int_{\substack{ \Omega \times \Omega  \\ |x-y| > \varepsilon}} & [u(y) - u(x)][\varphi^{1 + \beta}(x)u^{-\beta}(x) - \varphi^{1 + \beta}(y)u^{-\beta}(y)] K_\Omega(x,y)dxdy \\
	\leq & - \int_{\substack{ \Omega \times \Omega  \\ |x-y| \leq \varepsilon}}  [u(y) - u(x)][\varphi^{1 + \beta}(x)u^{-\beta}(x) - \varphi^{1 + \beta}(y)u^{-\beta}(y)] K_\Omega(x,y)dxdy.
	\end{aligned}
	\]
	
	Now, we apply \cite[Lemma 2.5]{K} with $a = u(x)$, $b = u(y)$, $\tau_1 = \varphi(x)$, $\tau_2 = \varphi(y)$ and $p = \beta$, integrate on $(\Omega \times \Omega) \cap \{|x-y| > \varepsilon\}$ and use the above inequality to obtain
	\[
	\begin{aligned}
	\int\int_{\substack{ \Omega \times \Omega \\ |x-y| > \varepsilon}} & \varphi(x) \varphi(y) \left[ \left(\frac{u(x)}{\varphi(x)}\right)^{\frac{1-\beta}{2}} - \left(\frac{u(y)}{\varphi(y)}\right)^{\frac{1-\beta}{2}}\right]^2 K_\Omega(x,y) dxdy \\
	\leq & \, c_{\beta} \int_{\substack{ \Omega \times \Omega  \\ |x-y| > \varepsilon}}  [\varphi(x) - \varphi(y)]^2 \left[ \left(\frac{u(x)}{\varphi(x)}\right)^{1-\beta} + \left(\frac{u(y)}{\varphi(y)}\right)^{1-\beta}\right] K_\Omega(x,y) dxdy \\
	& - (\beta - 1)\int_{\substack{ \Omega \times \Omega  \\ |x-y| \leq \varepsilon}}  [u(y) - u(x)][\varphi^{1 + \beta}(x)u^{-\beta}(x) - \varphi^{1 + \beta}(y)u^{-\beta}(y)] K_\Omega(x,y)dxdy,
	\end{aligned}
	\]
	where $c_\beta := \max\{\frac{\beta-1}{2},\frac{6(\beta-1)^2}{16}\} \leq \beta^2$, since $\beta > 1$. Since $\eta = \varphi^{1+\beta} u^{-\beta} \in H_K(D_R)$, the last term converges to zero when we pass to limit as $\varepsilon \to 0$.
	Thus, we deduce
	\[
	\begin{aligned}
	\int_\Omega\int_{ \Omega } & \varphi(x) \varphi(y) \left[ \left(\frac{u(x)}{\varphi(x)}\right)^{\frac{1-\beta}{2}} - \left(\frac{u(y)}{\varphi(y)}\right)^{\frac{1-\beta}{2}}\right]^2 K_\Omega(x,y) dxdy \\
	\leq & \, \beta^2 \int_{ \Omega}\int_\Omega  [\varphi(x) - \varphi(y)]^2 \left[ \left(\frac{u(x)}{\varphi(x)}\right)^{1-\beta} + \left(\frac{u(y)}{\varphi(y)}\right)^{1-\beta}\right] K_\Omega(x,y) dxdy.
	\end{aligned}
	\]
	Now, using that $\varphi\equiv1$ in $D_r$, we bound from below the left hand side as
	\[
	\begin{aligned}
	\int_{ \Omega}\int_\Omega  &\varphi(x) \varphi(y) \left[ \left(\frac{u(x)}{\varphi(x)}\right)^{\frac{1-\beta}{2}} - \left(\frac{u(y)}{\varphi(y)}\right)^{\frac{1-\beta}{2}}\right]^2 K_\Omega(x,y) dxdy \\
	\geq & \int_{D_r}\int_{D_r}  \varphi(x) \varphi(y) \left[ \left(\frac{u(x)}{\varphi(x)}\right)^{\frac{1-\beta}{2}} - \left(\frac{u(y)}{\varphi(y)}\right)^{\frac{1-\beta}{2}}\right]^2 K_\Omega(x,y) dxdy \\
	\geq & \, c \int_{D_r}\int_{D_r} \frac{\big[ u(x)^{\frac{1-\beta}{2}} - u(y)^{\frac{1-\beta}{2}} \big]^2}{|x - y|^{N+2s}} dxdy,
	\end{aligned}
	\]
	where $c > 0$ depends only on $N$, $s$ and the Lipschitz constant of $\partial \Omega \cap B_{2R}(x_0)$. Here we have used \eqref{new-kernel} and that $k_\Omega \geq 0$.

	On the other hand, by symmetry, we have
	\[
	\begin{aligned}
	\int_{ \Omega}\int_\Omega  &[\varphi(x) - \varphi(y)]^2 \left[ \left(\frac{u(x)}{\varphi(x)}\right)^{1-\beta} + \left(\frac{u(y)}{\varphi(y)}\right)^{1-\beta}\right] K_\Omega(x,y) dxdy \\
	= & \, 2 \int_{ \Omega}\int_\Omega  \varphi(x)^{\beta-1} [\varphi(x) - \varphi(y)]^2  u(x)^{1-\beta} K_\Omega(x,y)  dxdy \\
	\leq & \, 2 \int_{D_R} u(x)^{1-\beta} \int_{\Omega} [\varphi(x) - \varphi(y)]^2 K_\Omega(x,y) dydx.
	\end{aligned}
	\]

Therefore, we have proved that
\[
\begin{aligned}
\int_{D_r}\int_{D_r} \frac{\big[ u(x)^{\frac{1-\beta}{2}} - u(y)^{\frac{1-\beta}{2}} \big]^2}{|x - y|^{N+2s}} dxdy \leq C \beta^2  \int_{D_R} u(x)^{1-\beta} \int_\Omega [\varphi(x) - \varphi(y)]^2 K_\Omega(x,y)dy dx.
\end{aligned}
\]
To finish the proof, we have to estimate the integral
	\[
	\begin{aligned}
	\int_{\Omega} [\varphi(x) - \varphi(y)]^2 K_\Omega(x,y) dy &= 	\int_{D_1(x)} [\varphi(x) - \varphi(y)]^2 K_\Omega(x,y) dy \\
& \quad + \int_{\Omega \setminus B_1(x)} [\varphi(x) - \varphi(y)]^2 K_\Omega(x,y) dy
:= J_1 + J_2,
	\end{aligned}
	\]
	where $x \in D_R$ is fixed and $d := d(x) < 1$. In view of \eqref{new-kernel} and \eqref{new-kernel2-estimates}, have
	\begin{equation}\label{eq:Est1J1Neumann2}
		\begin{aligned}
			J_1 &\leq C \int_{D_{1}(x)} \frac{|\log|x-y|| + |\log d(x)| + |\log d(y)| }{|x-y|^{N+2s-2}} dy \\
&\leq C (1+|\log d(x)|) + C\int_{D_{1}(x)\cap \{d(x)/2\leq d(y)\leq 2\}} \frac{|\log d(y)| }{|x-y|^{N+2s-2}} dy\\
& \ \ \ \ \ \ + C\int_{D_{1}(x)\cap \{d(y)\leq d(x)/2\}} \frac{|\log d(y)| }{|x-y|^{N+2s-2}} dy \\
&= C (1+|\log d(x)|) + I_1+ I_2.
		\end{aligned}
	\end{equation}
	Now, taking into account that $|\log d(y)|\leq C (1+|\log d(x)|)$ when $d(x)/2\leq d(y)\leq 2$ we obtain that $I_1 \leq C (1+|\log d(x)|)$. Next, in order to estimate $I_2$ it is enough to consider the case in which $D_1$ is flat since any other Lipschitz domain can be transform through a bi-Lipschitz transformation. In that case,
	\begin{equation}\label{eq:Est2J1Neumann2}
	\begin{aligned}
	I_2 &= C\int_{D_{1}(x)\cap \{0\leq y_N\leq x_N/2\}} \frac{|\log y_N| }{|x-y|^{N+2s-2}} dy \\
	&\leq - C\int_{x_N/2}^{x_N} \log(x_N-y_N) \left(\int_{B_1\subset \RR^{N-1}} (y_N^2+|z|^2)^{\frac{-N-2s+2}{2}} dz \right) dy_N \\
		&\leq - C\int_{x_N/2}^{x_N} \log(x_N-y_N) (1+y_N^{1-2s}) dy_N \\
		&\leq C (1+|\log x_N|) = C (1+|\log d(x)|).
	\end{aligned}
	\end{equation}
	Here, we have used the following estimate
	\begin{equation*}
	\begin{aligned}
	\int_{B_1\subset \RR^{N-1}} &(y_N^2+|z|^2)^{\frac{-N-2s+2}{2}} dz \leq C\, y_N^{1-2s} \int_0^{1/y_N} \frac{r^{N-2}}{(1+r^2)^\frac{N+2s-2}{2}} dr \\
	&\leq C\,y_N^{1-2s} \left(\int_0^{1/2} r^{N-2} dr + \int_{1/2}^{1/y_N} r^{-2s} dr \right) \\
	&\leq C\,(1+y_N^{1-2s}).
	\end{aligned}
	\end{equation*}
	
	Putting together \eqref{eq:Est1J1Neumann2} and \eqref{eq:Est2J1Neumann2}, we find
	\[
	J_1 \leq C (1 + |\log d(x)|),
	\]
	for some constant $C > 0$ depending on $N$, $s$ and the Lipschitz constant of $\Omega$.

	To estimate $J_2$, we notice that
	\[
	J_2 \leq 2\int_{\Omega \setminus B_1(x)} K_\Omega(x,y) dy \leq C\int_{\Omega \setminus B_1(x)} \frac{1+\log^-\left(\frac{d_{x,y}}{|x-y|}\right)}{|x-y|^{N+2s}} dy,
	\]
	for some universal $C > 0$ and that the kernel is singular only near $\partial\Omega$, due to the fact that $|x-y| \geq 1$. Moreover, $y \to |\log d(y)|d(y)^{-N-2s}$ is integrable for $|y|$ large and thus repeating the arguments which have led to \eqref{eq:Est1J1Neumann2} and \eqref{eq:Est2J1Neumann2}, we find
	\[
	J_2 \leq C (1 + |\log d(x)|),
	\]
	for some $C >0$ depending on $N$, $s$ and the Lipschitz constant of $\Omega$, as wanted.
\end{proof}

Using the previous lemma, and a Moser-type iteration, we deduce the following.

\begin{cor}\label{Cor:BoundInfNegLp1}
Let $\Omega \subset \RR^N$ be a domain, $R > 0$, $x_0 \in \Omega$ and $\beta>1$. Moreover, assume that $\partial \Omega \cap B_R(x_0)$ is a Lipschitz graph.
Let $K_\Omega$ be of the form either \eqref{K1} or \eqref{K2}.
Let $u$ satisfy
	\[
	\begin{cases}
	L_\Omega u \geq 0 \quad &\text{ in } D_{R}(x_0) \\
	u > 0       \quad &\text{ in } D_{R}(x_0),
	\end{cases}
	\]
	with Neumann conditions on $\partial\Omega \cap B_{R}(x_0)$ in the sense of Definition \ref{Def:WeakSolNeumann}.
	
	Then, there exists a constant $C > 0$ depending only on the Lipschitz constant of $\partial \Omega$, $N$, $s$, and $\beta > 0$, such that
	\begin{equation}\label{eq:BoundBelowLocal1}
		\einf_{x \in D_{R/2}(x_0)} u(x) \geq C \left( \fint_{D_R(x_0)} u(x)^{-\beta} dx \right)^{-1/\beta}.
	\end{equation}
\end{cor}

\begin{proof}  By scaling, we can assume $x_0 = 0$ and $R =1$.
	
	Let $\{r_k\}_{k \in \NN}$ be a decreasing sequence satisfying $r_0 = 1$ and $r_k \to 1/2$ as $k \to +\infty$. For a given $\beta > 1$, we apply the Sobolev inequality to \eqref{eq:LocalRevPoin1} to obtain
	\[
	\left( \int_{D_{r_{k+1}}} u(x)^{(1-\beta)\gamma} dx \right)^{1/\gamma} \leq \frac{C \beta^2}{(r_k-r_{k+1})^{2s}}  \int_{D_{r_k}} u(x)^{1-\beta} \left( 1 +  \left|\log \frac{d(x)}{r_k-r_{k+1}}\right| \right) dx,
	\]
	where $\gamma := 2_s^{\ast}/2 > 1$ and where $C$ depends only on the Lipschitz constant of $\partial \Omega$, $N$, and $s$.
	
	Let $\varepsilon \in (0,\gamma-1)$ and apply H\"older inequality to the right hand side:
	\[
	\begin{aligned}
	\int_{D_{r_k}} u(x)^{1-\beta} \left( 1 +  \left|\log \frac{d(x)}{r_k-r_{k+1}}\right| \right) dx &\leq \left( \int_{D_{r_k}} u(x)^{(1-\beta)(1+\varepsilon)} dx \right)^{\frac{1}{1+\varepsilon}} \\
	&\quad \times \left( \int_{D_{r_k}} \left( 1 +  \left|\log \frac{d(x)}{r_k-r_{k+1}}\right| \right)^{\frac{1+\varepsilon}{\varepsilon}} dx \right)^{\frac{\varepsilon}{1+\varepsilon}} \\
	&= C_k \left( \int_{D_{r_k}} u(x)^{(1-\beta)(1+\varepsilon)} dx \right)^{\frac{1}{1+\varepsilon}},
	\end{aligned}
	\]
	where
	\[
	C_k := \left( \int_{D_{r_k}} \left( 1 +  \left|\log \frac{d(x)}{r_k-r_{k+1}}\right| \right)^{\frac{1+\varepsilon}{\varepsilon}} dx \right)^{\frac{\varepsilon}{1+\varepsilon}}.
	\]
	Notice that, since $r_k-r_{k+1} \to 0$ and $r_k \to 1/2$, we have
	\[
	\begin{aligned}
	C_k &\leq C \left( \int_{D_{1/2}} \left( 1 +  \left|\log \frac{d(x)}{r_k-r_{k+1}}\right| \right)^{\frac{1+\varepsilon}{\varepsilon}} dx \right)^{\frac{\varepsilon}{1+\varepsilon}} \\
	&\leq C \left[ \left( \int_{D_{1/2}} \left( 1 +  \left|\log d(x)\right| \right)^{\frac{1+\varepsilon}{\varepsilon}} dx \right)^{\frac{\varepsilon}{1+\varepsilon}} + |\log (r_k-r_{k+1})|  \right] \leq C |\log (r_k-r_{k+1})|,
	\end{aligned}
	\]
	for some $C$. Further, for any fixed $\alpha \in (0,1)$,
	\[
	|\log (r_k-r_{k+1})| \leq C_\alpha (r_k-r_{k+1})^{-\alpha},
	\]
	for some $C_\alpha$, and so
	\begin{equation}\label{eq:EstCKMoserLocal1}
		C_k \leq C_\alpha (r_k-r_{k+1})^{-\alpha},
	\end{equation}
	for some $C_\alpha$. Now, changing $1-\beta \to - \beta$, we easily deduce
	\[
	\left( \int_{D_{r_{k+1}}} u(x)^{-\beta \gamma} dx \right)^{-\frac{1}{\beta\gamma}} \geq \left[ \frac{(r_k-r_{k+1})^{2s}}{C C_k (1+\beta)^2} \right]^{\frac{1}{\beta}} \left( \int_{D_{r_k}} u(x)^{-\beta(1+\varepsilon)} dx \right)^{-\frac{1}{\beta(1+\varepsilon)}}.
	\]
	Further, setting $v := u^{1+\varepsilon}$, $\sigma := \frac{\gamma}{1+\varepsilon} > 1$, and using \eqref{eq:EstCKMoserLocal1}, it follows
	\begin{equation}\label{eq:ReversedMoserLocal1}
		\left( \int_{D_{r_{k+1}}} v(x)^{-\beta \sigma} dx \right)^{-\frac{1}{\beta\sigma}} \geq \left[ \frac{(r_k-r_{k+1})^{2s+\alpha}}{C (1+\beta)^2} \right]^{\frac{1+\varepsilon}{\beta}} \left( \int_{D_{r_k}} v(x)^{-\beta} dx \right)^{-\frac{1}{\beta}},
	\end{equation}
	for some $C$. Thus, given $\beta_0 > 0$, we define $\beta_k := \beta_0\sigma^k$, $k \geq 1$.
	Iterating \eqref{eq:ReversedMoserLocal1} with $\beta = \beta_0$, we obtain
	\begin{equation}\label{eq:MoserIterScheme1}
		\begin{aligned}
			\|v\|_{L^{-\beta_k}(D_{r_k})} &\geq \prod_{j=0}^{k-1} \left[ \frac{(r_j-r_{j+1})^{2s+\alpha}}{C (1+\beta_j)^2} \right]^{\frac{1}{\beta_j}} \|v\|_{L^{-\beta_0}(D_{r_0})} \\
			& = \prod_{j=0}^{k-1} \left[\frac{(r_j-r_{j+1})^{2s+\alpha}}{ C (1+\beta_0\sigma^j)^2} \right]^{\frac{1}{\beta_0\sigma^j}} \|v\|_{L^{-\beta_0}(D_{r_0})},
		\end{aligned}
	\end{equation}
	up to changing the constant $C > 0$, independently of $k \in \mathbb{N}$. Now, we notice that
	\[
	\prod_{j=0}^{k-1} \left[ \frac{(r_j-r_{j+1})^{2s+\alpha}}{ C (1+\beta_0\sigma^j)^2} \right]^{\frac{1}{\beta_0\sigma^j}} = \exp \left\{ \frac{1}{\beta_0} \sum_{j = 0}^{k-1} \frac{1}{\sigma^j} \log \left[ \frac{(r_j-r_{j+1})^{2s+\alpha}}{ C (1+\beta_0\sigma^j)^2} \right] \right\},
	\]
	for all $k \geq 1$, and so, choosing $r_j$ such that $(r_j - r_{j+1})^{2s+\alpha} = C\beta_0^2 \sigma^{-2j}$ for $j \in \NN$ large enough, we obtain
	\[
	\sum_{j = 0}^\infty \frac{1}{\sigma^j} \log \left[ \frac{(r_j-r_{j+1})^{2s+\alpha}}{C (1+\beta_0\sigma^j)^2} \right] \geq  - C\sum_{j = 0}^\infty \frac{j}{\sigma^j} > -\infty.
	\]
	Consequently, we can pass to the limit in \eqref{eq:MoserIterScheme1} and deduce \eqref{eq:BoundBelowLocal1}, thanks to the fact that $\|v\|_{L^{-\beta_k}(D_{r_k})}  \to \einf_{x \in D_{R/2}(x_0)} v(x)$ as $k \to +\infty$ and $v = u^{1+\varepsilon}$.
\end{proof}

Combining Lemma \ref{Lemma:LpLminuspBound1} and Corollary \ref{Cor:BoundInfNegLp1}, we finally deduce the following.

\begin{thm}\label{Theorem:WeakHarnack01}
Let $\Omega \subset \RR^N$ be a domain, $R > 0$, $x_0 \in \Omega$ and $\beta>1$. Assume that $\partial \Omega \cap B_{3R}(x_0)$ is a Lipschitz graph.
Let $K_\Omega$ be of the form either \eqref{K1} or \eqref{K2}.
Assume that $u$ satisfies
	\[
	\begin{cases}
	L_\Omega u \geq 0 \quad &\text{ in } D_{2R}(x_0) \\
	u > 0       \quad &\text{ in } D_{2R}(x_0),
	\end{cases}
	\]
	with Neumann conditions on $\partial\Omega \cap B_{2R}(x_0)$ in the sense of Definition \ref{Def:WeakSolNeumann}.
	
	Then for any $c>0$ and $\vartheta > 1$, there exist $\kappa > 0$ and $\gamma \in (0,2s)$  depending only on the Lipschitz constant of $\partial \Omega$, $N$, $s$, $c$ and $\vartheta$, such that if
	\begin{equation}\label{eq:AssBoundBelow1}
		\begin{cases}
			u(x) \geq c \left[ 1 - \left( \vartheta \frac{|x-x_0|}{R} \right)^\gamma \right]  \quad &\text{ for a.e. } x \in \Omega \setminus B_{2R}(x_0) \\
			\frac{|\{ u \geq 1 \} \cap D_{2R}(x_0)|}{|D_{2R}(x_0)|} \geq \frac{1}{2},
		\end{cases}
	\end{equation}
	then
	\begin{equation}\label{eq:BoundBelow1}
		\einf_{x \in D_{R/4}(x_0)} u(x) \geq \kappa.
	\end{equation}
\end{thm}
\begin{proof}
By scaling, it is enough to consider the case $R = 1$ and $x_0 = 0$.

 First, since $\partial \Omega \cap B_3$ is a Lipschitz graph, and $0\in\Omega$, we can show that there exists $\omega\in (0,1/2)$ such that
	$$\frac{|D_{1/2}|}{|D_2|} \leq \omega. $$
Indeed, this follows from the pointwise inequality $\sqrt{4-x^2} > 3 \sqrt{1/4-x^2}$, which shows that we can take $\omega = 2/5 <1/2$.

	Now we claim that the second condition in \eqref{eq:AssBoundBelow1} guarantees the existence of $r_0 \in (1/2,2)$ such that
	$$ \frac{|\{ u \geq 1 \} \cap D_{r_0}|}{|D_2|} \geq \frac{1+2\omega}{4}, \qquad \frac{|\{ u \geq 1 \} \cap (D_2 \setminus  D_{r_0})|}{|D_2|} \geq \frac{1-2\omega}{4}.$$
	Let us define the functions
	$$ h(\rho) :=\frac{|\{ u \geq 1 \} \cap D_\rho|}{|D_2|}, \qquad \tilde{h}(\rho):=\frac{|\{ u \geq 1 \} \cap (D_2 \setminus D_\rho)|}{|D_2|}.$$
	It is clear that they are both continuous. Moreover, the first one is nondecreasing and satisfies $h(1/2)\leq \omega$ and $h(2)\geq 1/2$ by hypothesis. This means that there exists $r_0\in (1/2,2)$ such that
	$h(r_0) = (1/2+\omega)/2 = (1+2\omega)/4$. If we now use that $h(\rho)+\tilde{h}(\rho) \geq 1/2$, the claim easily follows.
	
	Applying Corollary \ref{Cor:BoundInfNegLp1} (with $R=r_0$), we obtain that for any $\beta > 0$
	\begin{equation}\label{eq:MoserIneBeta}
		\einf_{x \in D_{r_0/2}} u(x) \geq C \left( \fint_{D_{r_0}} u(x)^{-\beta} dx \right)^{-1/\beta},
	\end{equation}
	for some constant $C > 0$ depending only on the Lipschitz constant of $\partial \Omega$, $N$, $s$, and $\beta$. Now, by Lemma \ref{Lemma:TechAssumNonlocal1} with $R=2$, $\delta_0 = (1-2\omega)/4$ and $r=r_0$, there is $\gamma \in (0,2s)$ depending only on the Lipschitz constant of $\partial \Omega$, $N$, $s$, $c$ and $\vartheta$ such that
	\[
	\int_{\Omega \setminus B_{r_0}} u(x) K_\Omega(x,y)dx \geq 0 \quad \text{ for a.e. } y \in D_{r_0}.
	\]
	On the other hand, by Lemma \ref{Lemma:LpLminuspBound1} (with $R=r_0$), there exists $\beta_0 \in (0,1)$ depending only on the Lipschitz constant of $\partial \Omega$, $N$, and $s$ such that
	\[
	\left( \fint_{D_{r_0}} u(x)^{\beta_0} dx \right)^{1/\beta_0} \leq C \left(\fint_{D_{r_0}} u(x)^{-\beta_0} dx \right)^{-1/\beta_0},
	\]
	and thus, choosing $\beta = \beta_0$ in \eqref{eq:MoserIneBeta}, it follows
	\[
	\begin{aligned}
	\einf_{x \in D_{r_0/2}} u(x) &\geq C \left( \fint_{D_{r_0}} u(x)^{\beta_0} dx \right)^{1/\beta_0} \geq C \left( \frac{1}{|D_{r_0}|} \int_{D_{r_0}\cap\{u \geq 1\}} u(x)^{\beta_0} dx \right)^{1/\beta_0} \\
	& \geq C \left( \frac{|\{u \geq 1\}\cap D_{r_0}|}{|D_{r_0}|} \right)^{1/\beta_0} \geq C \left( \frac{|\{u \geq 1\}\cap D_{r_0}|}{|D_2|} \right)^{1/\beta_0} \\
&\geq C \left( \frac{1+2\omega}{4} \right)^{1/\beta_0} := \kappa.
	\end{aligned}
	\]
	Since $r_0 \geq 1/2$, the thesis follows.
\end{proof}

As a first consequence, we can prove a version of the above theorem that allows a right hand side $f$.

\begin{thm}\label{Theorem:WeakHarnackf1} (Weak Harnack inequality)  Let $\Omega \subset \RR^N$ be a domain, $R > 0$, $x_0 \in \Omega$ and $f \in L^{q}(D_{2R}(x_0))$ with $q>\frac{N}{2s}$.  Assume that $\partial \Omega \cap B_{3R}(x_0)$ is a Lipschitz graph.
Let $K_\Omega$ be of the form either \eqref{K1} or \eqref{K2}.
Assume that $u$ satisfies
	\[
	\begin{cases}
	L_\Omega u \geq f \quad &\text{ in } D_{2R}(x_0) \\
	u > 0       \quad &\text{ in } D_{2R}(x_0),
	\end{cases}
	\]
	with Neumann conditions on $\partial\Omega \cap B_{2R}(x_0)$, in the sense of Definition \ref{Def:WeakSolNeumann}.

	Then for any $c>0$ and $\vartheta > 1$, there exist $\kappa_0 > 0$, $\kappa > 0$ and $\gamma \in (0,2s)$ depending only on the Lipschitz constant of $\partial \Omega$, $N$, $s$, $c$ and $\vartheta$, such that if \eqref{eq:AssBoundBelow1} holds, then
	\begin{equation}\label{eq:BoundBelowWithf1}
	\einf_{x \in D_{R/4}(x_0)} u(x) + \kappa_0 \,R^{2s-\frac{N}{q}}\, \|f\|_{L^{q}(D_{2R}(x_0))} \geq \kappa.
	\end{equation}
\end{thm}
\begin{proof}
We assume $R = 1$, $x_0 = 0$. Let us consider the function $w := u + v$, where $v$ satisfies \eqref{eq:ProbLinfLqestimate} (with $R=1$ and $x_0 = 0$). Then,  $w$ satisfies
	\[
	\begin{cases}
	L_\Omega w \geq 0 \quad &\text{ in } D_{2} \\
	w > 0       \quad &\text{ in } D_{2},
	\end{cases}
	\]
	with Neumann conditions on $\partial\Omega \cap B_{2}$ in the sense of Definition \ref{Def:WeakSolNeumann}. Notice that $w \geq u$ in $\Omega$ and thus it satisfies the assumptions in \eqref{eq:AssBoundBelow1}. Consequently, we can apply Theorem~\ref{Theorem:WeakHarnack01} to the function $w$ and, since $v \leq \kappa_0 \, ||f||_{L^q(D_{1})}$ in $D_{1/2}$ (by Lemma \ref{lemma: SolutionDirichletNeumann}), we deduce
	\[
	\einf_{x \in D_{1/4}} u(x) + \kappa_0 \, \|f\|_{L^{q}(D_1)} \geq \einf_{x \in D_{1/4}} w(x) \geq \kappa,
	\]
	which proves \eqref{eq:BoundBelowWithf1}.
\end{proof}

We finally use the previous weak Harnack inequality to deduce the H\"older regularity of solutions.

\begin{proof}[Proof of Theorem \ref{Theorem:HolderRegularity1}] The result follows by iterating the previous weak Harnack inequality, with an argument similar to those in \cite{K,2006Silvestre:art}.
By scaling and a covering argument as in \cite[Remark~2.13]{FR-book}, it is sufficient to assume that $u$ is a weak bounded solution to
	\[
	L_\Omega u = f \quad \text{ in } D_3(x_0),
	\]
	with Neumann conditions on $\partial\Omega \cap B_3(x_0)$ (in the sense of Definition \ref{Def:WeakSolNeumann}) and prove
	\begin{equation*}
	|u(x) - u(y)| \leq C |x - y|^\alpha \left[ \|u\|_{L^{\infty}(\Omega)} + \|f\|_{L^{q}(D_{3}(x_0))}\right]
	\end{equation*}
	for a.e. $x,y \in D_{1/2}(x_0)$.

	\emph{Step 1.} Let us take $\vartheta = 4$, $c = 2$, $\kappa \in (0,1)$, $\gamma \in (0,2s)$ and $\kappa_0 > 0$ as in Theorem \ref{Theorem:WeakHarnackf1} (depending only on the Lipschitz constant of $\partial \Omega$, $N$, $s$, and $q$). We set $\overline{\kappa} := \kappa/2$.

	Given any $z_0 \in D_1(x_0)$, we construct a non-decreasing sequence $(m_n)_{n \in \ZZ}$ and a non-increasing sequence $(M_n)_{n \in \ZZ}$ such that
	\begin{equation}\label{eq:AuxSeqHolder}
	\begin{aligned}
	&m_n \leq u(y) \leq M_n \quad \text{ for a.e. } y \in D_{\vartheta^{-n}}(z_0) \\
	&M_n - m_n = K \vartheta^{-n\alpha},
	\end{aligned}
	\end{equation}
	for all $n \in \ZZ$, some $\alpha \in (0,1)$ and $K > 0$ to be determined (independently of $z_0$ and $x_0$). We choose
	\begin{equation}\label{eq:ChoiceEps0}
	0 < \varepsilon_0 \leq \min \left\{ \frac{1}{2},\frac{\kappa}{4\kappa_0} \right\}
	\end{equation}
	and
	\[
	M_0 := \|u\|_{L^{\infty}(\Omega)} + \frac{1}{\varepsilon_0} \|f\|_{L^{q}(D_3(x_0))}, \qquad m_0 := -\|u\|_{L^{\infty}(\Omega)},
	\]
	so that
	\[
	K := M_0 - m_0 =  2\|u\|_{L^{\infty}(\Omega)} + \frac{1}{\varepsilon_0}\|f\|_{L^{q}(D_3(x_0))}.
	\]
	Now, we assume that \eqref{eq:AuxSeqHolder} holds and show how \eqref{eq:HolderEstimate1} follows. Since $u$ is bounded, whenever $x,y \in D_1(x_0)$ satisfy $|x - y| \geq 1$, \eqref{eq:HolderEstimate1} follows with $C = 2$ and any $\alpha \in (0,1)$.

	Thus it is enough to check the validity of \eqref{eq:HolderEstimate1} when $x \not= y$ and $|x - y| < 1$. In such case, we take $x = z_0$ and consider $n \in \NN$ (depending on $y$) such that
	\[
	\vartheta^{-(n+1)} \leq |x - y| < \vartheta^{-n}.
	\]
	Consequently,
	\[
	\begin{aligned}
	|u(x) - u(y)| &\leq \osc_{B_{\vartheta^{-n}}(x)} u \leq M_n - m_n = K \vartheta^{-n\alpha} \leq K \vartheta^\alpha |x-y|^\alpha \\
	& \leq \frac{\vartheta^\alpha}{\varepsilon_0} |x-y|^\alpha \left[ \|u\|_{L^{\infty}(\Omega)} + \|f\|_{L^{q}(D_3(x_0))}\right],
	\end{aligned}
	\]
	which is exactly \eqref{eq:HolderEstimate1} with $C = \vartheta^\alpha/\varepsilon_0$. Using the arbitrariness of $x,y \in D_1(x_0)$ with $|x - y| < 1$ and $x\not=y$, the estimate \eqref{eq:HolderEstimate1} follows.

	\emph{Step 2.} Notice that, since $u$ is bounded in $\Omega$, the choice of $K$ guarantees that \eqref{eq:AuxSeqHolder} hold true for $n = 0$ and, moreover, setting $M_n = M_0$ and $m_n = m_0$ for all negative integers $n$, \eqref{eq:AuxSeqHolder} hold true for any $n \in \ZZ$, $n < 0$.

	\emph{Step 3.} We construct the sequences $(m_n)_{n\in\NN}$ and $(M_n)_{n\in\NN}$ by induction on $n \in \NN$. So, we assume that there exists $k \geq 1$ such that \eqref{eq:AuxSeqHolder} hold for all $n \leq k-1$, and we show how to choose $m_k$ and $M_k$ such that \eqref{eq:AuxSeqHolder} hold for $n = k$.

	We define
	\begin{equation}\label{eq:DefAlpha}
	\alpha := \min \left\{\gamma, \ln \left(\frac{2}{2-\overline{\kappa}}\right) / \ln \vartheta \right\},
	\end{equation}
	and we consider the function
	\[
	v(x) := \left( u(\vartheta^{-(k-1)} x + z_0) - \frac{M_{k-1} + m_{k-1}}{2} \right) \frac{2\vartheta^{(k-1)\alpha}}{K}.
	\]
	Notice that, in view of \eqref{eq:AuxSeqHolder}, we have
	\[
	|v| \leq 1 \quad \text{ in } \widetilde{D}_1,
	\]
	where $B_1 = B_1(0)$, $\widetilde{\Omega} := \{x \in \RR^N: \vartheta^{-(k-1)} x + z_0 \in \Omega\}$ and $\widetilde{D}_1 := B_1\cap\widetilde{\Omega}$. Note that since $\widetilde{\Omega}$ is a dilation, its Lipschitz constant does not increase. Now, we divide the proof in two cases. First, we assume
	\begin{equation}\label{eq:Ass1Holder}
	\frac{|\{v \leq 0\} \cap \widetilde{D}_1|}{|\widetilde{D}_1|} \geq \frac{1}{2}.
	\end{equation}
	In order to apply Theorem \ref{Theorem:WeakHarnackf1}, we study the decaying of $v$ in $\widetilde{\Omega} \setminus B_1$. So, for any $y \in \widetilde{\Omega} \setminus B_1$ we have $|y| \geq 1$ and thus there is $j \in \NN$, $j \geq 1$ (depending on $y$) such that
	\[
	\vartheta^{j-1} \leq |y| < \vartheta^j.
	\]
	Using that $(m_n)_{n \in \NN}$ is non-decreasing, the fact that $y \in B_{\vartheta^j}$ and \eqref{eq:AuxSeqHolder}, we obtain
	\[
	\begin{aligned}
	v(y) &= \frac{2\vartheta^{(k-1)\alpha}}{K} \left( u(\vartheta^{-(k-1)} y + z_0) - \frac{M_{k-1} + m_{k-1}}{2} \right) \\
	&\leq \frac{2\vartheta^{(k-1)\alpha}}{K} \left( M_{k-j-1} - m_{k-j-1} + m_{k-j-1} - \frac{M_{k-1} + m_{k-1}}{2} \right) \\
	&\leq \frac{2\vartheta^{(k-1)\alpha}}{K} \left( M_{k-j-1} - m_{k-j-1} - \frac{M_{k-1} - m_{k-1}}{2} \right) \\
	&= \frac{2\vartheta^{(k-1)\alpha}}{K} \left( K\vartheta^{-(k-j-1)\alpha} - \frac{K}{2}\vartheta^{-(k-1)\alpha} \right) = 2 \vartheta^{j\alpha} - 1 \leq 2 \vartheta^\alpha |y|^\alpha - 1,
	\end{aligned}
	\]
	which, setting $w := 1 - v$, is equivalent to
	\[
	w(y) \geq 2 \left[ 1 - \left(\vartheta |y| \right)^\alpha \right] \quad \text{ for a.e. } y \in \widetilde{\Omega} \setminus B_1.
	\]
	Furthermore, $w$ is a weak solution to
	\[
	 L_\Omega w = -\frac{2}{K} \vartheta^{(\alpha-2s) (k-1)} f \quad \text{ in }  \widetilde{D}_2,
	\]
	and so, thanks to assumption \eqref{eq:Ass1Holder} and the fact that $\alpha \leq \gamma$ (see \eqref{eq:DefAlpha}), we can apply Theorem \ref{Theorem:WeakHarnackf1} (with $R = 1$) to deduce
	\[
	\einf_{x \in \widetilde{D}_{\vartheta^{-1}}} w(x) + \frac{2 \kappa_0}{K} \vartheta^{(\alpha-2s)(k-1)} \|f\|_{L^{q}(\widetilde{D}_2)} \geq \kappa,
	\]
which implies
\[
	v(x) \leq 1 - \kappa + \frac{2 \kappa_0}{K} \vartheta^{(\alpha-2s)(k-1)} \|f\|_{L^{q}(\widetilde{D}_2)} \quad \text{ for a.e. } x \in \widetilde{D}_{\vartheta^{-1}}.
	\]
Notice that, using the definition of $K$ and that $\alpha \leq 2s$ (cfr. with \eqref{eq:DefAlpha}) and $\vartheta > 1$, we have
\[
	\frac{2 \kappa_0}{K} \vartheta^{(\alpha-2s)(k-1)} \|f\|_{L^{q}(\widetilde{D}_2)} \leq 2 \kappa_0 \frac{\varepsilon_0 \|f\|_{L^{q}(\widetilde{D}_2)}}{\varepsilon_0\|u\|_{L^{\infty}(\widetilde{D}_2)} + \|f\|_{L^{q}(\widetilde{D}_2)}}  \leq 2\kappa_0 \varepsilon_0 \leq \frac{\kappa}{2},
	\]
thanks to the choice of $\varepsilon_0 > 0$ in \eqref{eq:ChoiceEps0}. Consequently,
\[
	v(x) \leq 1 - \frac{\kappa}{2} := 1 - \overline{\kappa}  \quad \text{ for a.e. } x \in \widetilde{D}_{\vartheta^{-1}}.
	\]
So, using the definition of $v$ and the above inequality, we obtain
	\[
	\begin{aligned}
	u(x) & \leq \frac{1-\overline{\kappa}}{2} K \vartheta^{-(k-1)\alpha} + \frac{M_{k-1} + m_{k-1}}{2} =  \frac{1-\overline{\kappa}}{2} (M_{k-1} - m_{k-1}) + \frac{M_{k-1} + m_{k-1}}{2} \\
	& = m_{k-1} + \left( 1 - \frac{\overline{\kappa}}{2} \right)(M_{k-1} - m_{k-1})  \\
	\end{aligned}
	\]
for a.e. $x \in D_{\vartheta^{-k}}(z_0)$. Finally, using \eqref{eq:DefAlpha}, we have that $1 - \frac{\overline{\kappa}}{2} \leq \vartheta^{-\alpha}$, and so from the definition of $K$, we deduce
\[
u(x) \leq m_{k-1} + K \vartheta^{-k\alpha} \quad \text{ for a.e. } x \in D_{\vartheta^{-k}}(z_0).
\]
Choosing $m_k := m_{k-1}$ and $M_k := m_{k-1} + K \vartheta^{-k\alpha}$, it follows that \eqref{eq:AuxSeqHolder} is satisfied for $n = k$ and we complete the proof of the first case.

Finally, if \eqref{eq:Ass1Holder} is not satisfied, it is sufficient to notice that it holds for $\tilde{v} := - v$ and repeat the above procedure working with $\tilde{v}$.
\end{proof}

\section{A Neumann Liouville theorem in the half-space}
\label{sec5}

The goal of this section is to prove the following Liouville-type theorem in a half-space with nonlocal Neumann boundary conditions.

\begin{thm} \label{Thm:LiouvilleND2}
	Let $\Omega=\RR^N_+ = \{x_N>0\}$, and $s\in (\frac12,1)$. Let $L_\Omega$ and $K_\Omega$ be given by either \eqref{new-operator}-\eqref{new-kernel}-\eqref{new-kernel2}, or \eqref{operator-reg}. Assume $v$ is a weak solution to
	$$L_\Omega v = 0 \quad \text{ in } \mathbb{R}^N_+$$
	with Neumann condition on $\partial \mathbb{R}^N_+=\{x_N=0\}$ (in the sense of Definition~\ref{Def:WeakSolNeumann}).
	Let $\alpha>0$ be given by Theorem \ref{Thm:BoundaryRegBH}, and assume that
		$$ ||v||_{L^\infty(B_R^+)} \leq C_0 (1 + R^{2s-1+\varepsilon}) \quad \textrm{for all}\quad R > 0,$$
	for some $C_0$ and $\varepsilon \in (0,\alpha)$. Then,
	$$ v(x) = a + b\cdot x $$
	for some $a\in \RR$ and $b\in \RR^N$ with $b_N = 0$.
	Moreover, if $2s-1+\epsilon<1$ then $b=0$.
\end{thm}

The proof of this result is not standard and does not follow from classical tools such as even reflection for harmonic functions.
Moreover, the extension problem for the fractional Laplacian is of no use here, and therefore the proof must be different from the Dirichlet case, too.

We stress that, even in 1D, we do not know how to prove a better Liouville theorem (allowing more growth on $v$).
This seems a challenging open problem, which is strongly related to the higher boundary regularity of solutions to \eqref{eq-intro}.

\subsection{1D barriers}

We need sub- and supersolutions for both problems \eqref{eq-intro} and \eqref{eq-intro-reg}.
We start with the following.

\begin{lem} (Supersolution for \eqref{eq-intro} and \eqref{eq-intro-reg})\label{Lemma:1DSupersolution}
	Let $N=1$, $\Omega=(0,\infty)$, and $s\in (\frac12,1)$.
	Let $L_\Omega$ and $K_\Omega$ be given by either \eqref{new-operator}, \eqref{new-kernel}-\eqref{new-kernel2} or \eqref{operator-reg}.
	Given any $r_0 > 0$, let us consider $\eta \in C_0^\infty([0,2r_0))$ satisfying $0 \leq \eta \leq 1$ and $\eta = 1$ in $[0,r_0]$.
	
	Then, there exists $\overline{c} > 0$ (depending only on $r_0$) such that the function
	\[
	\overline{\varphi}(x) := \eta(x)x^{2s-1}
	\]
	satisfies
	$$L_\Omega \overline{\varphi} \geq \overline{c} \ \text{ in } \ (0,r_0).$$
	Moreover, if $L_\Omega$ and $K_\Omega$ are given by \eqref{new-operator}, \eqref{new-kernel}-\eqref{new-kernel2}, a logarithmic improvement can be done. That is,
	$$L_\Omega \overline{\varphi} \geq \overline{c} \left(1+ \log^-\left(\frac{x}{r_0}\right)\right)\ \text{ in } \ (0,r_0).$$
\end{lem}

\begin{proof} We prove the result for $K_\Omega$ of the form \eqref{new-operator}, \eqref{new-kernel}-\eqref{new-kernel2}; the case \eqref{operator-reg} is simpler.
	
	By scaling, we may assume $r_0=1$. Given $x \in (0,1)$ and using the definition of $\overline{\varphi}$, we compute
	\[
	\begin{aligned}
	L_\Omega\overline{\varphi}(x) &= \int_0^\infty \left\{ x^{2s-1} - \eta(y)y^{2s-1} \right\}\,K_\Omega(x,y)\,dy \\
	&= \int_0^\infty \left\{ x^{2s-1} - y^{2s-1} \right\}\,K_\Omega(x,y)\,dy + \int_0^\infty y^{2s-1}(1- \eta(y))\,K_\Omega(x,y)\,dy := I_1 + I_2.
	\end{aligned}
	\]
	Now, by the symmetry and the scaling of the kernel $K_\Omega$ (see Section \ref{sec2}), it is easy to check that $L_\Omega(x^{2s-1}) = 0$ in $\RR_+$ and so $I_1(x) = 0$. On the other hand, we know that $\eta = 1$ in $[0,1]$ while $\eta = 0$ in $[2,\infty)$. Moreover, if we use that $1\leq y-x \leq y$ for all $x < 1 < 2 \leq y$,  it follows
	\[
	\begin{aligned}
	I_2(x) &= \int_{1}^{2} y^{2s-1}(1- \eta(y))\,K_\Omega(x,y)\,dy + \int_{2}^\infty y^{2s-1}(1- \eta(y))\,K_\Omega(x,y)\,dy \\
	&\geq \int_{2}^\infty y^{2s-1}(1- \eta(y))\,K_\Omega(x,y)\,dy = \int_{2}^\infty y^{2s-1}\,K_\Omega(x,y)\,dy \\
	&\geq c \int_{2}^\infty y^{2s-1} \frac{1+\log^-\left(\frac{x}{y-x}\right)}{(y-x)^{1+2s}}\,dy\\
	&= c \int_{2}^{\infty} \frac{y^{2s-1}}{(y-x)^{1+2s}}\,dy + c\int_{2}^{\infty} y^{2s-1}\,\frac{\log\left(\frac{y-x}{x}\right)}{(y-x)^{1+2s}} dy \\
	&\geq c \int_{2}^{\infty} y^{-2} dy + c\int_{2}^{\infty} y^{-2}\, \log\left(\frac{y-x}{x}\right) dy \geq c \int_{2}^{\infty} y^{-2} dy + c\int_{2}^{\infty} y^{-2}\, \log\left(\frac{1}{x}\right) dy \\
&\geq \overline{c} \left(1+ \log^-x\right).
	\end{aligned}
	\]
\end{proof}

We next show the following construction of subsolutions.

\begin{lem}(Subsolution for \eqref{eq-intro} and \eqref{eq-intro-reg})\label{Lemma:1DSubsolution}
	Let $N=1$, $\Omega=(0,\infty)$, and $s\in (\frac12,1)$.
	Let $L_\Omega$ and $K_\Omega$ be given by either \eqref{new-operator}, \eqref{new-kernel}-\eqref{new-kernel2} or \eqref{operator-reg}.
	 Given any $r_0 > 0$, let us consider $\eta \in C_0^\infty([0,2r_0))$ satisfying $0 \leq \eta \leq 1$, $\eta = 1$ in $[0,r_0]$ and $\zeta \in C_0^\infty((r_0,2r_0))$ satisfying $0 \leq \zeta \leq 1$ and $\zeta \not\equiv 0$.
	
	 Then, for any $\underline{c} \geq 0$, there exist $M > 0$ (depending on $\underline{c}$, $s$ and $r_0$) such that the function
	\[
	\underline{\varphi}(x) := \eta(x)x^{2s-1} + M\zeta(x)
	\]
	satisfies
	$$L_\Omega\underline{\varphi} \leq - \underline{c}\ \text{ in } \ (0,r_0).$$
	Moreover, if $L_\Omega$ and $K_\Omega$ are given by \eqref{new-operator}, \eqref{new-kernel}-\eqref{new-kernel2}, a logarithmic improvement can be done. That is,
	$$L_\Omega\underline{\varphi} \leq - \underline{c}\, \left(1+ \log^-\left(\frac{x}{r_0}\right)\right) \ \text{ in } \ (0,r_0).$$
\end{lem}

\begin{proof} We proceed as in the previous lemma, proving the result only in the case $L_\Omega$ and $K_\Omega$ are given by \eqref{new-operator}, \eqref{new-kernel}-\eqref{new-kernel2} and $r_0=1$. Given $x\in (0,1)$ and using the properties of $\zeta$ and the identity $L_\Omega(x^{2s-1}) = 0$ in $\RR_+$, we obtain
	\[
	\begin{aligned}
	L_\Omega\underline{\varphi}(x) &= L_\Omega\overline{\varphi}(x) + M L_\Omega\zeta(x) \\
	&= \int_{1}^{2} y^{2s-1}(1 - \eta(y))\,K_\Omega(x,y)\,dy + \int_{2}^\infty y^{2s-1}\,K_\Omega(x,y)\,dy - M \int_{1}^{2} \zeta(y)\,K_\Omega(x,y)\,dy \\
	&:= I_1(x) + I_2(x) - M\,I_3(x).
	\end{aligned}
	\]
	Now, we consider separately each of the three terms. That is,
	\begin{align*}
	I_1(x) &= \int_{1}^{2} y^{2s-1}(1 - \eta(y))\,K_\Omega(x,y)\,dy \\
	&\leq C \int_{1}^{2} y^{2s-1}(1 - \eta(y))\,\frac{1+\log^-\left(\frac{x}{y-x}\right)}{(y-x)^{1+2s}}\,dy \\
	&= C \int_{1}^{2} \frac{y^{2s-1}(1 - \eta(y))}{(y-x)^{1+2s}}\,dy + C\int_{1}^{2} y^{2s-1}(1 - \eta(y))\,\frac{\log\left(\frac{y-x}{x}\right)}{(y-x)^{1+2s}}\,\chi_{\{y>2x\}}\,dy \\
	&:= I_{11}(x)+I_{12}(x).
	\end{align*}
	On the one hand, we know the existence of two positive constants $\delta$ and $C$, such that
	\[
	1 - \eta(y) \leq C (y - 1)^2 \quad \text{ for all } y \in [1,1 + \delta).
	\]
	This follows from the fact that $\eta'(1) = 0$ and that $\eta''(1)$ is bounded (notice that $\delta$ and $C$ depend only on $\eta''$). Consequently, since $x \in (0,1)$, we have $y-x \geq y - 1$ and, moreover, when $y \in (1+\delta,2)$ we have $y-x \geq \delta$. Thus,
	\[
	\begin{aligned}
	I_{11}(x) &= C\int_{1}^{1 + \delta} \frac{y^{2s-1}(1 - \eta(y))}{(y-x)^{1 + 2s}}dy + C\int_{1 + \delta}^{2} \frac{y^{2s-1}(1 - \eta(y))}{(y-x)^{1 + 2s}}dy \\
	&\leq C\int_{1}^{1 + \delta} y^{2s-1} (y-1)^{1 - 2s} dy + C\,\delta^{-1-2s}\int_{1 + \delta}^{2} y^{2s-1} dy < C < +\infty.
	\end{aligned}
	\]
	On the other hand, taking into account that $y-x \geq y - 1/2 \geq 1/2$ when $x \leq 1/2$, whilst $y-x > x > 1/2$ when $x > 1/2$ and $y > 2x$ we arrive at	
	\begin{align*}
	I_{11}(x) &\leq C\int_{1}^{2} \frac{\log\left(\frac{y-x}{x}\right)}{(y-x)^{1+2s}}\,\chi_{\{y>2x\}}\,dy \leq C \int_{1}^{2} \log\left(\frac{2}{x}\right)\,dy \\
	&\leq C \left( 1+ \log\left(\frac{1}{x}\right) \right) \leq C \left( 1 + \log^- x\right).
	\end{align*}
	Thus, we obtain
	$$ I_1(x) \leq C_1 \left( 1 + \log^- x\right). $$
	
	Next, we proceed with the estimate of the term $I_2$. That is, since $y - x \geq y/2$ when $x\leq 1 <2 \leq y$ we get
	\[
	\begin{aligned}
	I_2(x) &= \int_{2}^{\infty} y^{2s-1}\,K_\Omega(x,y)\,dy \leq C \int_{2}^{\infty} y^{2s-1}\,\frac{1+\log^-\left(\frac{x}{y-x}\right)}{(y-x)^{1+2s}}\,dy \\
	&\leq C \int_{2}^{\infty} y^{-2} \left( 1+\log\left(\frac{y-x}{x}\right)\right) dy \\
	&\leq C \int_{2}^{\infty} y^{-2}\left( 1+\log\left(\frac{y}{x}\right)\right)\,dy \leq C \left( 1 - \log x \right) \int_{2}^{\infty} y^{-2}\left( 1+\log y\right)\,dy \\ &\leq C_2 \left( 1 + \log^- x\right).
	\end{aligned}
	\]

Finally, we consider $I_3$. By using again again that $1/2 \leq y-x \leq 2$ when $2x<y<2$, we arrive at
	\[
	\begin{aligned}
	I_3(x) &= \int_{1}^{2} \zeta(y)\,K_\Omega(x,y)\,dy \geq c \int_{1}^{2} \zeta(y)\,\frac{1+\log^-\left(\frac{x}{y-x}\right)}{(y-x)^{1+2s}}\,dy \\
	&\geq c\int_{1}^{2} \zeta(y)\,\frac{1+\log\left(\frac{y-x}{x}\right)\chi_{\{y>2x\}}}{(y-x)^{1+2s}}\,dy \\
	&\geq c\int_{1}^{2} \zeta(y)\,\left( 1+\log\left(\frac{y-x}{x}\right)\chi_{\{1>2x\}} \right) \,dy \\
	&\geq c\int_{1}^{2} \zeta(y)\,\left( 1+\log\left(\frac{1}{2x}\right)\chi_{\{1>2x\}} \right) \,dy \geq c\left( 1+\log\left(\frac{1}{2x}\right)\chi_{\{1>2x\}} \right) \\
	&\geq C_3 \left(1 + \log^-x\right).
	\end{aligned}
	\]
Therefore, as a consequence of the previous computations, for all $x \in (0,1)$ and all $\underline{c} \geq 0$, we obtain
	\[
	L_\Omega \underline{\varphi}(x) = I_1 + I_2 - M I_3 \leq  (C_1+C_2- M C_3) \left(1+ \log^-x\right) := - \underline{c} \left(1+ \log^-x\right),
	\]
	if we take $M > 0$ large enough, depending only on $s$ and $\underline{c}$.
\end{proof}

\subsection{A 1D boundary Harnack}

We now prove a boundary Harnack estimate in dimension 1, by using the previous sub/supersolutions and following the general steps from \cite{RS-Dir}.

For any $R > 0$, we define
\[
I_R := (0,R) \quad \text{ and } \quad I^+_R := (R/4,R/2).
\]
The first step is the following.

\begin{lem} \label{Lemma:InfInfBH}
Let $N=1$, $\Omega=(0,\infty)$, $s\in (\frac12,1)$, and $K_0\geq 0$. Assume that either $L_\Omega$ and $K_\Omega$ are given by \eqref{operator-reg} and $u$ satisfies
\[
\begin{cases}
L_\Omega u \geq -K_0 \quad &\text{ in } I_R,  \\
u \geq 0    \quad &\text{ in } \RR_+,
\end{cases}
\]
or $L_\Omega$ and $K_\Omega$ are given by  \eqref{new-operator}-\eqref{new-kernel}-\eqref{new-kernel2} and $u$ satisfies
\[
\begin{cases}
L_\Omega u \geq -K_0 \left[1 + \log^-\left(\frac{x}{R}\right)\right] \quad &\text{ in } I_R,  \\
u \geq 0    \quad &\text{ in } \RR_+.
\end{cases}
\]
Then, there exists $C > 0$ depending only on $s$, such that
\begin{equation}\label{eq:IneqBH1}
\inf_{x \in I_R^+} \frac{u(x)}{x^{2s-1}} \leq C \bigg[ \inf_{x \in I_{R/4}} \frac{u(x)}{x^{2s-1}} + K_0R  \bigg].
\end{equation}
\end{lem}
\begin{proof} We prove the result for $K_\Omega$ of the form \eqref{new-operator}, \eqref{new-kernel}-\eqref{new-kernel2} since the case \eqref{operator-reg} is completely analogous.
	
By scaling properties we may assume $R = 1$. The general case is recovered by applying \eqref{eq:IneqBH1} (with $R = 1$) to the function $u_R(x) := R^{-2s}u(Rx)$, $R > 0$.

\emph{Step 1.} Assume $K_0 = 0$. Let us define
\[
m := \inf_{x \in I_1^+} \frac{u(x)}{x^{2s-1}} \geq 0.
\]
If $m = 0$, the thesis follows immediately. So, assume $m > 0$. In this case, it holds
\[
u(x) \geq m x^{2s-1} \geq m r_0^{2s-1} \quad \text{ in } I_1^+.
\]
Now, for any $\varepsilon > 0$, we define
\[
\varphi(x) := \varepsilon \underline{\varphi}(x),
\]
where $\underline{\varphi}$ is the subsolution constructed in Lemma \ref{Lemma:1DSubsolution} for $r_0=1/4$ and $\underline{c} = 0$, satisfying $L_\Omega \underline{\varphi} \leq 0$ in $I_{1/4}$, and $\supp(\underline{\varphi}) \subset I_{1/2}$. Consequently, $\varphi$ is a subsolution in $I_{1/4}$ for any $\varepsilon > 0$ and, furthermore,
\[
\varphi(x) = \varepsilon [\eta(x)x^{2s-1} + M \zeta(x)] \leq \varepsilon (2^{1-2s} + M) \leq m 4^{1-2s} \leq u(x),
\]
for all $x \in [1/4,1/2)$, whenever $0 < \varepsilon \leq \varepsilon_0 := m 4^{1-2s} / (2^{1-2s} + M)$. Thus, choosing $\varepsilon = \varepsilon_0$ and recalling that $u$ is nonnegative, it follows that $\varphi \leq u$ in $[1/4,+\infty)$ and so applying the comparison principle in $I_{1/4}$ we obtain
\[
\varepsilon_0 x^{2s-1} = \varphi (x) \leq u(x) \quad \text{ in } I_{1/4}.
\]
Taking $C = (2^{1-2s} + M)/4^{1-2s}$ and using the definition of $\varepsilon_0$, it easily follows
\[
m \leq C \inf_{x \in I_{1/4}} \frac{u(x)}{x^{2s-1}},
\]
and the proof in the case $K_0 = 0$ is completed.

\emph{Step 2.} Assume $K_0 > 0$. For any $\kappa_0 > 0$, we define
\[
v(x) := \kappa_0 \overline{\varphi}(x) + u(x) = \kappa_0 x^{2s-1} + u(x) \quad \text{ in } I_1,
\]
where $\overline{\varphi}$ is the supersolution constructed in Lemma \ref{Lemma:1DSupersolution} (with $r_0 =1$ and $r_1 = 2$), satisfying $L_\Omega \overline{\varphi} \geq \overline{c} \left( 1 + \log^-x \right)$ in $I_1$, for some universal constant $\overline{c} > 0$, and $\supp(\underline{\varphi}) \subset I_2$. Thus, choosing $\kappa_0 = K_0/\overline{c}$ and recalling that $\overline{\varphi}$ is nonnegative, it follows
\[
\begin{cases}
L_\Omega v \geq 0 \quad &\text{ in } I_1 \\
v \geq 0    \quad &\text{ in } \RR_+.
\end{cases}
\]
Hence, we can apply \emph{Step 1} to the function $v$ to conclude the existence of a constant $C > 0$ (depending on $s$) such that
\[
\inf_{x \in I_1^+} \frac{v(x)}{x^{2s-1}} \leq C \inf_{x \in I_{1/4}} \frac{v(x)}{x^{2s-1}}.
\]
Finally, \eqref{eq:IneqBH1} follows easily since $v(x) = \kappa_0 x^{2s-1} + u(x)$ in $I_1$. Notice that the constant $C > 0$ changes passing from $v$ to~$u$.
\end{proof}

We will also need the following, which follows from the interior Harnack inequality (see for instance \cite{DKP}).

\begin{lem}  \label{Lemma:InteriorHarnack1D}
	Let $N=1$, $\Omega=(0,\infty)$, and $s\in (\frac12,1)$.
	Let $L_\Omega$ and $K_\Omega$ be given by either \eqref{new-operator}-\eqref{new-kernel}-\eqref{new-kernel2}, or \eqref{operator-reg}.
	Assume that
\[
\begin{cases}
|L_\Omega u| \leq K_0\left( 1 + \log^- \left(\frac{x}{R}\right) \right) \quad &\text{ in } I_R \\
u \geq 0    \quad &\text{ in } \RR_+,
\end{cases}
\]
for some $K_0 \geq 0$. Then there exists $C > 0$ depending only on $s$, such that
\begin{equation}\label{eq:IneqBH2}
\sup_{x \in I_R^+} \frac{u(x)}{x^{2s-1}} \leq C \bigg[ \inf_{x \in I_R^+} \frac{u(x)}{x^{2s-1}} + K_0R  \bigg].
\end{equation}
\end{lem}
\begin{proof} Again, it is enough to prove the case $R=1$.
Inequality \eqref{eq:IneqBH2} easily follows from the interior Harnack inequality (see (2.2)-(2.3) in \cite{DKP})
\[
\sup_{x \in I_1^+} u(x) \leq C \bigg[ \inf_{x \in I_1^+} u(x) + K_0 \bigg],
\]
and using that $x \in (1/4,1/2)$, and that $\log^- x$ is bounded in $[1/8,1]$.
\end{proof}

We can now prove the oscillation decay for the quotient $u/x^{2s-1}$.

\begin{lem} \label{Lemma:OscDecayBH}
	Let $N=1$, $\Omega=(0,\infty)$, $s\in (\frac12,1)$, and $K_0\geq 0$. Assume that either $L_\Omega$ and $K_\Omega$ are given by \eqref{operator-reg} and u satisfies
\[
\begin{cases}
|L_\Omega u| \leq K_0 \quad \text{ in } I_2 \\
u(0) = 0,
\end{cases}
\]
or $L_\Omega$ and $K_\Omega$ are given by  \eqref{new-operator}-\eqref{new-kernel}-\eqref{new-kernel2} and  $u$ satisfies
\[
\begin{cases}
|L_\Omega u| \leq K_0\left( 1 + \log^- x \right) \quad \text{ in } I_2 \\
u(0) = 0.
\end{cases}
\]
Moreover, assume that $u$ satisfies the growth condition
\begin{equation}\label{eq:GrowthCondBH}
|u(y)| \leq c_0 (1 + y^{2s - \varepsilon_0}), \quad \text{ for all } y > 0,
\end{equation}
for some $c_0 > 0$, $\varepsilon_0 > 1$. Then there exist $\alpha \in (0,1)$ and $C > 0$ (depending on $s$, $c_0$ and $\varepsilon_0$), such that
\begin{equation}\label{eq:IneqBH}
\sup_{x \in I_R} \frac{u(x)}{x^{2s-1}} - \inf_{x \in I_R} \frac{u(x)}{x^{2s-1}}\leq C R^\alpha \left[ \|u\|_{L^\infty(I_2)} + K_0 \right] ,
\end{equation}
for all $R \in (0,1]$.
\end{lem}
\begin{proof} As in the previous results we are only proving it in the case $L_\Omega$ and $K_\Omega$ are given by  \eqref{new-operator}-\eqref{new-kernel}-\eqref{new-kernel2}.
	
Let us fix $\vartheta = 4$ and $R=1$. Similar to the approach followed in the proof of Theorem \ref{Theorem:HolderRegularity1}, we construct a non-decreasing sequence $(m_n)_{n \in \NN}$ and a non-increasing sequence $(M_n)_{n \in \NN}$ such that
\begin{equation}\label{eq:AuxSeqBH}
\begin{aligned}
&m_n \leq \frac{u(y)}{y^{2s-1}} \leq M_n \quad \text{ for a.e. } y \in I_{\vartheta^{-n}} \\
&M_n - m_n = K \vartheta^{-n\alpha},
\end{aligned}
\end{equation}
for all $n \in \NN$, some $\alpha \in (0,1)$ and $K > 0$ to be suitably determined. We proceed by induction on $n \in \NN$.

\emph{Step 1.} We prove the case $n = 0$. Let $\eta \in C_0^\infty([0,2))$ satisfying $0 \leq \eta \leq 1$ and $\eta = 1$ in $[0,1]$ and define
\[
v(x) = \eta(x)u(x), \quad x \geq 0.
\]
Notice that for any $x \in I_1 = (0,1)$, we have $v(x) = u(x)$ and, furthermore,
\[
\begin{aligned}
|L_\Omega v(x)| &\leq |L_\Omega u(x)| + \int_1^2 \frac{|u(y)|[1 - \eta(y)]}{(y-x)^{1+2s}} dy + \int_2^{\infty} \frac{|u(y)|}{(y-x)^{1+2s}} dy \\
&\leq K_0(1 + \log^- x) + K_1(1 + \log^- x) + K_2(1 + \log^- x) := \overline{K}_0(1 + \log^- x),
\end{aligned}
\]
where $\overline{K_0}$ depends only on $c_0 > 0$, $\varepsilon_0 > 1$ and $s$. The above bounds follow by using that $x \in (0,1)$, $y > 1$ (and so $y - x > y - 1$), the regularity properties of $\eta$ and \eqref{eq:GrowthCondBH}.

Now, let $\overline{\varphi}$ be the supersolution constructed in Lemma \ref{Lemma:1DSupersolution} (with $r_0 = 1$, $r_1 = 2$) satisfying $L_\Omega \overline{\varphi} \geq \overline{c} (1 + \log^- x)$ in $I_1$, and let $\psi(x) := A \overline{\varphi}(x)$, $A > 0$. Since, $v$ is bounded and has support contained in $I_2$, we can choose $A$ large enough (for instance, $A \geq \max\{\|u\|_{L^\infty(I_2)}, \overline{K}_0/\overline{c} \}$) so that
\begin{equation}\label{eq:Comp1BoundHar1D}
\begin{aligned}
&\psi \geq v \quad \text{ in } [1,\infty), \\
&L_\Omega\psi \geq A\overline{c} (1 + \log^- x) \geq \overline{K}_0 (1 + \log^- x) \geq L_\Omega v \quad \text{ in } I_1,
\end{aligned}
\end{equation}
and so, recalling that $\psi(0) = v(0) = 0$, it follows $\psi \geq v$ in $I_1$ by the maximum principle. In particular, $u(x) \leq A x^{2s-1}$ for all $x \in I_1$. Notice that the function $\varphi = - \psi$ works as a subsolution in $I_1$ with $\varphi \leq -v$ in $[1,\infty)$ and so $|u(x)| \leq A x^{2s-1}$ for all $x \in I_1$.

Thus we can choose $M_0 = A$, $m_0 = -A$ and $K = M_0 - m_0 = 2A$. We anticipate that in the second part of the proof we will ask $K > 3C K_0$ (see \eqref{eq:ChoiceAlphaKBH}), where $C > 0$ is the constant appearing in Lemma \ref{Lemma:InfInfBH} and Lemma \ref{Lemma:InteriorHarnack1D}. To guarantee this, it is enough to choose
\begin{equation}\label{eq:ChoiceKBH}
K = 2A, \quad A = C_0 \left( \|u\|_{L^\infty(I_2)} + K_0 \right), \quad C_0 > \max\{1,3C/2, \overline{K}_0/(\overline{c} K_0)\}.
\end{equation}
Notice that this choice guarantees $A \geq \max\{\|u\|_{L^\infty(I_2)}, \overline{K}_0/\overline{c} \}$ and thus \eqref{eq:Comp1BoundHar1D} is justified.

\emph{Step 2.} We assume that \eqref{eq:AuxSeqBH} hold for all $n \leq k$ and we prove the existence of $m_{k+1}$ and $M_{k+1}$ satifying \eqref{eq:AuxSeqBH}, too. Define
\[
u_k(x) := u(x) - m_k x^{2s-1},
\]
and write $u_k = u_k^+ - u_k^-$. Notice that in view of \eqref{eq:AuxSeqBH} we have
\[
u_k^+ = u_k \quad \text{ in } I_{\vartheta^{-k}}.
\]
Using the monotonicity of $(m_k)_{k\in\NN}$ and $(M_k)_{k\in\NN}$, we easily deduce that given $x \in I_{\vartheta^{-j}}$, it satisfies
\[
\begin{aligned}
u_k(x) &= u(x) - m_k x^{2s-1} \geq (m_j - m_k)x^{2s-1} \geq (m_j -M_j + M_k - m_k)x^{2s-1} \\
& = K(-\vartheta^{-j\alpha} + \vartheta^{-k\alpha}) x^{2s-1} \geq - K \vartheta^{-j(2s-1)}(\vartheta^{-j\alpha} - \vartheta^{-k\alpha}),
\end{aligned}
\]
for all $j \leq k$. Now, for any $x > \vartheta^{-k}$, there is $j \leq k-1$ such that $\vartheta^{-j-1} < x \leq \vartheta^{-j}$, and thus, if $x \in I_{\vartheta^{-j}} \setminus I_{\vartheta^{-k}}$, we have
\begin{equation}\label{eq:BehUkOutisdeIkBH}
\begin{aligned}
u_k(x) &\geq - K \vartheta^{-j(2s-1)}(\vartheta^{-j\alpha} - \vartheta^{-k\alpha}) = - K \frac{\vartheta^{-j(2s-1)}}{\vartheta^{-k(2s-1)}} \vartheta^{-k(2s-1+\alpha)}\left( \frac{\vartheta^{-j\alpha}}{\vartheta^{-k\alpha}} - 1 \right) \\
&\geq - K \vartheta^{-k(2s-1+\alpha)}  \left(\frac{\vartheta x}{\vartheta^{-k}}\right)^{2s-1}
\left[ \left( \frac{\vartheta x}{\vartheta^{-k}} \right)^\alpha - 1 \right], \quad x \in I_{\vartheta^{-j}} \setminus I_{\vartheta^{-k}}.
\end{aligned}
\end{equation}
Since the r.h.s. of the above inequality does not depend on $j$, we conclude that \eqref{eq:BehUkOutisdeIkBH} holds for all $x \in \RR_+ \setminus I_{\vartheta^{-k}}$. Now, let us take $x \in I_{\vartheta^{-k}/2}$. Using that $u_k^- = 0$ in $I_{\vartheta^{-k}}$ and \eqref{eq:BehUkOutisdeIkBH}, we obtain
\[
\begin{aligned}
0 &\leq -L_\Omega u_k^-(x) = \int_0^\infty u_k^-(y) K(x,y) dy = \int_{\vartheta^{-k}}^\infty u_k^-(y) K(x,y) dy \\
&\leq  C_s \int_{\vartheta^{-k}-x}^\infty \frac{u_k^-(x+y)}{y^{1+2s}} \left( 1 + \left|\log \left(\frac{x}{y}\right)\right| \right)dy \\
&\leq C_s K \vartheta^{-k(2s-1+\alpha)}  \int_{\vartheta^{-k}-x}^\infty \left(\frac{\vartheta (x+y)}{\vartheta^{-k}}\right)^{2s-1}
\left[ \left( \frac{\vartheta (x+y)}{\vartheta^{-k}} \right)^\alpha - 1 \right] \frac{1 + \left|\log \left(\frac{x}{y}\right)\right|}{y^{1+2s}} dy \\
&\leq C_s K \vartheta^{-k(2s-1+\alpha)}  \int_{\vartheta^{-k}/2}^\infty \left(\frac{2\vartheta y}{\vartheta^{-k}}\right)^{2s-1}
\left[ \left( \frac{2\vartheta y}{\vartheta^{-k}} \right)^\alpha - 1 \right] \frac{1 + \left|\log \left(\frac{x}{y}\right)\right|}{y^{1+2s}} dy \\
&\leq C_s K \vartheta^{-k(\alpha-1)}  \int_{1/2}^\infty (2\vartheta y)^{2s-1}
\left[ (2\vartheta y)^\alpha - 1 \right] \frac{1 + |\log y| + |\log x| }{y^{1+2s}} dy \\
&\leq \varepsilon_0(\alpha) C_s K \vartheta^{-k(\alpha-1)} (1 + \log^- x),
\end{aligned}
\]
where
\[
\varepsilon_0(\alpha) := \int_{1/2}^\infty (2\vartheta y)^{2s-1}
\left[ (2\vartheta y)^\alpha - 1 \right] \frac{1 + |\log y|}{y^{1+2s}} dy.
\]
Notice that $\varepsilon_0(\alpha) \to 0$ as $\alpha \to 0$, since $(2\vartheta y)^\alpha \to 1$ as $\alpha \to 0$ for all $y > 1/2$ and Lebesgue dominated convergence theorem. Consequently, recalling that $K$ has been fixed in \eqref{eq:ChoiceKBH}, we choose $\alpha \in (0,1)$ in the following way: if $C > 0$ denotes the constant appearing in the statements of Lemma \ref{Lemma:InfInfBH} and Lemma \ref{Lemma:InteriorHarnack1D}, we take $\alpha$ small such that
\begin{equation}\label{eq:ChoiceAlphaKBH}
\varepsilon_0(\alpha) < \frac{1}{3C}, \quad \vartheta^{-\alpha} > 1 - \frac{1}{3C}.
\end{equation}
Notice that the second inequality above is guaranteed by \eqref{eq:ChoiceKBH}. Now, writing $u_k^+ = u_k + u_k^-$ and using that $L_\Omega(x^{2s-1}) = 0$ in $(0,\infty)$, $\vartheta \geq 1$ and $\alpha \in (0,1)$, we estimate
\[
\begin{aligned}
|L_\Omega u_k^+(x)| \leq |L_\Omega u(x)| + |L_\Omega u_k^-(x)| &\leq K_0 (1 + \log^- x) + \varepsilon_0(\alpha) C_s K \vartheta^{-k(\alpha-1)} (1 + \log^- x)\\
& \leq \left[K_0 + \varepsilon_0(\alpha) C_s K \right] \vartheta^{-k(\alpha-1)} (1 + \log^- x),
\end{aligned}
\]
for all $x \in I_{\vartheta^{-k}/2}$. Consequently, we can apply Lemma \ref{Lemma:InfInfBH} and Lemma \ref{Lemma:InteriorHarnack1D} to $u_k^+$ and, recalling that  $u_k^+ = u_k$ in $I_{\vartheta^{-k}}$, we deduce
\[
\begin{aligned}
\sup_{x \in I_{\vartheta^{-k}/2}^+} \left[ \frac{u(x)}{x^{2s-1}} - m_k  \right] &\leq C \bigg\{ \inf_{x \in I_{\vartheta^{-k}/2}^+} \left[ \frac{u(x)}{x^{2s-1}} - m_k  \right] + (K_0 + \varepsilon_0(\alpha) K) \vartheta^{-k\alpha} \bigg\} \\
&\leq C \bigg\{ \inf_{x \in I_{\vartheta^{-k}/4}} \left[ \frac{u(x)}{x^{2s-1}} - m_k  \right] + (K_0 + \varepsilon_0(\alpha) K) \vartheta^{-k\alpha} \bigg\}
\end{aligned}
\]
Now, defining
\[
u^k(x) := M_k x^{2s-1} - u(x),
\]
and repeating the above argument, we deduce
\[
\sup_{x \in I_{\vartheta^{-k}/2}^+} \left[ M_k - \frac{u(x)}{x^{2s-1}}   \right] \leq C \bigg\{ \inf_{x \in I_{\vartheta^{-k}/4}} \left[ M_k - \frac{u(x)}{x^{2s-1}} \right] + (K_0 + \varepsilon_0(\alpha) K) \vartheta^{-k\alpha} \bigg\}.
\]
Summing, it follows
\[
\begin{aligned}
M_k - m_k &\leq C \bigg\{ \inf_{x \in I_{\vartheta^{-k}/4}^+} \left[ \frac{u(x)}{x^{2s-1}} - m_k  \right] +  \inf_{x \in I_{\vartheta^{-k}/4}} \left[ M_k - \frac{u(x)}{x^{2s-1}} \right] \\
& \quad + (K_0 + \varepsilon_0(\alpha) K) \vartheta^{-k\alpha} \bigg\} \\
&= C \bigg\{ \inf_{x \in I_{\vartheta^{-k}/4}} \frac{u(x)}{x^{2s-1}} - \sup_{x \in I_{\vartheta^{-k}/4}} \frac{u(x)}{x^{2s-1}} + M_k - m_k  \\
& \quad + (K_0 + \varepsilon_0(\alpha) K) \vartheta^{-k\alpha} \bigg\}.
\end{aligned}
\]
In particular, we deduce
\[
\begin{aligned}
\sup_{x \in I_{\vartheta^{-(k+1)}}} \frac{u(x)}{x^{2s-1}} - \inf_{x \in I_{\vartheta^{-(k+1)}}} \frac{u(x)}{x^{2s-1}} &\leq  \frac{C-1}{C} (M_k - m_k) + (K_0 + \varepsilon_0(\alpha) K) \vartheta^{-k\alpha} \\
& = \left( \frac{C-1}{C} + \frac{K_0}{K} + \varepsilon_0(\alpha) \right) K \vartheta^{-k\alpha},
\end{aligned}
\]
and so, thanks to \eqref{eq:ChoiceKBH} and \eqref{eq:ChoiceAlphaKBH}, we find
\[
\frac{C-1}{C} + \frac{K_0}{K} + \varepsilon_0(\alpha) \leq \vartheta^{-\alpha}.
\]
Consequently, choosing
\[
M_{k+1} := \sup_{x \in I_{\vartheta^{-(k+1)}}} \frac{u(x)}{x^{2s-1}}, \qquad m_{k+1} := \inf_{x \in I_{\vartheta^{-(k+1)}}} \frac{u(x)}{x^{2s-1}},
\]
the thesis follows.
\end{proof}

We can finally prove the following.

\begin{thm} \label{Thm:BoundaryRegBH}
	Let $N=1$, $\Omega=(0,\infty)$, and $s\in (\frac12,1)$.
	Let $L_\Omega$ and $K_\Omega$ be given by either \eqref{new-operator}-\eqref{new-kernel}-\eqref{new-kernel2}, or \eqref{operator-reg}.
	Let $R > 0$ and $f \in L^{\infty}(I_{2R})$.
	Assume that
\[
\begin{cases}
L_\Omega u = f \quad \text{ in } I_{2R} \\
u(0) = 0,
\end{cases}
\]
and $u$ satisfies \eqref{eq:GrowthCondBH} for some $c_0 > 0$, $\varepsilon_0 > 1$. Then the function
\[
x \to \frac{u(x)}{x^{2s-1}}
\]
can be continuously extended up to $x = 0$ and, furthermore, there exist $\alpha \in (0,1)$ and $C > 0$ (depending on $s$, $c_0$ and $\varepsilon_0$), such that
\begin{equation}\label{eq:CalphaBH}
\left| \frac{u(x)}{x^{2s-1}} - \frac{u(y)}{y^{2s-1}} \right| \leq C R^{1-2s}\left( \frac{|x-y|}{R} \right)^\alpha \left[ \|u\|_{L^\infty(I_{2R})} + R^{2s}\|f\|_{L^\infty(I_{2R})}\right],
\end{equation}
for all $x,y \in \overline{I_R}$.
\end{thm}
\begin{proof} We define $\delta(x) := x$, $v := u/\delta^{2s-1}$, $K_0 := \|f\|_{L^\infty(I_2)}$ and we set $R = 1$. First, from \emph{Step 1} of the proof of Lemma \ref{Lemma:OscDecayBH}, we have
\begin{equation}\label{eq:LinftyvBH}
\|v\|_{L^\infty(I_1)} \leq  C_0 \left( \|u\|_{L^\infty(I_2)} + K_0 \right),
\end{equation}
for some suitable $C_0 > 0$ depending only on $s$, $c_0$ and $\varepsilon_0$. Further, by Lemma \ref{Lemma:OscDecayBH}, we have also (see \eqref{eq:IneqBH})
\begin{equation}\label{eq:OscDecayvBHHolder}
\sup_{I_\varrho} v - \inf_{I_\varrho} v \leq C \varrho^\gamma \left[ \|u\|_{L^\infty(I_2)} + K_0 \right],
\end{equation}
for some $\gamma \in (0,1)$, $C > 0$ (depending only on $s$, $c_0$ and $\varepsilon_0$) and all $\varrho \in (0,1]$. In particular, notice that from \eqref{eq:OscDecayvBHHolder} one can easily deduce that $v$ can be continuously extended up to $x = 0$.

Now, for any $x \in I_1$, we set $J_r^+ := (x/2,3x/2)$. Thus,
\[
[u]_{C^{0,\beta}(\overline{J_r^+})} \leq C r^{-\beta} \left[ \|u\|_{L^\infty(I_2)} + K_0 \right],
\]
for all $\beta \in (0,\beta_\ast)$ and some suitable $\beta_\ast \in (0,1)$ (cf. Theorem \ref{Theorem:HolderRegularity1}). On the other hand, it is not difficult to check that
\[
\|\delta^{1-2s}\|_{L^\infty(\overline{J_r^+})} \leq C_s r^{1-2s}, \qquad [\delta^{1-2s}]_{C^{0,1}(\overline{J_r^+})} \leq C_s r^{-2s},
\]
for some $C_s > 0$ depending only on $s$. As a consequence, by interpolation
\[
[\delta^{1-2s}]_{C^{0,\beta}(\overline{J_r^+})} \leq C_s r^{1-2s-\beta},
\]
for all $\beta \in (0,1)$. Thus, for any $\beta \in (0,\beta_\ast)$ and all $z,y \in J_r^+$ ($z \not= y$), using the definition of $v$, it follows
\[
\begin{aligned}
\frac{|v(z) - v(y)|}{|z-y|^\beta} & \leq \|\delta^{1-2s}\|_{L^\infty(\overline{J_r^+})} \frac{|u(z) - u(y)|}{|z-y|^\beta} + \|u\|_{L^\infty(\overline{I_2})}\frac{|\delta^{1-2s}(z) - \delta^{1-2s}(y)|}{|z-y|^\beta} \\
& \leq C r^{1-2s-\beta} \left[ \|u\|_{L^\infty(I_2)} + K_0 \right],
\end{aligned}
\]
for some new constant $C > 0$, which implies
\begin{equation}\label{eq:BetaHolderBoundBH}
[v]_{C^{0,\beta}(\overline{J_r^+})} \leq C r^{1-2s-\beta} \left[ \|u\|_{L^\infty(I_2)} + K_0 \right],
\end{equation}
for all $\beta \in (0,\beta_\ast)$. Now, we see how \eqref{eq:LinftyvBH}, \eqref{eq:OscDecayvBHHolder}, and \eqref{eq:BetaHolderBoundBH} lead to
\[
[v]_{C^{0,\alpha}(\overline{I_1})} \leq C \left[ \|u\|_{L^\infty(I_2)} + K_0 \right],
\]
for some $\alpha \in (0,1)$ depending only on $s$, $c_0$ and $\varepsilon_0$.

Given $x,y \in \overline{I_1}$, we suppose $x \geq y$, and set $\widetilde{\varrho} = x$, $\varrho = |x-y|$. Notice that thanks to \eqref{eq:LinftyvBH}, we can assume $\varrho \in (0,1)$. Finally, we fix
\[
p > \frac{\beta + 2s -1}{\beta},
\]
where $\beta \in (0,\beta_\ast)$ as above. There are two possible cases:

\emph{Case 1.} $\varrho \geq \widetilde{\varrho}^p/2$. Then, thanks to \eqref{eq:OscDecayvBHHolder},
\[
\begin{aligned}
|v(x) - v(y)| &\leq |v(x) - v(0)| + |v(0) - v(y)| \leq C \left[ \|u\|_{L^\infty(I_2)} + K_0 \right] \widetilde{\varrho}^{\gamma} \\
&\leq C \varrho^{\gamma / p} \left[ \|u\|_{L^\infty(I_2)} + K_0 \right],
\end{aligned}
\]
and so it is enough to choose $\alpha = \gamma / p$.

\emph{Case 2.} Assume $\varrho \leq \widetilde{\varrho}^p/2$. Since $p > 1$, we see that $y \in J_{\widetilde{\varrho}}^+ = (x/2,3x/2)$ and so, using \eqref{eq:BetaHolderBoundBH}, it follows
\[
|v(x) - v(y)| \leq C \widetilde{\varrho}^{\,1-2s-\beta} \varrho^\beta \left[ \|u\|_{L^\infty(I_2)} + K_0 \right] \leq C \varrho^{\beta - \frac{\beta + 2s -1}{p}} \left[ \|u\|_{L^\infty(I_2)} + K_0 \right],
\]
and so we complete the proof by choosing $\alpha := \min\left\{\frac{\gamma}{p}, \beta - \frac{\beta + 2s -1}{p} \right\}> 0$.
\end{proof}

\subsection{Proof of the Liouville theorem}

First, as a consequence of the 1D boundary Harnack, we can deduce the following Neumann Liouville theorem in the half-line.

\begin{cor} \label{Cor:Liouville1D}
	Let $N=1$, $\Omega=(0,\infty)$, and $s\in (\frac12,1)$.
	Let $L_\Omega$ and $K_\Omega$ be given by either \eqref{new-operator}-\eqref{new-kernel}-\eqref{new-kernel2}, or \eqref{operator-reg}.
	Assume that
\begin{equation}\label{eq:HalfLineProb}
\begin{cases}
L_\Omega u = 0 \quad \text{ in } \mathbb{R}_+ \\
u(0) = 0,
\end{cases}
\end{equation}
and $u$ satisfies
\begin{equation}\label{eq:GrowthCondLiouville1D}
|u(y)| \leq c_0 (1 + y^{2s-1+\varepsilon}), \quad y > 0,
\end{equation}
for some $c_0 > 0$ and $\varepsilon \in (0,\alpha)$, where $\alpha \in (0,1)$ is as in Theorem \ref{Thm:BoundaryRegBH}.
Then,
\[
u(x) = A x^{2s-1},
\]
for some $A \in \mathbb{R}$.

Furthermore, if in addition $u$ satisfies \eqref{eq:HalfLineProb} in the weak sense with Neumann condition (in the sense of Definition~\ref{Def:WeakSolNeumann}) at $x = 0$, then $u = 0$ in $\overline{\mathbb{R}_+}$.
\end{cor}
\begin{proof} From \eqref{eq:GrowthCondLiouville1D}, we immediately see that
\[
\|u\|_{L^\infty(I_{2R})} \leq C_0(1 + R^{2s-1+\varepsilon}),
\]
for some $C_0 > 0$ depending only on $s$, $c_0$ and $\varepsilon$, and all $R > 0$. On the other hand, we notice that all the assumptions of Theorem \ref{Thm:BoundaryRegBH} are satisfied (in particular, \eqref{eq:GrowthCondLiouville1D} implies \eqref{eq:GrowthCondBH}). Thus, setting $v(x):= u(x)/x^{2s-1}$, and combining \eqref{eq:CalphaBH} with the above inequality, it follows
\[
[v]_{C^{0,\alpha}(I_R)} \leq C R^{1-2s-\alpha}  \|u\|_{L^\infty(I_{2R})} \leq C R^{\varepsilon-\alpha},
\]
for some new constant $C > 0$ and all $R > 0$. Since $\varepsilon \in (0,\alpha)$, we can pass to the limit as $R \to +\infty$ to deduce $[v]_{C^{0,\alpha}(\mathbb{R}_+)} = 0$, which trivially implies that $v = A$ for some $A \in \mathbb{R}$, i.e. the first part of our thesis.

To show the second part, we recall that $u$ satisfies
\[
\int_{\mathbb{R}_+} \int_{\mathbb{R}_+} [u(x) - u(y)][\eta(x) - \eta(y)] K_\Omega(x,y) dxdy = 0,
\]
for all $\eta \in C_0^\infty(\overline{\mathbb{R}_+})$ and, since $u \in C^\infty(\mathbb{R}_+)$ (see \cite{MY}), it satisfies $L_\Omega u = 0$ in $\mathbb{R}_+$. Consequently, from the first part of the statement we deduce that $u(x) = A x^{2s-1}$, for some $A \in \mathbb{R}$.

However, assume $A > 0$ and take $\eta \in C_0^\infty((-\infty,1])$, with $\eta' \leq 0$ and $\eta \not\equiv 0$. Using that $x \to x^{2s-1}$ is strictly increasing in $\mathbb{R}_+$, it follows
\[
\begin{aligned}
0 &= A\int_{\mathbb{R}_+} \int_{\mathbb{R}_+} [x^{2s-1} - y^{2s-1}][\eta(x) - \eta(y)] K_\Omega(x,y) dxdy \\
&= A\int_{\{x < y\}}  \underbrace{[x^{2s-1} - y^{2s-1}]}_{< 0}\underbrace{[\eta(x) - \eta(y)]}_{\geq 0} K_\Omega(x,y) dxdy \\
& \quad + A\int_{\{x \geq y\}} \underbrace{[x^{2s-1} - y^{2s-1}]}_{\geq 0}\underbrace{[\eta(x) - \eta(y)]}_{\leq 0} K_\Omega(x,y) dxdy < 0,
\end{aligned}
\]
since $\eta \not\equiv 0$ (similar if we assume $A < 0$). This leads to a contradiction, unless $A = 0$, and thus $u = 0$.
\end{proof}

In order to extend the previous Neumann Liouville theorem to higher dimensions we need some preliminary lemmata. The first one concerns H\"older regularity of solutions in the half-space.

\begin{lem}
 \label{Lemma:LiouvilleND}
	 Let $\Omega=\RR^N_+ = \{x_N>0\}$, and $s\in (\frac12,1)$. Let $L_\Omega$ and $K_\Omega$ be given by either \eqref{new-operator}-\eqref{new-kernel}-\eqref{new-kernel2}, or \eqref{operator-reg}.
	Assume that $v$ is a weak solution to
	$$ L_\Omega v = 0 \quad \text{ in } \mathbb{R}^N_+$$
	with Neumann condition on $\partial \mathbb{R}^N_+=\{x_N=0\}$ (in the sense of Definition~\ref{Def:WeakSolNeumann}). If
	$$||v||_{L^\infty(B_R^+)} \leq R^{\sigma}, \quad R \geq  1,$$
	for some $0<\sigma<2s$. Then
	$$[v]_{C^\alpha (B_R^+)} \leq C R^{\sigma-\alpha}, \quad R \geq 1,$$
	for some constant $C>0$ depending only on $N, s$, and $\sigma$, and $\alpha$ as in Theorem~\ref{Theorem:HolderRegularity1}.
\end{lem}

\begin{proof} As usual along this paper, we are proving the result in the case $L_\Omega$ and $K_\Omega$ are given by \eqref{new-operator}-\eqref{new-kernel}-\eqref{new-kernel2}. The other case is analogous, but without the logarithmic corrections.
	
	The main idea is to apply Theorem~\ref{Theorem:HolderRegularity1} but, since $v$ is not bounded, we first need to cut it in a suitable way in order to making use of the H\"older estimate. By scaling, it is enough to prove the result for the case $R=1$.
	
	Let us define the auxiliary function $w(x) = v(x) \chi_{B_{4}}(x)$. It is clear, due to the growth condition on $v$, that this new function $w$ is bounded in $\mathbb{R}^N_+$. Indeed,
	$$ ||w||_{L^\infty(\RR^N_+)} \leq 4^\sigma. $$
	
	First, we prove that $w$ satisfies
	$$ L_\Omega w = f \quad \text{ in } B_{2}^+,$$
	in the weak sense with Neumann condition on $ \partial \RR^N_+ \cap B_{2}$, where $ f\in L^q(B_2^+)$ is a function which will be determined next. So, given any test function $\eta \in C^\infty_0(B_{2})$ and using the equation satisfied by $v$ we have
	\begin{align*}
	B(w,\eta) &= \int_{\RR^N_+}\int_{\RR^N_+}(w(x)-w(y))(\eta(x)-\eta(y))\,K_\Omega(x,y) dx dy \\ &= \int_{\RR^N_+}\int_{\RR^N_+}(v(x)\chi_{B_{4}}(x)-v(y)\chi_{B_{4}}(y))(\eta(x)-\eta(y))\,K_\Omega(x,y) dx dy \\
	&=\int_{B_{4}^+}\int_{B_{4}^+} (v(x)-v(y))(\eta(x)-\eta(y)) \,K_\Omega(x,y) dx dy \\
	&\ \ \ \ \ \ + 2 \int_{B_{4}^+} dx \int_{(B_{4}^c)^+} dy\, v(x)\eta(x) \,K_\Omega(x,y)  \\
	&= \int_{B_{2}^+} \left(2\int_{(B_{4}^c)^+} v(y)\,K_\Omega(x,y) dy \right) \eta(x) dx =: \int_{B_{2}^+} f(x) \eta(x) dx
	\end{align*}
	
	Then, given any $x\in B_{2}^+$ we claim that $f$ satisfies the following pointwise estimate
	$$ |f(x)| \leq  C \left(1+\log^-(x_N) \right), $$
	for some positive constant $C$ depending only on $N, s$ and $\sigma$. In particular, it follows that $f\in~L^q(B_2^+)$ for any $1\leq q < \infty$.
	
	Now, if we apply Theorem~\ref{Theorem:HolderRegularity1} to $w$ with $q=N/s$, and we take into account that $v\equiv w$ in $B_{2}^+$ we obtain
	$$ [v]_{C^\alpha (B_1^+)} = [w]_{C^\alpha (B_1^+)} \leq C \left( ||w||_{L^\infty(\RR^N_+)} + ||f||_{L^q(B_{2}^+)} \right) \leq C, $$
	as we wanted.

	Finally, let us prove the pointwise estimate for $f$. Letting $d = d_{x,y}$, using \eqref{new-kernel-estimates} and taking into account that $|y|/2 \leq |x-y| \leq 2|y|$ and $d\leq |x-y|$ when $ x\in B_2$ and $ y\in B_{4}^c$, we have
	\begin{align*}
	|f(x)| &= 2\left|\int_{(B_{4}^c)^+} v(y)\,K_\Omega(x,y) dy\right| \leq C\int_{(B_{4}^c)^+} |y|^\sigma \,\frac{1+\log^-\left(\frac{d}{|x-y|}\right)}{|x-y|^{N+2s}} dy \\
	&= C \int_{(B_{4}^c)^+} \frac{|y|^\sigma dy}{|x-y|^{N+2s}} +  C \int_{(B_{4}^c)^+} |y|^\sigma \,\frac{\log\left(\frac{|x-y|}{d}\right)}{|x-y|^{N+2s}} dy \\
&\leq C \int_{(B_{4}^c)^+} \frac{dy}{|y|^{N + 2s - \sigma}}  + C \int_{(B_{4}^c)^+} \frac{\log (2|y|) + |\log d|}{|y|^{N + 2s}} dy\\
&\leq C + C\int_{(B_{4}^c)^+ \cap \{x_N \leq y_N \}} \frac{\log |y| + |\log x_N|}{|y|^{N + 2s}} dy \\
& \quad + C\int_{(B_{4}^c)^+ \cap \{y_N<x_N \}} \frac{\log |y| + |\log y_N|}{|y|^{N + 2s}} dy \leq  C \left(1+\log^-(x_N)\right),
	\end{align*}
	for some positive constant $C$ depending only on $N, s$ and $\sigma$. Here, it is crucial the fact that $\sigma<2s$ and $d(x) = x_N$ together with the integrability of $\log d$ close to $\partial \mathbb{R}^N_+$.
\end{proof}

Next step is proving that weak solutions to $L_\Omega v = 0$ in $\Omega=\RR^N_+$ are linear functions.
\begin{prop} \label{Prop:LiouvilleND}
	Let $\Omega=\RR^N_+ = \{x_N>0\}$, and $s\in (\frac12,1)$. Let $L_\Omega$ and $K_\Omega$ be given by either \eqref{new-operator}-\eqref{new-kernel}-\eqref{new-kernel2}, or \eqref{operator-reg}. Assume $v$ is a weak solution to
	$$ L_\Omega v = 0 \quad \text{ in } \mathbb{R}^N_+ = \{x_N>0\} $$
	with Neumann condition on $\partial \mathbb{R}^N_+=\{x_N=0\}$ (in the sense of Definition~\ref{Def:WeakSolNeumann}). If
	\begin{equation} \label{eq: GrowthConditionPropLiouvilleND}
	||v||_{L^\infty(B_R^+)} \leq c_0 (1 + R^{\sigma}), \quad R > 0,
	\end{equation}
	for some $c_0 > 0$ and $0<\sigma<2s$. Then, there exist functions $w_0,...,w_{N-1}$ such that
	$$ v(x) = w_0(x_N) + \sum_{i=1}^{N-1} w_i(x_N) x_i.$$
	Furthermore, $v(x) = w_0(x_N)$ if $\sigma<1$.
\end{prop}

\begin{proof}
	Note that we can assume that $||v||_{L^\infty(B_R^+)} \leq R^{\sigma}$ for every  $R > 1$, after dividing by a suitable constant.
	
	First, we prove that $v$ is a polynomial in the first $N-1$ variables with coefficients depending on $x_N$, i.e.,
	$$ v(x) = \sum_{|j|\leq N} a_j(x_N) \tilde{x}^j, $$
	where $j=(j_1,...,j_{N-1})$ is a multiindex and $\tilde{x}^j = x_1^{j_1}\cdots x_{N-1}^{j_{N-1}}$.
	
	By Lemma~\ref{Lemma:LiouvilleND} we know that $[v]_{C^\alpha (B_R)} \leq C R^{\sigma-\alpha}$. Now, given any direction $e=(\tilde{e},0)\in S^{N-1}$ and any $h>0$ we define the function
	$$ v_{h,1}^e(x) = \frac{v(x+he)-v(x)}{C |h|^\alpha}, $$
	where $C$ is the positive constant appearing in the statement of Lemma~\ref{Lemma:LiouvilleND}. Then, since $e_N = 0$, it is clear that $v_{h,1}^e$  satisfies
	\[
	\begin{cases}
	L_\Omega v_{h,1}^e = 0 \quad \text{ in } \mathbb{R}^N_+ \\
	||v_{h,1}^e||_{L^\infty(B_R^+)} \leq R^{\sigma-\alpha}, \quad R > 1.
	\end{cases}
	\]
Now, since $v_{h,1}^e$ satisfies the same equation of $v$ and an ``improved'' growth condition, we can iterate this procedure and, defining recursively
	$$  v_{h,k}^e(x) = \frac{v_{h,k-1}^e(x+he)-v_{h,k-1}^e(x)}{C |h|^\alpha}, $$
	we obtain that $||v_{h,k}^e||_{L^\infty(B_R^+)} \leq R^{\sigma-k\alpha}$.  Therefore, if we choose $k\geq d+1 := \lceil\sigma/\alpha\rceil $ and take $R\to\infty$ we get that $v_{h,d+1}^e\equiv 0$ in $\RR_+^N$. By definition, this means that the discrete differences of order $d$ of $v$ in every direction $e$ are zero and thus $v$ is a polynomial of degree $d$ in the first $N-1$ variables. Furthermore, in view of \eqref{eq: GrowthConditionPropLiouvilleND} and that $\sigma<2s<2$, it follows  $d=1$ and therefore $v$ has the form stated above. Indeed, since for any given $x_N>0$, $v(\cdot, x_N)$ is a polynomial of degree $d$, then $||v(\cdot,x_N)||_{L^\infty(B_R^+)}\geq c R^d$, for some constant $c$ depending on $x_N$ and any $R>1$. On the other hand, by \eqref{eq: GrowthConditionPropLiouvilleND} we obtain that $||v(\cdot,x_N)||_{L^\infty(B_R^+)}\leq C R^\sigma$ with $\sigma<2$. It thus follows that $d=1$. Notice that when $\sigma<1$ we get that $d=0$ and so we conclude $v(x) = w_0(x_N)$.
\end{proof}
\begin{lem} \label{Lemma:NDto1D}
	Let $\Omega=\RR^N_+ = \{x_N>0\}$, and $s\in (\frac12,1)$. Let $B_\Omega$ be given by \eqref{B} with $K_\Omega$ either of the form \eqref{new-kernel}-\eqref{new-kernel2}, or \eqref{operator-reg}. Assume $v,\tilde{v}\in H_K(\RR^N_+)$ and $\eta\in C^\infty_0(\RR^N)$ are functions of the form $v(x) = x_i w(x_N)$ for some $i\in\{1,...,N-1\}$, $\tilde{v}(x) = \tilde{w}(x_N)$ and $\eta(x) = \tilde{\eta}(\bar{x})\eta_N(x_N)$ with $x=(\bar{x},x_N)\in \RR^{N-1}\times \RR_+$. Then,
	$$ B_{\RR^N_+}(\tilde{v},\eta) = \left( \int_{\RR^{N-1}} \tilde{\eta}(z) \ dz \right) B_ {\RR_+}(\tilde{w},\eta_N), $$
	and
	$$ B_{\RR^N_+}(v,\eta) = \left( \int_{\RR^{N-1}} z_i \tilde{\eta}(z) \ dz \right) B_{\RR_+}(w,\eta_N). $$
\end{lem}

\begin{proof}
	The proof comes from direct computation. On the one hand, if we use the form of $\tilde{v}$ and $\eta$, add and subtract the term $\eta_N(x_N)\tilde{\eta}(\bar{y}) (\tilde{w}(x_N)- \tilde{w}(y_N))$ and rearrange them, we arrive at
	\begin{align*}
	B_{\RR^N_+}(\tilde{v},\eta) &= \int_{\RR^N_+} \int_{\RR^N_+} (\tilde{v}(x)-\tilde{v}(y))(\eta(x)-\eta(y))\,K_{\RR^N_+}(x,y) \ dx dy \\
	& =  \int_{\RR^N_+} \int_{\RR^N_+} (\tilde{w}(x_N)- \tilde{w}(y_N))(\tilde{\eta}(\bar{x})\eta_N(x_N)-\tilde{\eta}(\bar{y})\eta_N(y_N))\,K_{\RR^N_+}(x,y) \ dx dy \\
	& = \int_{\RR^N_+} \int_{\RR^N_+} \eta_N(x_N)(\tilde{w}(x_N)-\tilde{w}(y_N)) (\tilde{\eta}(\bar{x})-\tilde{\eta}(\bar{y}))\,K_{\RR^N_+}(x,y) \ dx dy \\
	& \ \ + \int_{\RR^N_+} \int_{\RR^N_+} \tilde{\eta}(\bar{y}) (\tilde{w}(x_N)-\tilde{w}(y_N))(\eta_N(x_N)-\eta_N(y_N))\,K_{\RR^N_+}(x,y) \ dx dy \\
	&=: J_1+J_2.
	\end{align*}
	Now, one can conclude that $J_1=0$ due to the antisymmetry of the integrand with respect to the variables $\bar{x}$ and $\bar{y}$. Next, we can use the identity
	\begin{equation} \label{eq:OperatorOf1DFunction}
	\int_{\RR^{N-1}} K_{\RR^N_+}(\bar{x},x_N,\bar{y},y_N) \,  d\bar{x} = K_{\RR^+}(x_N,y_N),
	\end{equation}	
	which can be easily checked in both frameworks: $K_\Omega$ either of the form \eqref{new-kernel}-\eqref{new-kernel2}, or \eqref{operator-reg}, in order to deduce that
	$$ J_2 = \left( \int_{\RR^{N-1}} \tilde{\eta}(\bar{y}) \ d\bar{y} \right) B_ {\RR_+}(\tilde{w},\eta_N). $$
	
	On the other hand, if we use the form of $v$ and $\eta$, add and subtract again different terms and rearrange them, we arrive at
	\begin{align*}
	B_{\RR^N_+}(v,\eta) &= \int_{\RR^N_+} \int_{\RR^N_+} \frac{(v(x)-v(y))(\eta(x)-\eta(y))}{|x-y|^{N+2s}} \ dx dy \\
	& =  \int_{\RR^N_+} \int_{\RR^N_+} (x_i w(x_N)-y_i w(y_N))(\tilde{\eta}(\bar{x})\eta_N(x_N)-\tilde{\eta}(\bar{y})\eta_N(y_N))\,K_{\RR^N_+}(x,y) \ dx dy \\
	& = \int_{\RR^N_+} \int_{\RR^N_+} \eta_N(x_N)(w(x_N)-w(y_N)) y_i (\tilde{\eta}(\bar{x})-\tilde{\eta}(\bar{y}))\,K_{\RR^N_+}(x,y) \ dx dy \\
	& \ \ + \int_{\RR^N_+} \int_{\RR^N_+} w(x_N)(\eta_N(x_N)-\eta_N(y_N))(x_i-y_i) \tilde{\eta}(\bar{y})\,K_{\RR^N_+}(x,y) \ dx dy \\
	& \ \ + \int_{\RR^N_+} \int_{\RR^N_+} w(x_N)\eta_N(x_N) (x_i-y_i)(\tilde{\eta}(\bar{x})-\tilde{\eta}(\bar{y}))\,K_{\RR^N_+}(x,y) \ dx dy \\
	& \ \ + \int_{\RR^N_+} \int_{\RR^N_+} y_i \tilde{\eta}(\bar{y}) (w(x_N)-w(y_N))(\eta_N(x_N)-\eta_N(y_N))\,K_{\RR^N_+}(x,y) \ dx dy \\
	& = I_1 + I_2 + I_3 + I_4.
	\end{align*}
	Now, we show that the first three integrals are zero while the last one give us the desired result. That is, by symmetrization with respect to the variables $\bar{x}$ and $\bar{y}$ and the translation invariance and odd symmetry of the kernel $K_{\RR^N_+}(x,y)$ in the first $N-1$ variables, we get
	\begin{align*}
	I_1 & = \int_{\RR^N_+} \int_{\RR^N_+} \eta_N(x_N)(w(x_N)-w(y_N)) y_i (\tilde{\eta}(\bar{x})-\tilde{\eta}(\bar{y}))\,K_{\RR^N_+}(x,y) \ dx dy \\
	& = -\frac{1}{2} \int_{\RR^N_+} \int_{\RR^N_+} \eta_N(x_N)(w(x_N)-w(y_N)) (x_i-y_i) (\tilde{\eta}(\bar{x})-\tilde{\eta}(\bar{y}))\,K_{\RR^N_+}(x,y) \ dx dy \\
	& = \int_{\RR^N_+} \int_{\RR^N_+} \eta_N(x_N)(w(x_N)-w(y_N)) (x_i-y_i) \tilde{\eta}(\bar{y})\,K_{\RR^N_+}(x,y) \ dx dy \\
	& = \int_{\RR^N_+} \int_0^\infty \int_{\RR^{N-1}} \eta_N(x_N)(w(x_N)-w(y_N)) z_i \tilde{\eta}(\bar{y})\,K_{\RR^N_+}(z+\bar{y},x_N,\bar{y},y_N) \ dz dx_N dy \\
	& = 0
	\end{align*}
	The computations of $I_2$ and $I_3$ are completely analogous, although we do not have to do the first symmetrization. Next, we proceed with $I_4$. By using again the identity \eqref{eq:OperatorOf1DFunction} we arrive at
	\begin{align*}
	I_4 & = \int_{\RR^N_+} \int_{\RR^N_+} y_i \tilde{\eta}(\bar{y}) (w(x_N)-w(y_N))(\eta_N(x_N)-\eta_N(y_N))\,K_{\RR^N_+}(x,y) \ dx dy \\
	& = \left( \int_{\RR^{N-1}} y_i \tilde{\eta}(\bar{y}) \ d\bar{y} \right) B_ {\RR_+}(w,\eta_N).
	\end{align*}
\end{proof}

Finally we present the proof of Theorem~\ref{Thm:LiouvilleND2}.

\begin{proof}[Proof of Theorem~\ref{Thm:LiouvilleND2}]
	First, by applying Proposition~\ref{Prop:LiouvilleND} with $\sigma = 2s-1+\varepsilon$ we know that
	$$v(x) = w_0(x_N) + \sum_{i=1}^{N-1} w_i(x_N) x_i.$$
	Now, we are going to take advantage of Lemma~\ref{Lemma:NDto1D} to prove that every $w_i$ satisfies
	\begin{equation} \label{eq: Equationw_i}
	L_\Omega w_i = 0 \quad \text{ in } \mathbb{R}_+
	\end{equation}
	in the weak sense with Neumann boundary condition at $0$. To do this, let us take any test function with separated variables, i.e., $\eta(z) = \tilde{\eta}(\bar{z}) \eta_N(z_N)$. Then, by applying Lemma~\ref{Lemma:NDto1D} and the fact that $v$ is a weak solution of the problem ($B_{\RR^N_+}(v,\eta)=0$), we obtain
	$$ B_ {\RR_+}(w_0,\eta_N) \int_{\RR^{N-1}} \tilde{\eta}(z) \ dz + \sum_{i=1}^{N-1} \left( B_ {\RR_+}(w_i,\eta_N) \int_{\RR^{N-1}} z_i \tilde{\eta}(z) \ dz \right) = 0 , $$
	for any given $\tilde{\eta} \in C^\infty_0(\RR^{N-1})$ and $\eta_N \in C^\infty_0(\RR_+)$.
	
	We claim that this equality is equivalent to $B_{\RR_+}(w_i,\eta_N)=0$ for any $\eta_N \in C^\infty_0(\RR_+)$ and therefore that $w_i$ satisfies \eqref{eq: Equationw_i}, as we wanted. In order to show that we only need to choose $\tilde{\eta}$ properly. On the one hand, by taking a radial $\tilde{\eta}$, we get that $B_{\RR_+}(w_0,\eta_N)=0$. On the other, if we choose the test function $\tilde{\eta}$ to be odd with respect to the $i^{th}$-variable and even with respect to the others we conclude $B_{\RR_+}(w_i,\eta_N)=0$ for $i>0$.
 	
	Moreover, it is clear that each $w_i$ satisfies the same growth condition as $v$, i.e., $||w_i||_{L^\infty(B_R^+)} \leq c_0 (1 + R^{2s-1+\varepsilon})$ for any $R > 0$ and so, applying Corollary~\ref{Cor:Liouville1D} to each $w_i$, we obtain the desired result:
	$$ v(x) = a + \sum_{i=1}^{N-1} b_i x_i, $$
as wanted.
\end{proof}

\section{Higher regularity by blow-up}
\label{sec6}

The aim of this final section is to establish a $C^{2s-1+\alpha}$ estimate (in case $s>\frac12$), by combining the $C^\alpha$ estimate from Section \ref{sec4}, a blow-up argument in the spirit of \cite{2016RosOtonSerra:art}, and the Liouville theorem with nonlocal Neumann conditions established in Section \ref{sec5}.

We will also need the following.

\begin{lem}\label{Lemma:LocalHsBound}
Let $\Omega \subset \mathbb{R}^N$ be any Lipschitz domain, $f \in L^2_{loc}(\Omega)$ and $x_0 \in \Omega$.
Let $L_\Omega$ and $K_\Omega$ be given by either \eqref{new-operator}-\eqref{new-kernel}-\eqref{new-kernel2}, or \eqref{operator-reg}.
Assume that $u$ satisfies
	\[
	L_\Omega u = f \quad \text{ in } \Omega
	\]
	with Neumann conditions on $\partial\Omega$. 	
Assume that
\[|u(x)|\leq M_0(1+|x|^{s-\varepsilon})\quad \textrm{in}\quad \RR^N.\]
Then, for any $0 < r < R$ and any $x_0 \in \Omega$, we have
	\[
	[u]_{H_K(D_r(x_0))}^2 \leq C \left\{ \|f\|_{L^2(D_R(x_0))}^2 + M_0^2 \right\},
	\]
with $C$ depending only on $N$, $s$, $x_0$, $\varepsilon$, $r$ and $R$.
\end{lem}

\begin{proof} Fix $x_0 \in \Omega$ and $0 < r < R$. Let $\varphi \in C_0^\infty(B_R(x_0))$, such that $0 \leq \varphi \leq 1$ and $\varphi = 1$ in $B_r(x_0)$. Testing the weak formulation with $\eta = u\varphi$, we obtain
	\[
	B(u,\eta) := \int_\Omega \int_\Omega [u(x) - u(y)][u(x)\varphi(x) - u(y)\varphi(y)] K_\Omega(x,y) dxdy = \int_\Omega fu\varphi dx.
	\]
	Writing
	\[
	u(x)\varphi(x) - u(y)\varphi(y) = [u(x) - u(y)]\varphi(x) + u(y)[\varphi(x) - \varphi(y)],
	\]
	we deduce by symmetry
	\[
	2[u(x) - u(y)][u(x)\varphi(x) - u(y)\varphi(y)] = [u(x) - u(y)]^2[\varphi(x) + \varphi(y)] + [u^2(x) - u^2(y)][\varphi(x) - \varphi(y)].
	\]
	Consequently, using the symmetry of $K_\Omega$ and the definition of $\varphi$, it follows
	\[
	\begin{aligned}
	2B(u,\eta) &= \int_\Omega \int_\Omega [u(x) - u(y)]^2[\varphi(x) + \varphi(y)] K_\Omega(x,y) dxdy \\
	& \quad + \int_\Omega \int_\Omega [u^2(x) - u^2(y)][\varphi(x) - \varphi(y)] K_\Omega(x,y) dxdy \\
	&\geq 2[u]_{H_K(D_r(x_0))}^2 - 2 \int_\Omega u^2(x) |L_\Omega \varphi(x)| dx.
	\end{aligned}
	\]
	Now, since $\varphi \in C_0^\infty(B_R(x_0))$, we claim hat
	\begin{equation} \label{claim-smooth}
	\int_\Omega u^2(x) |L_\Omega\varphi(x)| dx \leq CM_0^2\int_\Omega (1+|x|^{2s-2\varepsilon}) |L_\Omega\varphi(x)| dx\leq CM_0^2,
	\end{equation}
	for some constant $C$ depending on $\Omega$, $N$, $s$, $R$, $\varepsilon$, and $x_0$.
	If \eqref{claim-smooth} holds, then
	\[
	[u]_{H^s(D_r(x_0))}^2 \leq \int_{D_R(x_0)} f u dx + CM_0^2,
	\]
	and combining Young's inequality with the growth condition on $u$ we complete the proof.
	Hence, it only remains to prove \eqref{claim-smooth}.
	
	Let us estimate $|L_\Omega \varphi|$.
	For this, notice first that since $\varphi$ is Lipschitz, then
	\[|L_\Omega \varphi(x)| \leq C\int_{\Omega} |x-y|\,K_\Omega(x,y)dy,\]
	which gives a universal bound whenever $s<\frac12$.
	However, in case $s\geq \frac12$ the bound is nontrivial, since we cannot immediately symmetrize the integral.
	In that case, we separate the proof into two cases.
	
\vspace{2mm}

\noindent$\bullet$  \emph{Assume first that $L_\Omega$ is given by \eqref{operator-reg}.} Let  $x \in B_{2R}(x_0)$ and $d = d(x)$.
	Then,
\[
\begin{aligned}
L_\Omega \varphi (x) &= PV \int_{B_d(x)} (\varphi(x) - \varphi(y)) K_\Omega(x,y) dy + \int_{\Omega\setminus B_d(x)} (\varphi(x) - \varphi(y)) K_\Omega(x,y) dy \\
&:= I + J.
\end{aligned}
\]
By the regularity of $\varphi$ and symmetry of $K_\Omega$ inside $B_d(x)$, it follows that
\[
|I| \leq \int_{B_d} \frac{|2\varphi(x) - \varphi(x-y) - \varphi(x+y)|}{|y|^{N+2s}} dy \leq C \int_{B_d} \frac{dy}{|y|^{N+2s-2}} dy \leq C,
\]
for some constant depending on $N$, $s$, $\Omega$ and $\varphi$.
Further, since $\varphi$ is Lipschitz, we obtain
\[
|J| \leq \int_{\Omega\setminus B_d(x)} |\varphi(x) - \varphi(y)| K_\Omega(x,y) dy  \leq C \int_{\mathbb{R}^N\setminus B_d(x)} \frac{dy}{|x-y|^{N+2s-1}}\leq C d(x)^{1-2s},
\]
with $C$ depending only on $N$, $s$ and $\varphi$.

Consequently, we have proved
\begin{equation}\label{eq:L2HsEstLPh}
|L_\Omega \varphi(x)| \leq C(1 + d^{1-2s}(x)), \quad x \in B_{2R}(x_0).
\end{equation}

Now, since $\varphi$ has compact support in $B_R(x_0)$, for all $x \in B_{2R}(x_0)^c$ we find
\begin{equation}\label{eq:BounbOpertestFunctions}
|L_\Omega\varphi(x)| \leq \int_\Omega |\varphi(y)| K_\Omega(x,y) dy \leq C \int_{\mathrm{supp} \varphi} \frac{dy}{|x-y|^{N+2s}} \leq \frac{C}{(1 + |x|)^{N+2s}}.
\end{equation}
Thus, combining \eqref{eq:L2HsEstLPh} and \eqref{eq:BounbOpertestFunctions}, \eqref{claim-smooth} follows.

\vspace{2mm}

\noindent $\bullet$  \emph{Assume now that $L_\Omega$ is given by \eqref{new-operator}-\eqref{new-kernel}-\eqref{new-kernel2}.}
For $x \in B_{2R}(x_0)$ we have
\[
\begin{aligned}
L_\Omega \varphi (x) &= PV \int_{B_{d/2}(x)} (\varphi(x) - \varphi(y)) K_\Omega(x,y) dy + \int_{\Omega\setminus B_{d/2}(x)} (\varphi(x) - \varphi(y)) K_\Omega(x,y) dy \\
&:= I + J,
\end{aligned}
\]
and
\[
I = c_{N,s} PV \int_{B_{d/2}(x)} \frac{\varphi(x) - \varphi(y)}{|x-y|^{N+2s}} dy + PV \int_{B_{d/2}(x)} (\varphi(x) - \varphi(y)) k_\Omega(x,y) dy.
\]
Exactly as above, the first integral is bounded, by symmetry.
Moreover, thanks to Proposition \ref{kernel-estimates}, in $B_{d/2}(x)$ we have $|k_\Omega(x,y)|\leq Cd^{-N-2s}$, and thus since $\varphi$ is Lipschitz we deduce that
\[|I|\leq C(1 + d^{1-2s}(x)).\]

On the other hand, using \eqref{new-kernel-estimates} and the fact that $\varphi$ is Lipschitz, it is not difficult to see that
\[
|J| \leq C \int_{\Omega\setminus B_{d/2}(x)} |x-y| \, \frac{ 1 + \log^-\left( \frac{d_{x,y}}{|x-y|}\right) }{|x-y|^{N+2s}} dy  \leq C (1 + |\log d(x)|) (1+d(x)^{1-2s}).
\]
Therefore,
\begin{equation}\label{eq:L2HsEstLPh2}
|L_\Omega \varphi(x)| \leq C(1 + |\log d(x)|)(1 + d^{1-2s}(x)), \quad x \in B_{2R}(x_0).
\end{equation}

Finally, a similar computation shows that for $x\in B_{2R}^c(x_0)$ we have
\begin{equation}\label{eq:BounbOpertestFunctions2}
|L_\Omega\varphi(x)| \leq \int_\Omega |\varphi(y)| K_\Omega(x,y) dy \leq C \int_{\mathrm{supp} \varphi} K_\Omega(x,y)\,dy \leq \frac{C|\log d(x)|}{(1 + |x|)^{N+2s}},
\end{equation}
and thus \eqref{claim-smooth} follows.
\end{proof}

We can now proceed with the blow-up argument.

\begin{prop} \label{prop-blowup}
Let $\Omega \subset \mathbb{R}^N$ be a bounded $C^1$ domain, $s>\frac12$, and  $f \in L^q(\Omega)$ with  $q > N$.
Let $L_\Omega$ and $K_\Omega$ be given by either \eqref{new-operator}-\eqref{new-kernel}-\eqref{new-kernel2}, or \eqref{operator-reg}.
Assume that $u \in H_K(\Omega)$ is a weak solution to
\[
L_\Omega u = f \quad \text{ in } \Omega,
\]
with Neumann conditions on $\partial\Omega$ in the sense of Definition \ref{Def:WeakSolNeumann}.

Then, there exist $C$ and $\gamma>0$, depending only on $N$, $s$, $q$ and $\Omega$, such that for any $z \in \partial\Omega$ and $x \in \Omega$, we have
\begin{equation}\label{eq:MainEstBound1Stat}
|u(x) - u(z)| \leq C|x-z|^{2s-1+\gamma} \left[\|u\|_{L^\infty(\Omega)} + \|f\|_{L^q(\Omega)}\right].
\end{equation}
In particular, for any $z \in \partial\Omega$,
\begin{equation}\label{eq:FracNeumannDerBound1}
\lim_{\lambda \to 0^+} \frac{u(z) - u(z - \lambda \nu(z))}{\lambda^{2s-1}} = 0,
\end{equation}
where $\nu(z)$ denotes the exterior unit normal to $\partial\Omega$ at $z$.
\end{prop}
\begin{proof}
Recall that, thanks to Proposition \ref{Prop:GlobalL2Linf}, we have $u \in L^\infty(\Omega)$.
So, dividing $u$ by a constant if necessary, we may assume that $\|u\|_{L^\infty(\Omega)} + \|f\|_{L^q(\Omega)}\leq 1$,
and \eqref{eq:MainEstBound1Stat} can be written as
\begin{equation}\label{eq:MainEstBound11}
|u(x) - u(z)| \leq C|x-z|^{2s-1+\gamma},
\end{equation}
for all $x \in \Omega$ and $z \in \partial\Omega$.
Now, we prove \eqref{eq:MainEstBound11} with a blow-up and contradiction argument, for some $\gamma>0$ small enough, to be chosen later.

Assume by contradiction that there are sequences:

\smallskip

\noindent  $\bullet$  $(u_k)_{k\in\mathbb{N}}$ and $(f_k)_{k\in\mathbb{N}}$ of weak solutions to $L_\Omega u_k = f_k$ in $\Omega$ with Neumann conditions on $\partial\Omega$,  satisfying $\|u_k\|_{L^\infty(\Omega)} + \|f_k\|_{L^q(\Omega)} \leq 1$ for all $k \in \mathbb{N}$,

\noindent  $\bullet$  $(x_k)_{k\in\mathbb{N}} \in \Omega$ and $(z_k)_{k\in\mathbb{N}} \in \partial\Omega$,

\noindent  $\bullet$  and $C_k \to +\infty$ as $k \to +\infty$, such that
\begin{equation}\label{eq:Seqs1Blowup1}
\frac{|u_k(x_k) - u_k(z_k)|}{|x_k - z_k|^\sigma} \geq C_k,
\end{equation}
where $\sigma := 2s-1+\gamma$.

\smallskip

It follows $|x_k-z_k| \to 0$ as $k\to+\infty$ and so, up to passing to a subsequence, $x_k,z_k \to z_0$ as $k\to+\infty$, for some suitable $z_0 \in \partial\Omega$.

Now, the function
\[
\vartheta(r) := \sup_{k\in\mathbb{N}} \, \vartheta_k(r) := \sup_{k\in\mathbb{N}} \; \max_{\varrho \geq r} \varrho^{-\sigma} \|u_k - u_k(z_k)\|_{L^\infty(B_\varrho(z_k))}
\]
is clearly monotone non-increasing and, thanks to \eqref{eq:Seqs1Blowup1}, it satisfies $\vartheta(r) \to +\infty$ as $r \to 0^+$, that is
\begin{equation}\label{eq:HpAssurda1}
\sup_{k\in\mathbb{N}} \; \sup_{r > 0} r^{-\sigma} \|u_k - u_k(z_k)\|_{L^\infty(B_r(z_k))} = + \infty.
\end{equation}
Indeed, choosing $r_k = |x_k - z_k|$, we have
\[
\vartheta_k(r_k) \geq r_k^{-\sigma} \|u_k - u_k(z_k)\|_{L^\infty(B_{r_k}(z_k))} \geq \frac{|u_k(x_k) - u_k(z_k)|}{|x_k - z_k|^\sigma},
\]
and thus, in view of \eqref{eq:Seqs1Blowup1}, we can pass to the limit as $k \to + \infty$ and \eqref{eq:HpAssurda1} follows.

Furthermore, by the definition of $\vartheta$ we deduce the existence of two sequences $r_j \to 0^+$ and $(k_j)_{j\in\mathbb{N}}$ such that
\begin{equation}\label{eq:ChoiceRj1}
r_j^{-\sigma} \|u_{k_j} - u_{k_j}(z_{k_j})\|_{L^\infty(B_{r_j}(z_{k_j}))} \geq \frac{\vartheta(r_j)}{2}, \quad j \in \mathbb{N}.
\end{equation}

\emph{Step 1: Blow-up sequence.} Now, we introduce the blow-up sequence
\[
v_j(x) := \frac{u_{k_j}(r_jx + z_{k_j}) - u_{k_j}(z_{k_j})}{r_j^\sigma \vartheta(r_j)}, \quad j \in \mathbb{N},
\]
which satisfies $v_j(0) = 0$ for all $j \in \mathbb{N}$ and
\begin{equation}\label{eq:NormLinfty1}
\|v_j\|_{L^\infty(B_1)} \geq \frac{1}{2}, \quad \text{ for all } j \in \mathbb{N},
\end{equation}
thanks to \eqref{eq:ChoiceRj1}. Further, for any $R \geq 1$, we have
\[
\|v_j\|_{L^\infty(B_R)} = \frac{1}{r_j^\sigma \vartheta(r_j)} \|u_{k_j} - u_{k_j}(z_{k_j})\|_{L^\infty(B_{r_jR}(z_{k_j}))} \leq \frac{1}{r_j^\sigma \vartheta(r_j)} (r_jR)^\sigma \vartheta(r_jR) \leq R^\sigma,
\]
where we have used the definition of $\vartheta$ and its monotonicity: $\vartheta(r_jR) \leq \vartheta(r_j)$ for $j \in \mathbb{N}$ and all $R \geq 1$. Thus:
\begin{equation}\label{eq:BoundGrowthvj1}
\|v_j\|_{L^\infty(B_R)} \leq R^\sigma, \quad j \in \mathbb{N}, \; R \geq 1.
\end{equation}
On the other hand, each $v_j$ satisfies
\begin{equation}\label{eq:EqOfvj1}
L_jv_j(x) = \frac{r_j^{2s-\sigma}}{\vartheta(r_j)} f(r_jx + z_{k_j}) := \widetilde{f}_j(x), \quad x \in \Omega_j := r_j^{-1}(z_{k_j}-\Omega),
\end{equation}
in the weak sense with Neumann conditions on $\partial\Omega_j$, where $L_j := L_{\Omega_j}$, and
\begin{equation}\label{eq:BoundrhsEqvj1}
\|\widetilde{f}_j\|_{L^q(\Omega_j)} \leq \|f\|_{L^q(\Omega)} \frac{r_j^{2s - \frac{N}{q} -\sigma}}{\vartheta(r_j)}, \quad \text{ for all } j \in \mathbb{N}.
\end{equation}
Now, fix $R \geq 1$ and define $w_j := v_j \chi_{B_{4R}}$, $j \in \mathbb{N}$. Following the proof of Lemma \ref{Lemma:LiouvilleND} and setting $D_R^j := B_R \cap \Omega_j$, it is not difficult to verify that
\[
L_{\Omega_j}w_j = \overline{f}_j \quad \text{ in } D_{2R}^j,
\]
where
\[
\overline{f}_j := \widetilde{f}_j + 2\int_{\Omega_j\setminus B_{4R}} v_j(y)\,K_{\Omega_j}(x,y) dy.
\]
Using \eqref{eq:BoundrhsEqvj1} and that $q > N$, we can choose $\gamma>0$ small enough so that $2s - N/q - \sigma > 0$, and thus $\|\widetilde{f}_j\|_{L^q(\Omega_j)}$ is uniformly bounded. Further, using \eqref{eq:BoundGrowthvj1} and repeating the proof of Lemma \ref{Lemma:LiouvilleND}, we find that also the second term in the definition of $\overline{f}_j$ is bounded in $L^q(D_{2R}^j)$, uniformly w.r.t. $j$ (recall that we can reduce consider the case $\Omega_j = \mathbb{R}_+^N$ by using a local bi-Lipschitz transformation of $\Omega_j$). In particular, $\overline{f}_j$ is bounded in $L^q(D_{2R}^j)$, uniformly w.r.t. $j$ and thus Theorem \ref{Theorem:HolderRegularity1} implies
\[
[w_j]_{C^\alpha(D_R^j)} \leq C R^{-\alpha} \left[ \|w_j\|_{L^{\infty}(\Omega_j)} + R^{2s - \frac{N}{q}} \|\overline{f}_j\|_{L^q(D_{2R}^j)}\right].
\]
By the argument above and since $\|w_j\|_{L^{\infty}(\Omega_j)} = \|v_j\|_{L^{\infty}(\Omega_j\cap B_{4R})} \leq C R^\sigma$ (see \eqref{eq:BoundGrowthvj1}), it follows that $[w_j]_{C^\alpha(D_R^j)} \leq C_R$ for all $j \in \mathbb{N}$ and some constant $C_R > 0$ (independent of $j$). In particular, since $w_j = v_j$ in $D_R^j$, we obtain
\begin{equation}\label{eq:HolBoundvj1}
[v_j]_{C^\alpha(D_R^j)} \leq C_R.
\end{equation}
Moreover, choosing $\gamma>0$ small enough so that $\sigma < s$, we combine Lemma \ref{Lemma:LocalHsBound}, \eqref{eq:BoundGrowthvj1} and \eqref{eq:BoundrhsEqvj1}, to deduce
\begin{equation}\label{eq:LocHsBoundBlowup1}
\begin{aligned}
\left[v_j\right]_{H^s(D_R^j)}^2 \leq C_R,
\end{aligned}
\end{equation}
for any fixed $R \geq 1$ and some new constant $C_R > 0$ independent of $j \in \mathbb{N}$.

\emph{Step 2: Compactness.} Using simultaneously \eqref{eq:BoundGrowthvj1}, \eqref{eq:HolBoundvj1}, the fact that $\Omega$ is of class $C^1$ together with $z_{k_j} \to z_0 \in \partial\Omega$, and the Ascoli-Arzel\`a theorem, it follows that for any $R \geq 1$ and any $\nu \in (0,\alpha)$,
\[
v_j \to v,
\]
uniformly in $\overline{B_R \cap H}$ (and in $C^\nu$), where $H := \{e \cdot x > 0\}$, for some unit vector $e$ depending on $z_0 \in \partial\Omega$.
Moreover,  $v \in C^\nu(\overline{B_R \cap H})$ and $v(0) = 0$. Further, in view of \eqref{eq:LocHsBoundBlowup1}, the sequence $\{v_j\}_{j\in\mathbb{N}}$ is uniformly bounded in $H_{K,loc}(\overline{\Omega}_j)$ and so $v \in H_{K,loc}(\overline{H})$.

Notice also that by uniform convergence, we obtain that $v$ satisfies
\begin{equation}\label{eq:BoundGrowthv1}
\|v\|_{L^\infty(B_1)} \geq \frac{1}{2}, \qquad \|v\|_{L^\infty(B_R)} \leq R^\sigma, \quad \text{ for all } R \geq 1,
\end{equation}
once we pass to the limit in \eqref{eq:NormLinfty1} and \eqref{eq:BoundGrowthvj1}.

\emph{Step 3: Passage to the limit into the equation.} Since the $v_j$'s satisfy \eqref{eq:EqOfvj1} in the weak sense with Neumann conditions on $\partial\Omega_j$, they satisfy the same equation in the distributional sense, that is
\begin{equation}\label{eq:DisSol1}
\int_{\Omega_j} v_j L_j\eta dx = \frac{1}{2} \int_{\Omega_j} \widetilde{f}_j \eta,
\end{equation}
for all $\eta \in C_0^\infty(\mathbb{R}^N)$, and all $j \in \mathbb{N}$. To justify this, we fix $\eta \in C_0^\infty(\mathbb{R}^N)$, $j \in \mathbb{N}$, $\varepsilon \in (0,1)$ and we notice that, by the symmetry of the kernel, we have
\begin{equation}\label{eq:DisSolEps1}
\begin{aligned}
\int_{\Omega_j} v_j(x) &\int_{\Omega_j\setminus B_\varepsilon(x)} [\eta(x) - \eta(y)] K_{\Omega_j}(x,y)dy dx = \iint_{D_j^\varepsilon} v_j(x) [\eta(x) - \eta(y)] K_{\Omega_j}(x,y)dxdy \\
&= \frac{1}{2} \iint_{D_j^\varepsilon} [v_j(x) - v_j(y)][\eta(x) - \eta(y)] K_{\Omega_j}(x,y) dxdy,
\end{aligned}
\end{equation}
where $D_j^\varepsilon := \left\{(x,y) \in \Omega_j \times \Omega_j: |x-y|>\varepsilon \right\}$. For any $x \in \Omega_j$, we define
\[
L_j^\varepsilon \eta (x) := \int_{\Omega_j\setminus B_\varepsilon(x)} [\eta(x) - \eta(y)]K_{\Omega_j}(x,y) dy.
\]
Notice that $L_j^\varepsilon\eta \to L_j\eta$ a.e. in $\mathbb{R}^N$ as $\varepsilon \to 0^+$ and
\begin{equation}\label{eq:NeumUnifBoundVepj}
|L_j^\varepsilon \eta(x)| \leq \frac{h_j(x)}{(1+|x|)^{N+2s}},
\end{equation}
for some $h_j \in L^1_{loc}(\mathbb{R}^N)$ independent of $\varepsilon \in (0,1)$; see  \eqref{eq:L2HsEstLPh}-\eqref{eq:BounbOpertestFunctions} and \eqref{eq:L2HsEstLPh2}-\eqref{eq:BounbOpertestFunctions2}  in the proof of Lemma~\ref{Lemma:LocalHsBound}. Noticing that the function $x \to (1+|x|)^{-N-\alpha} h_j(x)$ belongs to $L^1(\mathbb{R}^N)$ for any $\alpha > 0$, recalling
\eqref{eq:BoundGrowthvj1} and that $\sigma < s$, we can pass to the limit into \eqref{eq:DisSolEps1} to obtain
\[
\int_{\Omega_j} v_j(x) \int_{\Omega_j\setminus B_\varepsilon(x)} [\eta(x) - \eta(y)] K_{\Omega_j}(x,y)dy dx \to \int_{\Omega_j} v_j L_j\eta dx
\]
as $\varepsilon \to 0$, thanks to the dominated convergence theorem. On the other hand, since $D_j^\varepsilon \to \Omega_j\times\Omega_j$, we find
\[
\iint_{D_j^\varepsilon} [v_j(x) - v_j(y)][\eta(x) - \eta(y)] K_{\Omega_j}(x,y) dxdy \to B_{\Omega_j}(v_j,\eta) = \int_{\Omega_j} \widetilde{f}_j \eta,
\]
and so, in view of \eqref{eq:DisSolEps1}, \eqref{eq:DisSol1} is proved.

Now, we fix an arbitrary $\eta \in C_0^\infty(\mathbb{R}^N)$ and we pass to the limit as $j \to +\infty$ in \eqref{eq:DisSol1}. Using \eqref{eq:BoundrhsEqvj1} and that $2s - N/q - \sigma > 0$, the right hand side of the equation converges to $0$ as $j\to +\infty$.  Further, using that $\chi_j \to \chi_H$ and $K_{\Omega_j} \to K_H$ a.e. in $\mathbb{R}^N$, we apply the Vitali's convergence theorem (here we use again \eqref{eq:L2HsEstLPh}-\eqref{eq:BounbOpertestFunctions} and \eqref{eq:L2HsEstLPh2}-\eqref{eq:BounbOpertestFunctions2}), to deduce $L_j\eta \to L_H\eta$ a.e. in $\mathbb{R}^N$. Writing
\[
\bigg| \int_{\Omega_j} v_j L_j\eta dx - \int_H v L_H \eta dx \bigg| \leq \bigg| \int_{\Omega_j} v_j (L_j\eta - L_H\eta) dx \bigg| + \bigg|\int_H (v_j - v) L_H \eta dx \bigg| := I_j + \overline{I}_j,
\]
we easily see that both $I_j$ and $\overline{I}_j$ go to $0$ as $j \to +\infty$. Indeed, since $L_j\eta \to L_H\eta$, the $v_j$'s satisfy \eqref{eq:BoundGrowthvj1} and $\sigma < 2s$, we obtain $I_j \to 0$ as $j \to +\infty$, applying the Vitali's convergence theorem again. Similar for $\overline{I}_j$, using that $v_j \to v$ uniformly on compact sets of $\mathbb{R}^N$.

As a consequence, we can pass to the limit and deduce that $v$ satisfies
\begin{equation}\label{eq:DisSolL1}
\int_H v \,L_H \eta dx = 0, \quad \text{ for all } \eta \in C_0^\infty(\mathbb{R}^N).
\end{equation}
From interior regularity estimates and \eqref{eq:LocHsBoundBlowup1}, we know that $v \in C^\infty(H)\cap H_{K,loc}(\overline{H})$ and thus $v$ is a weak solution to
\begin{equation}\label{eq:WeakSolH1}
L_Hv = 0 \quad \text{ in } H,
\end{equation}
with Neumann conditions on $\partial H$ in the sense of Definition \ref{Def:WeakSolNeumann}. Indeed, let $\eta \in C_0^\infty(\mathbb{R}^N)$ and set
\[
L_H^\varepsilon \eta(x) := \int_{H \setminus B_\varepsilon(x)} [\eta(x) - \eta(y)] K_H(x,y) \, dy.
\]
By \eqref{eq:DisSolEps1}, we have
\begin{equation}\label{eq:DisSolEpsH1}
\int_H v(x) L_H^\varepsilon \eta(x) dx = \frac{1}{2} \iint_{D^\varepsilon} [v(x) - v(y)][\eta(x) - \eta(y)] K_H(x,y) dxdy,
\end{equation}
where $D^\varepsilon := \left\{(x,y) \in H \times H: |x-y|>\varepsilon \right\}$. Now, proceeding as above, it follows
\[
\int_H v(x) L_H^\varepsilon \eta(x) dx \to \int_H v(x) L_H \eta(x) dx,
\]
as $\varepsilon \to 0^+$ and so, in view of \eqref{eq:DisSolL1} and the fact that $D^\varepsilon \to H \times H$ as $\varepsilon \to 0^+$, we obtain
\[
\int_H \int_H [v(x) - v(y)][\eta(x) - \eta(y)] K_H(x,y) dxdy = 0.
\]
Recalling that $v \in H_{K,loc}(\overline{H})$, \eqref{eq:WeakSolH1} follows.

\emph{Step 4: Conclusion.} In view of \eqref{eq:BoundGrowthv1} and Theorem \ref{Thm:LiouvilleND2}, we deduce that $v$ is constant in $H$. On the other hand, recalling that $v(0) = 0$, it must be $v\equiv 0$ in $H$, a contradiction with \eqref{eq:BoundGrowthv1}.
\end{proof}

We will also need the following observation.

\begin{lem} \label{lem-obs}
Let $\Omega \subset \mathbb{R}^N$ be a bounded $C^1$ domain, $\sigma\in (0,2s)$, and assume that $u$ satisfies:
\begin{itemize}

\item $|u|\leq C_0$ in $\Omega$,

\item $\mathcal N_s u=0$ in $\Omega^c$,

\item $|u(x) - u(z)| \leq C_0|x-z|^{\sigma}$ for all $z\in \partial\Omega$, $x\in \Omega$.
\end{itemize}

Then, we have
\begin{equation}\label{trew}
|u(x) - u(z)| \leq CC_0|x-z|^{\sigma}\quad \textrm{for all}\quad z\in \partial\Omega, \quad x\in \R^N.
\end{equation}
The constant $C$ depends only on $\Omega$.
\end{lem}

\begin{proof}
Notice that, since $\mathcal N_s u=0$ in $\Omega^c$, then
\[u(x) \int_\Omega \frac{dy}{|x-y|^{N+2s}} = \int_\Omega \frac{u(y)}{|x-y|^{N+2s}}\,dy,\]
for all $x\in \Omega^c$, and thus
\[(u(x)-u(z)) \int_\Omega \frac{dy}{|x-y|^{N+2s}} = \int_\Omega \frac{u(y)-u(z)}{|x-y|^{N+2s}}\,dy,\]
for any $z\in \partial\Omega$.

When $d(x)>1$ the bound \eqref{trew} holds trivially, so we will assume $d(x)\leq1$.
In that case, by \cite[Lemma 2.1]{Ab} we have
\[\int_\Omega \frac{dy}{|x-y|^{N+2s}} \asymp d^{-2s}(x).\]
Moreover, since $\Omega$ is $C^1$,  choosing $z$ to be the projection of $y$ onto $\partial\Omega$, we have
\[\int_\Omega \frac{|u(y)-u(z)|}{|x-y|^{N+2s}}\,dy \leq C \int_\Omega \frac{|y-z|^\sigma}{|x-y|^{N+2s}}\,dy
\leq C\int_\Omega \frac{|y-z|^\sigma}{(d(x)+|y-z|)^{N+2s}}\,dy, \]
for some $C >0$ depending on $\Omega$.
Since
\[\int_{\R^N} \frac{|y-z|^\sigma}{(A+|y-z|)^{N+2s}}\,dy \asymp A^{\sigma-2s},\]
we deduce
\[\int_\Omega \frac{|u(y)-u(z)|}{|x-y|^{N+2s}}\,dy\leq Cd^{\sigma-2s}(x)=C|x-z|^{\sigma-2s}.\]
Combining the previous estimates, the result follows.
\end{proof}

Finally, to prove Theorems \ref{thm-intro} and \ref{thm-intro-reg}, we will also need the following interior regularity results.
The first one is probably well known, we give a short proof for completeness.

\begin{lem} \label{lemma:Interior-Regularity-frac}
	Let $N\geq 2$ and $s>\frac12$.
	Assume that $u\in L^\infty(B_1)$, $(1+|x|)^{-N-2s}u(x)\in L^1(\Omega)$,  satisfies
	$$ (-\Delta)^s u = f \ \ \ \text{ in } \ B_1,$$
	for some $f\in L^q(B_1)$ with $q>N/(2s)$.
	Then, for any $\gamma \leq 2s-N/q$,
	$$ \|u\|_{C^\gamma (B_{1/2})} \leq C (||f||_{L^q(B_1)}+||(1+|x|)^{-N-2s}u(x)||_{L^1(\R^N)}+||u||_{L^\infty(B_1)}),  $$
	where  $C$ is a positive constant depending only on $N$, $s$, $q$ and $\gamma$.
\end{lem}

\begin{proof}
We can decompose $u=v+w$, where $v = (-\Delta)^{-s} f$ (in the sense that $v$ is the Riesz potential of order $2s$ of the function $f$ extended by zero outside $B_1$) and $w$ satisfies $(-\Delta)^s w=0$ in $B_1$.
Then, if we apply the estimates in \cite[Theorem 1.6 (ii)]{RS-extremal} and \cite[Corollary~2.5]{RS-Dir}, we get
	$$ [v]_{C^\gamma(\RR^N)} \leq C ||f||_{L^q(B_1)}, \ \ \ \ \ ||(1+|x|)^{-N-2s}v(x)||_{L^1(\R^N)} \leq C ||f||_{L^q(B_1)}, $$
	and
	$$ [w]_{C^{\gamma}(B_{1/2})} \leq C (||(1+|x|)^{-N-2s}w(x)||_{L^1(\Omega)}+||w||_{L^\infty(B_2)}). $$
	The result then follows from these estimates.
\end{proof}

The second one is for the regional fractional Laplacian.

\begin{lem} \label{lemma:Interior-Regularity}
	Let $\Omega\subset \RR^N$ be any domain with $N\geq 2$ and $s>\frac12$. Let $L_\Omega$ be given by \eqref{operator-reg}.
	Assume that $u\in L^\infty(B_2)$, $(1+|x|)^{-N-2s}u(x)\in L^1(\Omega)$ and satisfies
	$$ L_\Omega u = f \ \ \ \text{ in } \ B_3\subset \Omega,$$
	for some $f\in L^q(B_3)$ with $q>N/(2s)$.
	Then, for any $\gamma \leq 2s-N/q$,
	$$ [u]_{C^\gamma (B_{1/2})} \leq C (||f||_{L^q(B_2)}+||(1+|x|)^{-N-2s}u(x)||_{L^1(\Omega)}+||u||_{L^\infty(B_2)}),  $$
	where  $C$ is a positive constant depending only on $N$, $s$, $q$ and $\gamma$.
\end{lem}

\begin{proof}
Extend $u$ to be zero outside $\Omega$.
	Then, for any $x\in B_2$, it is clear that
	$$ (-\Delta)^s u(x) = L_\Omega u(x) + u(x) \int_{\Omega^c} |x-y|^{-N-2s} = f(x) + u(x) \int_{\Omega^c} |x-y|^{-N-2s}\, dy =:g(x). $$
	Moreover,
	$$ |g|\leq |f|+C|u|\int_{B_3^c} |y|^{-N-2s}\, dy \leq |f|+C|u|, $$
	which means that $||g||_{L^q(B_2)} \leq C (||f||_{L^q(B_2)} + ||u||_{L^\infty(B_2)})$.
	
	Hence, $u$ satisfies
	$$ (-\Delta)^s u = g \ \ \ \text{ in } \ B_2\subset \Omega,$$
	for some $g\in L^q(B_2)$ with norm depending only on $N$, $s$ and $f$.
	The result then follows from Lemma \ref{lemma:Interior-Regularity-frac}.
\end{proof}

We can now give the:

\begin{proof} [Proof of Theorem~\ref{thm-intro}]
	We divide the proof in two steps:
	
	\vspace{2mm}
	
	\emph{Step 1: $C^\alpha$ estimate.} Since $\Omega$ is bounded and Lipschitz, it can be covered with a finite number of balls in such way that $\partial \Omega \cap B$ is a Lipschitz graph for any ball $B$. Consequently, combining the interior estimate of Lemma~\ref{lemma:Interior-Regularity-frac} and the boundary one of Theorem~\ref{Theorem:HolderRegularity1}, we deduce
	$$|u(x)-u(y)| \leq C \left(\|f\|_{L^q(\Omega)} + \|u\|_{L^2(\Omega)}\right)|x-y|^{\alpha}$$
	for every $x,y\in \overline{\Omega}$ with $\alpha$ and $C$ depending only on $N$, $s$, $q$ and $\Omega$.
	
	\medskip
	
	\emph{Step 2: $C^{2s-1+\alpha}$ estimate for $s>\frac{1}{2}$.} Dividing $u$ by a constant if needed, we may assume $\|f\|_{L^q(\Omega)} + \|u\|_{L^2(\Omega)}\leq 1$.
	Now, given $x,y\in \overline{\Omega}$, we define $r=|x-y|$ and $\rho=\min\{d(x),d(y)\}$ and,
	without loss of generality, we assume $\rho=d(x)$. We divide the proof in two cases.

On the one hand, when $\rho \leq 2r$, we take $z\in \partial \Omega$ such that $|z-x|=\rho$ and, using Proposition~\ref{prop-blowup}, we conclude
	\begin{align*}
	|u(x)-u(y)|
	&\leq |u(x)-u(z)|+|u(y)-u(z)| \leq C \left( |x-z|^{2s-1+\alpha} + |y-z|^{2s-1+\alpha} \right) \\ &\leq C \left( d(x)^{2s-1+\alpha} + (d(x)+r)^{2s-1+\alpha} \right) \leq C\,r^{2s-1+\alpha} = C \,|x-y|^{2s-1+\alpha},
	\end{align*}
	for some $\alpha>0$ small enough.

On the other, if $\rho > 2r$ we have $B_{2r}(y) \subset \Omega$. We define the auxiliary function $u_r(x) = u(y + r x) - u(y)$ and the set $\Omega_r := (\Omega-x)/r$. Then, it is clear that $u_r$ satisfies
	$$ L_{\Omega_r} \, u_r (x) = r^{2s} f(y+r x) =:f_r (x) \ \ \ \text{ in } \  B_2, $$
	with $ ||f_r||_{L^q(B_2)} \leq Cr^{2s-N/q}$.
	Moreover, by using Proposition~\ref{prop-blowup} and Lemma \ref{lem-obs} we know that $  |u_r(x)| < C|r x|^{2s-1+\alpha}$ for some $\alpha$ small enough, which yields
	$$ ||u_r||_{L^\infty(B_2)} < Cr^{2s-1+\alpha} \ \ \ \ \text{ and } \ \ \ ||(1+|x|)^{-N-2s}u_r(x)||_{L^1(\R^N)} < Cr^{2s-1+\alpha}. $$
	Furthermore, since $q>N$, we can take $\alpha$ small enough such that $2s-N/q>2s-1+\alpha$. Thus, applying Lemma~\ref{lemma:Interior-Regularity-frac} with $\gamma = 2s-1+\alpha$, we arrive at
	\begin{align*}
		[u_r]_{C^{2s-1+\alpha} (B_1)} &\leq C (||f_r||_{L^q(B_2)}+||(1+|x|)^{-N-2s}u_r(x)||_{L^1(\R^N)}+||u_r||_{L^\infty(B_2)}) \\
		&\leq C (r^{2s-N/q} + r^{2s-1+\alpha} + r^{2s-1+\alpha}) \leq Cr^{2s-1+\alpha},
	\end{align*}
which is equivalent to say
\[
[u]_{C^{2s-1+\alpha} (B_r(y))} \leq C,
\]
for some constant independent of $y$ and $r$. Consequently,
\[
\begin{aligned}
|u(x) - u(y)| &= r^{2s-1+\alpha} \frac{|u(x) - u(y)|}{|x-y|^{2s-1+\alpha}} \leq r^{2s-1+\alpha} \sup_{x \in B_r(y)} \frac{|u(x) - u(y)|}{|x-y|^{2s-1+\alpha}} \\
&\leq r^{2s-1+\alpha} [u]_{C^{0,2s-1+\alpha}(B_r(y))} \leq C r^{2s-1+\alpha} = C |x-y|^{2s-1+\alpha}.
\end{aligned}
\]
Since $x,y \in \overline{\Omega}$ have been arbitrarily chosen, the thesis follows.
\end{proof}

Finally, we give the:

\begin{proof} [Proof of Theorem~\ref{thm-intro-reg}]
	The proof is basically the same as the previous one, applying Lemma~\ref{lemma:Interior-Regularity} instead of Lemma~\ref{lemma:Interior-Regularity-frac}.
\end{proof}

\appendix
\section{Equivalence for weak solutions}

For completeness, we prove here the equivalence established in \cite{Ab} for classical solutions, in the setting of weak solutions.

\begin{prop}
	Let $u\in C^2(\RR^N)\cap L^\infty(\RR^N)$ be such that
	$$ \begin{cases}
	(-\Delta)^s u = f \ \ \ \text{ in } \ \Omega,\\
	\mathcal{N}_s u = 0 \ \ \ \text{ in } \ \RR^N\setminus \overline{\Omega}.
	\end{cases} $$
	Then, it satisfies
	$$ \int_\Omega \left\{ u(x)-u(y)\right\}\,K_\Omega(x,y) \, dy = f(x) \ \ \ \text{ in } \ \Omega, $$
	where $K_\Omega$ is given by \eqref{new-kernel}-\eqref{new-kernel2}.
\end{prop}

\begin{proof}
	Given any $z\in \Omega^c$, we have
	\begin{align*}
		0 &= \mathcal{N}_s u(z) = \int_\Omega \frac{u(z)-u(y)}{|z-y|^{N+2s}}\,dy \\
		&= u(z) \int_\Omega |z-y|^{-N-2s} \, dy - \int_\Omega \frac{u(y)}{|z-y|^{N+2s}} \, dy,
	\end{align*}
	and so
	$$ u(z) = \frac{\int_\Omega u(y) |z-y|^{-N-2s} \, dy}{ \int_\Omega |z-\overline{z}|^{-N-2s} \, d\overline{z}}\ \ \ \text{ in } \ \RR^N\setminus \overline{\Omega}. $$
	Now, we substitute this identity in the fractional Laplacian. Given any $x\in \Omega$
	\begin{align*}
		\frac{(-\Delta)^s u (x)}{c_{N,s}} &= \int_{\RR^N} \frac{u(x)-u(y)}{|x-y|^{N+2s}} \, dy = \int_{\Omega} \frac{u(x)-u(y)}{|x-y|^{N+2s}} \, dy + \int_{\Omega^c} \frac{u(x)-u(z)}{|x-z|^{N+2s}} \, dz \\ &= \int_{\Omega} \frac{u(x)-u(y)}{|x-y|^{N+2s}} \, dy + \int_{\Omega^c} \frac{u(x)-\frac{\int_\Omega u(y) |z-y|^{-N-2s} \, dy}{ \int_\Omega |z-\overline{z}|^{-N-2s} \, d\overline{z}}}{|x-z|^{N+2s}} \, dz \\
		&= \int_{\Omega} \frac{u(x)-u(y)}{|x-y|^{N+2s}} \, dy + \int_{\Omega^c} \frac{\int_\Omega \frac{ u(x)-u(y)}{ |z-y|^{N+2s}} \, dy}{|x-z|^{N+2s}\int_\Omega |z-\overline{z}|^{-N-2s} \, d\overline{z}} \, dz \\
		&= \int_{\Omega} \frac{u(x)-u(y)}{|x-y|^{N+2s}} \, dy + \int_{\Omega} \left\{u(x)-u(y)\right\} \int_{\Omega^c} \frac{|x-z|^{-N-2s} |y-z|^{-N-2s} }{\int_\Omega |z-\overline{z}|^{-N-2s} \, d\overline{z}}\, dz \, dy,
	\end{align*}
and the result follows.
\end{proof}

In what follows, we denote
$$ ||w||_{H^s_\Omega}^2 :=  ||w||_{L^2(\Omega)}^2 + \int \int_{(\RR^N\times\RR^N)\setminus (\Omega^c\times \Omega^c)} \frac{|w(x)-w(y)|^2}{|x-y|^{N+2s}} dx\, dy.$$

\begin{lem} \label{Lemma: FromFracLaplToRestWithKErnel}
	Let $v,w:\RR^N \to \RR$ be such that $\mathcal{N}_s w = 0$ in $\RR^N\setminus \Omega$. Then,
	\begin{align*}
	\int_\Omega \int_{\Omega} &\left\{v(x)-v(y)\right\}\left\{w(x)-w(y)\right\} \, K_\Omega(x,y) dx\, dy \\
	&= c_{N,s}\int \int_{(\RR^N\times\RR^N)\setminus (\Omega^c\times \Omega^c)} \frac{\left\{v(x)-v(y) \right\}\left\{w(x)-w(y) \right\}}{|x-y|^{N+2s}} dx\, dy,
	\end{align*}
	where  $K_\Omega$ is given by \eqref{new-kernel}-\eqref{new-kernel2}.
\end{lem}

\begin{proof}
Note that adding and subtracting the terms $w(z)(v(x)+v(y)+v(z))$ and $v(z)(w(x)+w(y)+w(z))$, and rearranging them, we obtain
	\begin{align*}
	\int_\Omega \int_{\Omega} &\left\{v(x)-v(y)\right\}\left\{w(x)-w(y)\right\} \, K_\Omega(x,y) dx\, dy \\
	&= c_{N,s}\int_\Omega \int_{\Omega} \frac{\left\{v(x)-v(y) \right\}\left\{w(x)-w(y) \right\}}{|x-y|^{N+2s}} dx\, dy \\
	&\hspace{5mm}+ c_{N,s}\int_\Omega dx \int_\Omega dy \int_{\Omega^c} dz \frac{\left\{v(x)-v(y) \right\}\left\{w(x)-w(y) \right\}}{|x-z|^{N+2s} |y-z|^{N+2s}\int_\Omega |z-\overline{z}|^{-N-2s} \, d\overline{z}} \\
	&= c_{N,s}\int_\Omega \int_{\Omega} \frac{\left\{v(x)-v(y) \right\}\left\{w(x)-w(y) \right\}}{|x-y|^{N+2s}} dx\, dy \\
	&\hspace{5mm}+ c_{N,s}\int_\Omega dx \int_\Omega dy \int_{\Omega^c} dz \frac{\left\{v(x)-v(z) \right\}\left\{w(x)-w(z) \right\}}{|x-z|^{N+2s} |y-z|^{N+2s}\int_\Omega |z-\overline{z}|^{-N-2s} \, d\overline{z}} \\
	&\hspace{5mm}+ c_{N,s}\int_\Omega dx \int_\Omega dy \int_{\Omega^c} dz \frac{\left\{v(z)-v(y) \right\}\left\{w(x)-w(z) \right\}}{|x-z|^{N+2s} |y-z|^{N+2s}\int_\Omega |z-\overline{z}|^{-N-2s} \, d\overline{z}} \\
	&\hspace{5mm}+ c_{N,s}\int_\Omega dx \int_\Omega dy \int_{\Omega^c} dz \frac{\left\{v(x)-v(z) \right\}\left\{w(z)-w(y) \right\}}{|x-z|^{N+2s} |y-z|^{N+2s}\int_\Omega |z-\overline{z}|^{-N-2s} \, d\overline{z}} \\
	&\hspace{5mm}+ c_{N,s}\int_\Omega dx \int_\Omega dy \int_{\Omega^c} dz \frac{\left\{v(z)-v(y) \right\}\left\{w(z)-w(y) \right\}}{|x-z|^{N+2s} |y-z|^{N+2s}\int_\Omega |z-\overline{z}|^{-N-2s} \, d\overline{z}} \\
	&= I_1+I_2+I_3+I_4+I_5.
	\end{align*}
	By symmetry in the variables $x$ and $y$ it is clear that $I_2 = I_5$ and $I_3=I_4$. Now, let us simplify them. On the one hand
	\begin{align*}
	I_2 &= I_5 = c_{N,s}\int_\Omega dx \int_\Omega dy \int_{\Omega^c} dz \frac{\left\{v(x)-v(z) \right\}\left\{w(x)-w(z) \right\}}{|x-z|^{N+2s} |y-z|^{N+2s}\int_\Omega |z-\overline{z}|^{-N-2s} \, d\overline{z}} \\
	&=c_{N,s}\int_\Omega dx \int_{\Omega^c} dz \frac{\left\{v(x)-v(z) \right\}\left\{w(x)-w(z) \right\}}{|x-z|^{N+2s} \int_\Omega |z-\overline{z}|^{-N-2s} \, d\overline{z}} \left( \int_\Omega |y-z|^{-N-2s} \, dy \right)\\
	&=c_{N,s}\int_\Omega dx \int_{\Omega^c} dz \frac{\left\{v(x)-v(z) \right\}\left\{w(x)-w(z) \right\}}{|x-z|^{N+2s}}.
	\end{align*}
	On the other hand, using the condition $\mathcal{N}_sw = 0$ in $\RR^N \setminus \overline{\Omega}$ we obtain
	\begin{align*}
	I_3 &= I_4 = c_{N,s}\int_\Omega dx \int_\Omega dy \int_{\Omega^c} dz \frac{\left\{v(z)-v(y) \right\}\left\{w(x)-w(z) \right\}}{|x-z|^{N+2s} |y-z|^{N+2s}\int_\Omega |z-\overline{z}|^{-N-2s} \, d\overline{z}} \\
	&= c_{N,s}\int_\Omega dy \int_{\Omega^c} dz \frac{v(z)-v(y)}{|y-z|^{N+2s}\int_\Omega |z-\overline{z}|^{-N-2s} \, d\overline{z}} \left(\int_\Omega \frac{w(x)-w(z)}{|x-z|^{N+2s}} dx\right) \\
	&= c_{N,s}\int_\Omega dy \int_{\Omega^c} dz \frac{-\mathcal{N}_s w (z)\left\{v(z)-v(y) \right\}}{|y-z|^{N+2s}\int_\Omega |z-\overline{z}|^{-N-2s} \, d\overline{z}}\\
	&= 0.
	\end{align*}
	Putting all the terms together we finally arrive at
	\begin{align*}
	\int_\Omega \int_{\Omega} &\left\{v(x)-v(y)\right\}\left\{w(x)-w(y)\right\} \, K_\Omega(x,y) dx\, dy \\
	&= c_{N,s}\int \int_{(\RR^N\times\RR^N)\setminus (\Omega^c\times \Omega^c)} \frac{\left\{v(x)-v(y) \right\}\left\{w(x)-w(y) \right\}}{|x-y|^{N+2s}} dx\, dy,
	\end{align*}
	as wanted.
\end{proof}

Finally, we prove:

\begin{prop}
	Let $u\in H^s_\Omega$ be such that
	\begin{equation} \label{Eq: WeakFormFracLapl}
		\frac{c_{N,s}}{2}\int \int_{(\RR^N\times\RR^N)\setminus (\Omega^c\times \Omega^c)} \frac{\left\{u(x)-u(y) \right\}\left\{v(x)-v(y) \right\}}{|x-y|^{N+2s}} dx\, dy = \int_\Omega f(x)v(x)\,dx
	\end{equation}
	for all test function $v\in H^s_\Omega$.
	Then, $u\in H_K(\Omega)$ and it satisfies
	\begin{equation} \label{Eq: WeakFormRestrWithKernel}
	\frac{1}{2}\int_\Omega \int_{\Omega} \left\{u(x)-u(y)\right\}\left\{\overline{v}(x)-\overline{v}(y)\right\} \, K_\Omega(x,y) dx\, dy = \int_\Omega f(x)\overline{v}(x)\,dx
	\end{equation}
	for all $\overline{v}\in H_K(\Omega)$, where $K_\Omega$ is given by \eqref{new-kernel}-\eqref{new-kernel2}.
	Moreover, $\mathcal{N}_s u = 0$ in $\RR^N\setminus \overline{\Omega}$.
\end{prop}

\begin{proof}
	Given any test function $\overline{v}\in H_K(\Omega)$ we define $v:\RR^N \to \RR$ in the following way
	$$ v(x) = \begin{cases}
	\overline{v}(x) \ \ \text{ if } \ x\in \Omega,\\
	\left(\int_\Omega \frac{\overline{v}(z)}{|x-z|^{N+2s}}\,dz\right) \left(\int_\Omega |x-z|^{-N-2s}\,dz\right)^{-1} \ \ \text{ if } \ x\in \RR^N\setminus \overline{\Omega}.
	\end{cases} $$
	Indeed, this is the extension of $v$ outside $\Omega$ that ensures $\mathcal{N}_s v = 0$ in $\Omega^c$. Then, applying Lemma~\ref{Lemma: FromFracLaplToRestWithKErnel}, we obtain
	\begin{align*}
	\int \int_{\Omega\times \Omega} &\left\{u(x)-u(y)\right\}\left\{\overline{v}(x)-\overline{v}(y)\right\} \, K_\Omega(x,y) dx\, dy \\
	&=\int \int_{\Omega\times \Omega} \left\{u(x)-u(y)\right\}\left\{v(x)-v(y)\right\} \, K_\Omega (x,y) dx\, dy\\
	&= c_{N,s}\int \int_{(\RR^N\times\RR^N)\setminus (\Omega^c\times \Omega^c)} \frac{\left\{u(x)-u(y) \right\}\left\{v(x)-v(y) \right\}}{|x-y|^{N+2s}} dx\, dy.
	\end{align*}
	Moreover, by using $v$ as a test function in \eqref{Eq: WeakFormFracLapl} we have
	\begin{align*}
	\frac{c_{N,s}}{2} \int \int_{(\RR^N\times\RR^N)\setminus (\Omega^c\times \Omega^c)} &\frac{\left\{u(x)-u(y) \right\}\left\{v(x)-v(y) \right\}}{|x-y|^{N+2s}} dx\, dy \\
	&= \int_\Omega f(x)v(x)\,dx = \int_\Omega f(x)\overline{v}(x)\,dx.
	\end{align*}
	Thus, \eqref{Eq: WeakFormRestrWithKernel} follows by putting together the previous identities. Notice that applying Lemma~\ref{Lemma: FromFracLaplToRestWithKErnel} with $w=v$, we conclude that $v\in H^s_\Omega$. Thus, we can use it as a test function in \eqref{Eq: WeakFormFracLapl}.
	
	Now, taking any $\varphi \in C_0^\infty(\RR^N\setminus \overline{\Omega})\subset H^s_\Omega$ and using it as a test function in \eqref{Eq: WeakFormFracLapl}, we deduce
	$$ \int_{\Omega^c} \varphi(y) \mathcal{N}_s u(y) \, dy = \int_{\Omega^c} \varphi(y) \left( \frac{u(x)-u(y)}{|x-y|^{N+2s}} dx \right) dy = 0, $$
	and so we get that $\mathcal{N}_s u = 0$ in $\RR\setminus \overline{\Omega}$. Furthermore, we can apply Lemma~\ref{Lemma: FromFracLaplToRestWithKErnel} with $v=w=u$ and, since $u\in H^s_\Omega$, we conclude that $u\in H_K(\Omega)$.
\end{proof}

\end{document}